\documentclass[reqno,11pt]{amsart}
\usepackage[margin=1.1in]{geometry}
\usepackage{amsthm,amsmath,amssymb,dsfont}
\usepackage{mathrsfs,amsfonts,functan,extarrows}

\usepackage{hyperref}
\hypersetup{hidelinks}

 \usepackage{mathtools}

\usepackage{xcolor}
\usepackage{cite}
\usepackage{indentfirst, latexsym, amssymb, enumerate,amsmath,graphicx}
\usepackage{float}
\usepackage{relsize}
\usepackage{marginnote}
\usepackage{stmaryrd}
\usepackage{esint}
\usepackage{graphicx}
\usepackage{bm}
\usepackage{caption}
\usepackage{subfigure}
\usepackage{cite}
\usepackage{graphicx}
\usepackage{subfigure}
\usepackage{float}
\usepackage{paralist}
\usepackage{indentfirst}
\usepackage{cite}
\allowdisplaybreaks 
\theoremstyle{plain}
\newtheorem{thm}{Theorem}[section]

\newtheorem{lem}[thm]{Lemma}
\newtheorem{prop}[thm]{Proposition}

\theoremstyle{remark}
\newtheorem{rem}{Remark}[section]

\numberwithin{equation}{section}

\usepackage{appendix}
\usepackage{xcolor}

\newcommand{\eps}{{\varepsilon}}

\begin{document}

\title[Navier-Stokes-Vlasov-Fokker-Planck equations]
{The incompressible inhomogeneous Navier-Stokes-Vlasov-Fokker-Planck equations: global well-posedness and  inviscid limit}

\author[F.  Li]{Fucai Li}
\address{School of Mathematics, Nanjing University, Nanjing
 210093, P. R. China}
\email{fli@nju.edu.cn}

\author[J.  Ni]{Jinkai Ni} 
\address{School  of Mathematics, Nanjing University, Nanjing
 210093, P. R. China}
\email{jni@smail.nju.edu.cn}

\author[L.-Y. Shou]{Ling-Yun Shou} 
\address{School of Mathematics and Ministry of Education Key Laboratory of NSLSCS, Nanjing Normal University, and Key Laboratory of Jiangsu Provincial Universities of FDMTA, Nanjing
 210023, P. R. China} 
\email{shoulingyun11@gmail.com}

\author[D. Wang]{Dehua Wang} 
\address{Department of Mathematics, University of Pittsburgh, Pittsburgh, PA 15260, USA}
\email{dhwang@pitt.edu} 

\begin{abstract}

The global well-posedness and inviscid limit are investigated for the fluid-particle interaction system, described by the Navier-Stokes equations for the inhomogeneous incompressible viscous flows coupled with the Vlasov-Fokker-Planck equation for particles through a density-dependent nonlinear friction force in three-dimensional space. It is challenging to establish the inviscid limit over large time periods for the incompressible Euler equations under the influence of the weak dissipative mechanism generated by the friction force. We first prove the global stability of the equilibrium, in the sense that initial perturbations with appropriate Besov spatial regularity lead to global well-posedness and uniform regularity estimates with respect to the viscosity coefficient for strong solutions of the inhomogeneous Navier-Stokes-Vlasov-Fokker-Planck equations. In particular, we establish the optimal rates of convergence to equilibrium uniformly in Navier-Stokes. Then, we construct global solutions to the inhomogeneous Euler-Fokker-Planck equations via the vanishing viscosity limit. Furthermore, by capturing the dissipation arising from two-phase interactions, we rigorously justify the global-in-time strong convergence of the inviscid limit process, with a convergence rate that is in sharp contrast to that in the pure incompressible fluid case. To achieve this global convergence, novel ideas and new techniques are developed in the analysis and may be applied to other significant problems.

%



\end{abstract}


 
 \date{\today}

\maketitle

\setcounter{equation}{0}
 \indent \allowdisplaybreaks

\section{Introduction}
Fluid-particle models are important 
in various fields, such as combustion theory \cite{Wfa-1985,Wfa-1958}, medical biosprays \cite{BBJM-esaim-2005}, spray dynamics \cite{BBBDLLT-irma-2005}, and diesel engine operation \cite{RM-1952,RM-1952-a}. In this paper, we are concerned with the following fluid-particle system:
\begin{equation}\label{I-1}
\left\{\begin{aligned}
&\partial_{t} \rho+{\rm div}\,(\rho u)=0,  \\
&\partial_{t}(\rho u)+\mathrm{div}(\rho u\otimes u)+\nabla P -\mu\Delta u=-\int_{\mathbb{R}^{3}}\kappa(\rho) (u-v)F {\rm d}v,\\
&{\rm div}\, u=0,\\
&\partial_{t}F+v\cdot\nabla F=\mathrm{div}_{v}\big(-\kappa(\rho)(u-v)F+\lambda(\rho)\nabla_{v}F\big),\\
 \end{aligned}
 \right.
\end{equation}
which governs the evolution of the distribution function $F=F(t,x,v)\geq 0$ for the particles that are
located at the position $x\in \mathbb{R}^3$ with velocity $v\in\mathbb{R}^3$ and the time $t>0$, coupled with the density $\rho=\rho(t,x)\geq 0 $, the velocity $u=u(t,x)\in \mathbb{R}^3$ and the pressure $P=P(t,x)\geq 0$ for the fluid. Here,  {$\mu>0$ is the viscosity coefficient,} $\kappa(\rho)>0$ is the friction coefficient, and $\lambda(\rho)>0$ is the coefficient related to the collision between particles. The system \eqref{I-1} (abbreviated as the inhomogeneous NS-VFP system) describes the Vlasov-Fokker-Planck equation \eqref{I-1}$_4$ coupled with the inhomogeneous incompressible Navier-Stokes equations \eqref{I-1}$_1$-\eqref{I-1}$_3$ through the nonlinear friction force $\kappa(\rho)(u-v)$. In many physical settings, the coefficients $\kappa(\rho)$ and $\lambda(\rho)$ depend on the density $\rho$ (see \cite{
BD-JHDE-2006,Rouke1,CGL-JCP-2008}), and 
for simplicity we take
\begin{align}
\kappa(\rho)=\lambda(\rho)=\rho.\label{1.2}
\end{align}
We are interested in the global stability of solutions of the Cauchy problem for \eqref{I-1} supplemented with the initial data:
\begin{equation}\label{I-20}
(\rho,u,F)|_{t=0}=( \rho_{0}(x),u_{0}(x),F_{0}(x,v)),\quad    (x, v)\in \mathbb{R}^3\times \mathbb{R}^3,
\end{equation}
around the equilibrium state:
\begin{align}
(\bar{\rho},\bar{u},M_{[\bar{n},\bar{u}]}),\label{1.6}
\end{align}
where $\bar{\rho}, \bar{n}>0$ are constants and $\bar{u}\in \mathbb{R}^3$ is a  constant vector, and the Maxwellian $M_{[n,u]}$ is defined by
$$
M_{[n,u]}=\frac{n}{(2\pi)^{\frac{3}{2}}}e^{-\frac{|v-u|^2}{2}}.
$$

When the viscosity term $\mu\Delta u$ is dropped off, the system \eqref{I-1}--\eqref{1.2} becomes the inhomogeneous incompressible Euler-Vlasov-Fokker-Planck equations (abbreviated as the inhomogeneous Euler-VFP system):
\begin{equation}\label{I-2}
\left\{\begin{aligned}
&\partial_{t} \rho+{\rm div}\,(\rho u)=0,  \\
&\partial_{t}(\rho u)+\mathrm{div}(\rho u\otimes u)+\nabla P= -\int_{\mathbb{R}^{3}}\rho (u-v)F {\rm d}v,\\
&{\rm div}\, u=0,\\
&\partial_{t}F+v\cdot\nabla F=\mathrm{div}_{v}\big(-\rho(u-v)F+\rho\nabla_{v}F\big).
 \end{aligned}
 \right.
\end{equation}
Formally, we can also obtain the inviscid system \eqref{I-2} by taking the limit $\mu\rightarrow 0$ in  the inhomogeneous NS-VFP system \eqref{I-1}--\eqref{1.2}.

The mathematical study of fluid-particle interaction  models has  attracted significant  attention since the works of  Anoshchenko and Boutet de Monvel-Berthier \cite{AB-mmas-1998} and Hamdache \cite{Hk-JJIAM-1998}.  
Goudon, Jabin, and Vasseur \cite{GJV-IUMJ-2004,GJV-2004-2-IUMJ}  investigated the hydrodynamic limit problems for the Navier-Stokes-Vlasov system in fine and light particle regimes, respectively.  Boudin,  Desvillettes, Grandmont, and  Moussa \cite{BDGM-DIE-2009} constructed global weak solutions for the Navier-Stokes-Vlasov equations in spatial periodic domains (torus) (see also the case of bounded domains \cite{Yc-JMPA-2013} or time-dependent domains \cite{BGM-JDE-2017}). Chae, Kang and Lee \cite{CKL-JDE-2011} established the global existence of weak solutions to the incompressible Navier-Stokes equations coupled with the Vlasov-Fokker-Planck equations in  $\mathbb{R}^n(n=2$ or $3)$ and studied the uniqueness and regularity issues in $\mathbb{R}^2$. Wang and Yu \cite{WY-JDE-2015} established the existence of global weak solutions to the inhomogeneous Navier-Stokes-Vlasov  system   with the density-dependent friction coefficient $\kappa(\rho)=\rho$ in a bounded domain.  

Note that in the periodic domain $\mathbb{T}^3:=[-\pi,\pi]^3$ case, the incompressible fluid-particle coupled systems with the Fokker-Planck equation admit a regular equilibrium $(u=0, F=\hat M)$  {with $\hat M=(2\pi)^{-\frac{3}{2}}|\mathbb{T}|^{-3}e^{-\frac{|v|^2}{2}}$.}  Goudon, He, Moussa, and Zhang \cite{GHMZ-2010-SIMA} demonstrated the existence and exponential
 stability of global strong solutions to incompressible NS-VFP equations close to the equilibrium in the periodic domain $\mathbb{T}^3$.  Meanwhile, Carrillo, Duan, and Moussa \cite{CDM-krm-2011} observed the dissipation structure from the Fokker-Planck term and the friction force and  extended this investigation to the inviscid case in $\mathbb{R}^3$, establishing the existence and algebraic
decay of global strong solutions for the incompressible Euler-VFP equations. Recently, 
Li, Ni and Wu \cite{LNW-24-1} obtained the global existence and large time behavior of classical solutions to the incompressible inhomogeneous fluid-particle model with heat conductivity  and  a linear friction force $F_d=u-v$ in $\mathbb{R}^3$.
For the initial boundary value problem of the incompressible NS-VFP equations, interested readers can refer to 
 \cite{LLY-JSP-2022}.

When the diffusion $\lambda$ for the particles is absent, the Fokker-Planck equation \eqref{I-1} reduces to the Vlasov equation, and the large-time behavior in this case is quite different: one expects a monokinetic behavior in velocity for the distribution function (concentration to a Dirac mass). The first conditional large-time behavior result is demonstrated by  Choi and Kwon \cite{CK-Nonlinearity-2015} in periodic domains. Glass, Han-Kwan and Moussa \cite{GHM-ARMA-2018} studied the stability of non-trivial equilibria in a 2D pipe with partially absorbing boundary conditions. Han-Kwan, Moussa and Moyano \cite{Han-Kwan:2020} justified the global boundedness conditions proposed in \cite{CK-Nonlinearity-2015} and proved the convergence of the distribution function to a Dirac mass in velocity with exponential rate in $\mathbb{T}^3\times\mathbb{R}^3$. The method used in \cite{Han-Kwan:2020} has been applied to the large-time asymptotics in different domains \cite{EDM-Nonlinearity-2021,Han-Kwanwhole,SWYZ-JDE-2024,LSZ-KRM-2025,Danchin-2024-kato} and the hydrodynamic limit issues \cite{HM-MAMS-2024}. Among them, the authors in \cite{LSZ-KRM-2025} studied the exponential stability of the three-dimensional inhomogeneous Navier-Stokes-Vlasov system with $\kappa(\rho)=1$ in the presence of vacuum in $\mathbb{R}^3\times\mathbb{R}^3$ and used an energy argument based on Lyapunov inequalities, which was developed further by Danchin \cite{Danchin-2024-kato} to study the stability of Fujita-Kato type solutions.

We also mention recent progress on 
 compressible fluid-particle systems. In the case that the friction coefficient $\kappa(\rho)$ is constant,
 Mellet and Vasseur \cite{MV-MMMAS-2007, MV-CMP-2008} proved the existence of global weak solutions with finite energy and justified the large-friction limit for the compressible NS-VFP system to a one-velocity two-phase flow model. Chae, Kang, and Lee \cite{CKL-JHDE-2013} established the global existence and exponential decay of solutions for compressible flows in $\mathbb{T}^3\times\mathbb{R}^3$. Considering the compressible Euler-VFP system, Carrillo and Goudon \cite{CG-CPDE-2006} carried out asymptotic analysis and hydrodynamical expansions in both flowing and bubbling regimes. Duan and Liu \cite{DL-krm-2013} studied the global well-posedness of solutions to the compressible Euler-VFP equations for small initial perturbations around the equilibrium $(1,0,M)$.  
 Subsequently,
  Li,  Wang and Wang \cite{LWW-ARMA-2022} investigated the asymptotic stability of the planar rarefaction wave and justified the vanishing viscosity limit from the compressible NS-VFP system to the compressible Euler-VFP system in a finite time interval at the rate of $\mathcal{O}(\mu^{\frac{1}{2}})$. When the density-friction coefficient $\kappa(\rho)$ depends on the density, the main challenge arises from additional nonlinear terms involving $\rho$,  which have been addressed by Li, Mu and Wang \cite{LMW-SIAM-2017} by capturing the dissipation for $\rho$ caused by the pressure in the perturbative framework (see also the recent improvement by Li, Ni and Wu \cite{LNW-2025-preprint}). Baranger and Desvillettes \cite{BD-JHDE-2006} established the local existence of classical solutions to the compressible Vlasov-Euler equations in $\mathbb{R}^N\times \mathbb{R}^N$ $(N\geq 1)$.  Note that in the one-dimensional case, the global well-posedness and convergence to equilibrium hold for arbitrarily large data (see \cite{LS-JDE-2021,LS-CMR-2023,CJ-SIAM-2021}). 
 

So far, there have been no results on the global stability of the incompressible Euler-VFP system with variable density, i.e.\,\eqref{I-2}. Compared with the homogeneous incompressible NS-VFP system \cite{CDM-krm-2011,GHMZ-2010-SIMA,LLY-JSP-2022} or the compressible setting \cite{CKL-JHDE-2013,DL-krm-2013,LMW-SIAM-2017}, the new challenge lies in the lack of dissipation for the variable fluid density $\rho$.  
Furthermore, the inviscid limit problem for the three-dimensional pure incompressible Euler equations has been well studied in short times {\rm\cite{Swann1,Const:invisid,Const:invisid1,Kato:inviscid,mas:inviscid,MB-Book-2002}}. 
A natural question thereby emerges: is it possible to establish the inviscid limit over large time periods for the incompressible Euler equations under the influence of the weak dissipative mechanism generated by the friction force? This question serves as the main motivation for our investigation in this paper.
 

More precisely, the purpose of this paper is to establish the {\emph{global-in-time}} validity of the viscous approximation \eqref{I-1} for the inhomogeneous Euler-VFP system \eqref{I-2} via the vanishing viscosity limit. To that end, we first prove that for a fixed $\mu\in(0,1)$, the Cauchy problem for the inhomogeneous NS-VFP system \eqref{I-1} admits a unique global strong solution satisfying uniform regularity estimates with respect to $\mu$, provided that the initial data is uniformly close to the equilibrium state \eqref{1.6}. These uniform estimates allow us to rigorously justify the convergence toward the inhomogeneous Euler-VFP system \eqref{I-2} in the vanishing viscosity limit. Furthermore, we establish the sharp convergence rate of $\mathcal{O}(\mu)$ in the limit process, which achieves a  better convergence  rate
than those found in \cite{LWW-ARMA-2022}, where the convergence rate is $\mathcal{O}(\mu^{\frac12})$ .

\subsection{Main results}

 Without loss of generality, we normalize $\bar{\rho},\bar{u}$ and $M_{[\bar{n},\bar{u}]}$ as the following:
\begin{align*}
\bar{\rho}=1, \quad \bar{u}=0, \quad  M=M_{[1,0]}=\frac{1}{(2\pi)^{\frac{3}{2}}}e^{-\frac{|v|^2}{2}}.
\end{align*}
For the Euler-VFP system \eqref{I-2},  
applying the transformations 
$$
\rho = 1 + \varrho, \quad F = M + \sqrt{M} f,
$$
we can reformulate it  as
\begin{equation}\left\{
\begin{aligned}\label{rEVFP}
&\partial_{t}\varrho+u\cdot \nabla\varrho=0,\\ 
&\partial_{t}u+u\cdot\nabla u+\nabla P+u-b=\frac{\varrho}{1+\varrho}\nabla P{-au},\\
&{\rm div}\, u=0,\\
&\partial_{t}f+v\cdot\nabla_x f-u\cdot v\sqrt{M}-\mathcal{L}f\\
&\quad =\,{\frac{1}{2}}u\cdot vf-u\cdot\nabla_{v}f+\varrho\Big(\mathcal{L}f-u\cdot\nabla_v f+\frac{1}{2}u\cdot vf+u\cdot v\sqrt{M}\Big), 
\end{aligned}\right.
\end{equation}
with the initial data
\begin{align}\label{I-4}
(\varrho,u,f)|_{t=0}=\,& (\varrho_0(x),u_0(x),f_0(x,v))\nonumber\\
:=\,& \bigg( \rho_0(x)-1,u_0(x),\frac{F_0(x,v)-M}{\sqrt{M}}\bigg),\quad (x,v)\in \mathbb{R}^3\times \mathbb{R}^3.
\end{align}
In \eqref{rEVFP}, the linearized Fokker-Planck operator $\mathcal{L}$ is defined as 
\begin{align*}
\mathcal{L}f:=\frac{1}{\sqrt{M}}{\rm div}_v\bigg[M\nabla_v\Big(\frac{f}{\sqrt{M}}    \Big)    \bigg].
\end{align*}
Furthermore, $a=a^f$ and $b=b^f$ represent the moments of $f$, which are given by
\begin{align}
a^f(t,x):=\int_{\mathbb{R}^3}\sqrt{M}f(t,x,v)\mathrm{d}v,\quad  b^f(t,x):=\int_{\mathbb{R}^3}v\sqrt{M}f(t,x,v)\mathrm{d}v.  \label{ab}
\end{align}

For the solution to the NS-VFP system \eqref{I-1} with initial data \eqref{I-20}, to emphasize the dependence of $\mu$, we denote it by
 $(\rho^\mu,u^\mu,F^\mu).$ Similar to the Euler-VFP case, introducing the perturbations 
$$
\rho^\mu = 1 + \varrho^\mu,\quad F^\mu = M + \sqrt{M} f^\mu,
$$
 we can reformulate the NS-VFP system \eqref{I-1} by
\begin{equation}\label{rNSVFP}
\left\{
\begin{aligned}
&\partial_{t}\varrho^\mu+u^\mu\cdot \nabla\varrho^\mu=0,\\ 
&\partial_{t}u^\mu+u^\mu\cdot\nabla u^\mu+\nabla P^\mu+u^\mu-b^{f^\mu}=\frac{\mu \varrho^\mu }{1+\varrho^\mu}\Delta u^\mu+\frac{\varrho^\mu}{1+\varrho^\mu}\nabla P^\mu{-a^{f^\mu}u^\mu},\\
&{\rm div}\, u^\mu=0,\\
&\partial_{t}f^\mu+v\cdot\nabla_x f^\mu-u^\mu\cdot v\sqrt{M}-\mathcal{L}f^\mu\\
&\quad ={\frac{1}{2}}u^\mu\cdot vf^\mu-u^\mu\cdot\nabla_{v}f^\mu+\varrho^\mu\Big(\mathcal{L}f^\mu-u^\mu\cdot\nabla_v f^\mu+\frac{1}{2}u^\mu\cdot vf^\mu+u^\mu\cdot v\sqrt{M}\Big), 
\end{aligned}\right.
\end{equation}
with the initial data
\begin{align}\label{NJKI-10}
(\varrho^{\mu},u^{\mu},f^{\mu})|_{t=0}=\,& (\varrho^\mu_0(x),u^\mu_0(x),f^\mu_0(x,v))\nonumber\\
:=\,& \bigg( \rho^\mu_0(x)-1,u^\mu_0(x),\frac{F^\mu_0(x,v)-M}{\sqrt{M}}\bigg),\quad (x,v)\in \mathbb{R}^3\times \mathbb{R}^3. 
\end{align}

Our first result concerns the global existence and large-time behavior of solutions to the NS-VFP system \eqref{rNSVFP}. In particular, we establish the regularity estimates uniform with respect to both $\mu$ and time.

\begin{thm}\label{T1.1} 
Let $0<\mu<1$ and $1<s\leq \frac{3}{2}$. Assume that  $(\varrho^\mu_0, u^\mu_0,f^\mu_0)$
satisfy $F^\mu_0=M+$ $\sqrt{M} f^\mu_0 \geq 0$, $(\varrho^\mu_0, u^\mu_0)\in H^3\cap\dot{B}^{-s}_{2,\infty}$ and $f^\mu_0\in L_v^2(H^3\cap\dot{B}^{-s}_{2,\infty})$. Then, there exists a constant $\eps_0>0$ independent of $\mu$ such that if
\begin{align}
\mathcal{E}_0^\mu:=\|(\varrho^\mu_0, u^\mu_0)\|_{H^3\cap\dot{B}^{-s}_{2,\infty}}^2+\|f^\mu_0\|_{ L_v^2(H^3\cap\dot{B}^{-s}_{2,\infty})}^2\leq \eps_0,\label{small1}
\end{align}
then the Cauchy problem of  the inhomogeneous NS-VFP system   \eqref{rNSVFP}--\eqref{NJKI-10}   admits a unique global strong solution 
  $(\varrho^\mu, u^\mu,P^\mu, f^\mu)$ satisfying $\rho^\mu=1+\varrho^\mu>0$, $F^\mu=M^\mu+\sqrt{M} f^\mu \geq 0$,
\begin{align} 
  &\sup_{t\geq0}\big(\|(\varrho^\mu, u^\mu)(t)\|_{H^3}^2+\|f^\mu(t)\|_{L_{v}^2(H^3)}^2\big)\nonumber\\
  &\quad+\int_0^t \mathcal{D}(u^\mu,P^\mu,f^\mu)(\tau) {\rm d}\tau+\mu\int_0^t \|\nabla u^\mu(\tau)\|_{H^3}^2 {\rm d}\tau \leq C_0\mathcal{E}_0^\mu,\label{uniform1}
\end{align}
and 
\begin{equation}\label{1.13}
\left\{
\begin{aligned}
&\| u^\mu(t)\|_{L^{2}}+\|f^\mu(t)\|_{L_v^2(L^{2})}+\|\nabla P^\mu(t)\|_{L^2}\leq\, C_0\eps_0^{\frac{1}{2}} (1+t)^{-\frac{s}{2}},\\
&\| \nabla u^\mu(t)\|_{L^{2}}+\|\nabla f^\mu(t)\|_{L_v^2(L^2)}+\|\nabla^2 P^\mu(t)\|_{L^2}\leq   C_0\eps_0^{\frac{1}{2}}  (1+t)^{-\frac{s}{2}-\frac{1}{2}} ,         \\
&\| \nabla^2 u^\mu(t)\|_{H^{1}}+\|\nabla^2 f^\mu(t)\|_{L_v^2(H^1)} \leq   C_0\eps_0^{\frac{1}{2}}   (1+t)^{-\frac{s}{2}-\frac{1}{2}},
\end{aligned}
\right.
\end{equation}
for all $t\geq 0$. 
Here $C_0>0$ is a constant independent of $\mu$, the time and initial data. 
\end{thm}

\begin{rem} In Theorem \ref{T1.1}, the dissipation functional $\mathcal{D}(u^\mu,P^\mu,f^\mu)(t)$ is defined as\begin{align}\label{D}
\mathcal{D}(u^\mu,P^\mu,f^\mu)(t):=\,& \|(b^{f^\mu}-u^\mu)(t)\|_{H^3}^2
+\|\{\mathbf{I}-\mathbf{P}\}f^\mu(t)\|_{L^2_{v,\nu}(H^3)}^2\nonumber\\
& +\|\nabla (a^{f^\mu},b^{f^\mu})(t)\|_{H^2}^2+\|\nabla P^\mu(t)\|_{H^2}^2,
\end{align}
where $\{\mathbf{I}-\mathbf{P}\}f^\mu$ is the microscopic part of $f^\mu$, see the definition 
in Section 2.2. 
 And the homogeneous Besov spaces $\dot{B}^s_{p,r}(\mathbb{R}^3)$ and  $\dot{B}^s_{p,\infty}(\mathbb{R}^3)$  
 as well as the Besov norms are defined in Section 2.1.
   
\end{rem}

By virtue of the uniform regularity estimates obtained in Theorem \ref{T1.1}, we have the global existence of the Euler-VFP system \eqref{I-2} via the inviscid limit.

\begin{thm}\label{T1.2}
Let  $(\varrho_0, u_0,f_0)$
satisfy $F_0=M+$ $\sqrt{M} f_0 \geq 0$, $(\varrho_0, u_0)\in H^3\cap\dot{B}^{-s}_{2,\infty}$ and $f_0\in L_v^2(H^3\cap\dot{B}^{-s}_{2,\infty})$ with $1<s\leq \frac{3}{2}$. There exists a constant $\eps_1>0$ such that if we suppose 
\begin{align}
\mathcal{E}_0:=\|(\varrho_0, u_0)\|_{H^3\cap\dot{B}^{-s}_{2,\infty}}^2+\|f_0\|_{ L_v^2(H^3\cap\dot{B}^{-s}_{2,\infty})}^2\leq \eps_1,\label{small2}
\end{align}
then the following statements hold{\rm:}

\begin{itemize}
\item If $\{(\varrho_0^\mu,u_0^\mu,f_0^\mu)\}_{0<\mu<1}$ is a sequence such that $(\varrho_0^\mu,u_0^\mu)\rightarrow (\varrho_0,u_0)$ in $H^3\cap\dot{B}^{-s}_{2,\infty}$ and $f_0^\mu\rightarrow f_0$ in $L^2_v(H^3\cap\dot{B}^{-s}_{2,\infty})$. Then  for all $0<\mu<\mu_0$ with a suitably small $\mu_0\in(0,1)$, there exists a unique global solution $(\varrho^\mu,u^\mu,f^\mu)$ to the Cauchy problem for the inhomogeneous NS-VFP system subject to the initial data  $(\varrho_0^\mu,u_0^\mu,f_0^\mu)$. The sequence $\{(\varrho^\mu,u^\mu,P^\mu,f^\mu)\}_{0<\mu<\mu_0}$ satisfies  the uniform estimates \eqref{uniform1} and \eqref{1.13}.

\item As $\mu\rightarrow 0$, there exists a limit $(\varrho,u,P,f)$ such that up to a subsequence, the following  limits hold 
\begin{equation}\label{convergence0}
\left\{
\begin{aligned}
(\varrho^\mu,u^\mu) &\rightharpoonup (\varrho,u)\quad  \,\, \text{weakly$^*$~ ~~in ~~} L^{\infty}(\mathbb{R}_{+};H^3),\\
\nabla P^\mu &\rightharpoonup \nabla P \quad \quad \text{weakly~ in ~~} L^{2}(\mathbb{R}_{+};H^2),\\
f^\mu &\rightharpoonup f\quad \quad \quad  \,  \text{weakly$^*$~ in ~~} L^{\infty}(\mathbb{R}_{+};L^2_v(H^3))\cap  L^2(\mathbb{R}_{+};L^2_{v,\nu}(H^3)),\\
(\varrho^\mu,u^\mu) & \rightarrow (\varrho,u) \quad  \,\,  \,\, \text{strongly~ in ~~} C_{\rm loc}(\mathbb{R}_{+};H^2_{\rm loc}).
\end{aligned}
\right.
\end{equation}

\item The limit $(\varrho,u,f)$ is the global solution to the Cauchy problem \eqref{rEVFP}--\eqref{I-4} of the inhomogeneous Euler-VFP system. Furthermore, $(\varrho,u,P,f)$ satisfies $\rho=1+\varrho>0$, $F=M+\sqrt{M} f \geq 0$,
\begin{align}  
  &\sup_{t\geq0}\big(\|(\varrho, u)(t)\|_{H^3}^2+\|f(t)\|_{L_{v}^2(H^3)}^2\big)+\int_0^t D(u,P,f)(\tau) {\rm d}\tau \leq C_1 \mathcal{E}_0,\label{uniform2}
\end{align}
and for all $t\geq0$,
\begin{equation}\label{G1.7}
\left\{
\begin{aligned}
&\| u (t)\|_{L^{2}}+\|f (t)\|_{L_v^2(L^{2})}+\|\nabla P (t)\|_{L^2}\leq\, C_1\eps_1^{\frac{1}{2}} (1+t)^{-\frac{s}{2}},\\
&\| \nabla u (t)\|_{L^{2}}+\|\nabla f (t)\|_{L_v^2(L^2)}+\|\nabla^2 P (t)\|_{L^2}\leq   C_1\eps_1^{\frac{1}{2}} (1+t)^{-\frac{s}{2}-\frac{1}{2}} ,         \\
&\| \nabla^2 u (t)\|_{H^{1}}+\|\nabla^2 f (t)\|_{L_v^2(H^1)} \leq   C_1\eps_1^{\frac{1}{2}} (1+t)^{-\frac{s}{2}-\frac{1}{2}},
\end{aligned}
\right.
\end{equation}
{ {where $C_1>0$ is a constant independent of the time and initial data.}} 
\end{itemize}
\end{thm}

\begin{rem}
{ { Note that the strong convergence properties stated in Theorem \ref{T1.2} hold locally in time and space and up to a subsequence.
It is worth noting that $\varepsilon_1$ can be selected as $\varepsilon_1=\frac{\varepsilon_0}{2}$, where the constant $\varepsilon_0$ is given in \eqref{small1}. 
Moreover, the strong convergence in Theorem \ref{T1.2} holds only locally in time and space (up to a subsequence), which is guaranteed by the uniform bounds along with the Aubin-Lions compactness lemma. We  emphasize that this convergence will be sharpened in the subsequent theorem (see Theorem \ref{T1.3}).}}
\end{rem}

To justify the validity of the global-in-time inviscid limit from the NS-VFP model \eqref{rNSVFP} to the Euler-VFP model \eqref{rEVFP}, we establish the global error estimates between solutions for \eqref{rNSVFP} and  \eqref{rEVFP}. 

\begin{thm}[Global inviscid limit]\label{T1.3}  
Let $(\varrho^\mu,u^\mu,P^\mu,f^\mu)$ be the global solutions to the Cauchy problem for the inhomogeneous NS-VFP system \eqref{rNSVFP} with the initial data $(f_0^\mu,\varrho_0^\mu,u_0^\mu)$ given by Theorem \ref{T1.1}, and $(\varrho ,u ,P,f )$ be the corresponding
 global solution to the Cauchy problem for the inhomogeneous Euler-VFP system \eqref{rEVFP} with the initial data $(f_0 ,\varrho_0 ,u_0 )$ obtained in Theorem \ref{T1.2}. Assume further that
 \begin{align}
 \|( \varrho_0^\mu-\varrho_0 ,u_0^\mu-u_0 )\|_{H^1}+\|f_0^\mu-f_0\|_{L^2_v(H^1)}\leq \mu.\label{a3}
 \end{align}
 Then there exists a uniform constant $C_2>0$ such that
\begin{align}
  &\sup _{t \geq 0}\big((1+t)^{-1}\|(\varrho^\mu-\varrho)(t)\|_{H^1}^2+ \|( u^\mu-u )(t)\|_{H^1}^2+\|(f^\mu-f )(t)\|_{L_{v}^2(H^1)}^2\big)\nonumber\\
  &\quad+\int_0^t\big(\|\nabla(u^\mu-u)(\tau)\|_{H^1}^2+\|\nabla(P^\mu-P)(\tau)\|_{L^2}^2\big)\,{\rm d}\tau\nonumber\\
  &\quad\quad+ \int_0^t\big(\|\nabla (f^\mu-f )(\tau)\|_{L^2_{x,v}}^2+\|\{\mathbf{I}-\mathbf{P}\}(f^\mu-f )(\tau)\|_{L^2_{v,\nu}(H^1)}^2
\big){\rm d}\tau 
 \leq C_2\mu^2.\label{rate}
\end{align}
\end{thm}

\begin{rem}
To the best of our knowledge, Theorem \ref{T1.3} provides the first result concerning the global-in-time  convergence of the inviscid limit for the incompressible Euler equations with weak dissipative effects {\rm(}i.e.,
the dissipation mechanisms of $b^f-u$ from the friction force{\rm)}.  The coupled Euler-VFP system \eqref{I-1} is derived globally-in-time from the NS-VFP system \eqref{I-2} with the  convergence rate of  $\mathcal{O}(\mu)$. 

For the pure incompressible Euler equations in three dimensions, the inviscid limit   converge only locally in time due to possible finite-time blow-up. Regarding the inviscid limit problem of the pure incompressible Euler equations, we refer to the important works {\rm\cite{Swann1,Const:invisid,Const:invisid1,Kato:inviscid,mas:inviscid,MB-Book-2002}} and references therein.
\end{rem}

\begin{rem}
In the proof of our results, one of  the main difficulties arises in the nonlinear friction force $\varrho(u-v)$, which causes some additional nonlinear terms in the Fokker-Planck equations  \eqref{rEVFP}$_3$ and \eqref{rNSVFP}$_3$, for example, $\varrho \mathcal{L}f$, where $\varrho$ is degenerate, in the sense that it does not possess any dissipation due to its transport nature. To overcome this difficulty, we establish the uniform evolution of lower-order energy on the  spatial $\dot{B}^{-s}_{2,\infty}$ {\rm(}$1<s\leq \frac{3}{2}${\rm)} regularity for decay estimates, which ensures the $L^1$ time integrability of $\|\nabla u\|_{H^2}$.
\end{rem}

\begin{rem}
For $\mu \geq 1$, the well-posedness and large-time behavior stated in Theorem \ref{T1.1} still hold. Here we assume that $\mu<1$ to establish uniform estimates with respect to $\mu$ in this regime, 
in order to take the inviscid limit $\mu \rightarrow 0$. 
\end{rem}

\begin{rem}
Theorems \ref{T1.1}--\ref{T1.3} hold true for more general coefficients $\kappa(\rho)$ and $\lambda(\rho)$. For example, one may consider
$$
\kappa(\rho)=\lambda(\rho)=\rho^\alpha,\quad \alpha\geq 0.
$$
\end{rem}

\begin{rem}
The $L^2$ and $\dot{H}^1$ decay rates obtained in Theorems \ref{T1.1}--\ref{T1.2} are optimal since they coincide with the rates of the heat equation, while the $\dot{B}^{-s}_{2,\infty}$ regularity assumption seems to be sharp for these optimal decay rates {\rm(}see {\rm{\cite[Page 610] {brandolesehand}}}{\rm)}. If the $\dot{B}^{-s}_{2,\infty}$ regularity is replaced by the  $L^q$ regularity for some $q\in[1,\frac{6}{5})$, then the following classical $L^q$--$L^2$ type of
the optimal decay estimates hold{\rm:} 
\begin{align*} 
\|\nabla^l u(t)\|_{H^{2-l}}+\|\nabla^lf(t)\|_{L_v^2(H^{2-l})}\leq\,& C_0(1+t)^{-\frac{l+s_q}{2}},    \\ 
\|\nabla u(t)\|_{L^{p}}+\|\nabla f(t)\|_{L_v^2(L^{p})}\leq\,& C_0(1+t)^{-\frac{3}{2}(\frac{1}{2}-\frac{1}{p})-\frac{s_q}{2}},   
\end{align*}
for any $t\geq 0$, $2\leq p\leq 6$ and $l=0,1$.
Here, $s_q:=3(\frac{1}{q}-\frac{1}{2})$. It should be noted that such a range $q\in [1,\frac{6}{5})$ coincides with the significant work by Duan, Liu, Ukai, and  Yang {\rm{\cite{DLUY-2007-JDE}}}. 
\end{rem}

\subsection{Difficulties, strategies and novelty}

We now explain the challenges and main ideas used in the proofs of  the above theorems. The results
in Theorem \ref{T1.1} are based on {\emph{a priori}} estimates that are uniform with respect to both $\mu$ and time (see Proposition \ref{P3.1}). Note that the dissipation effects of the relative quantity $b^{f^\mu}-u^\mu$ and the microscopic part $\{\mathbf{I-P}\}f^\mu$ can be derived from the effects of the friction term and the Fokker-Planck operator $-\mathcal{L}$. On the other hand, using the ideas developed in \cite{Gy-IUMJ-2004,DF-jmp-2010,CDM-krm-2011}, one can capture the degenerate dissipation for $a^{f^\mu}$ and $b^{f^\mu}$ due to the interactions of $v\cdot\nabla_x f^\mu$ and $\mathcal{L} f^\mu$. Consequently, the dissipation for $u^\mu$ and $f^\mu$ can be derived as follows: 
\begin{align}\label{udis}
\|\nabla u^\mu(t)\|_{H^2}^2+\|\nabla f^\mu(t)\|_{L^2_v(H^2)}^2\lesssim \mathcal{D}(u^\mu,P^\mu,f^\mu)(t).
\end{align}
Note that the above inequality is independent of the viscosity $\mu$, which plays a key role in controlling nonlinear terms, e.g., $u^\mu\cdot \nabla u^\mu$, uniformly with respect to both the time and $\mu$. 

However, new observations and ideas are needed to handle the nonlinear terms involving the density perturbation $\varrho^\mu$. Indeed, different from the scenarios in compressible fluid \cite{DL-krm-2013,LMW-SIAM-2017}, the variable $\varrho^\mu$ does not exhibit any dissipation due to its transport nature and the divergence-free condition ${\rm div}\, u^\mu=0$. Consequently, it is only possible to obtain a $L^{\infty}$ time bound for $\varrho^\mu$, while $u^\mu$ and $f^\mu$ must take on all time integrability in the nonlinear analysis.  In view of the transport theorem, establishing the uniform bound for $\varrho^\mu$ relies on the $L^1$ time integrability of $\|\nabla u^\mu\|_{H^2}$:
\begin{align}\label{rhoH3}
    \|\varrho^\mu(t)\|_{H^3}\leq e^{\int_0^t \|\nabla u^\mu\|_{H^2}{\rm d}\tau} \|\varrho_0^\mu\|_{H^3}.
\end{align}
However, the dissipation generated from \eqref{udis} only provides an $L^2$ time integrability, which causes a time growth in the nonlinear analysis. To overcome this challenge, we aim to obtain sufficiently fast decay rates for $u^\mu$ under suitable initial conditions, ensuring the uniform-in-time control of $\|\nabla u^\mu\|_{L^1(0,t;H^2)}$. Due to the fact that there are higher velocity derivatives and weights in the nonlinear terms from the density-dependent friction force, e.g., \(\varrho\mathcal{L} f\), the classical method based on the use of pointwise estimates and Duhamel's principle as in \cite{CDM-krm-2011,DL-krm-2013} may not be applicable here.

We develop a pure energy argument to address the above loss of decay. Instead of relying on the standard $H^3$ energy, a key ingredient in our analysis is to construct a higher-order Lyapunov functional $\mathcal{E}_1(t)\sim \|\nabla u^\mu\|_{H^2}^2+\|\nabla f^\mu\|_{L^2_v(H^2)}^2$ such that
\begin{align}\label{G3.810}
 \frac{{\rm d}}{{\rm d}t}\mathcal{E}_1(t)+c\big(\|\nabla^2 u^\mu\|_{H^2}^2+\|\nabla^2 f^\mu\|_{L^2_v(H^2)}^2\big)\leq 0,
\end{align}
for some uniform constant $c>0$ (see \eqref{G3.81}). Such construction \eqref{G3.810} requires the careful control of every nonlinear term to be absorbed by the dissipation in \eqref{G3.810}. Then, inspired by \cite{GW-CPDE-2012}, we apply the interpolation control in \eqref{G3.810} to derive the time decay associated with the $\dot{B}^{-s}_{2,\infty}$ regularity:
\begin{align}\label{uniformdecay}
&\|\nabla u^\mu\|_{H^2}+\|\nabla f^\mu\|_{L^2_v(H^2)}\nonumber\\
&\quad\lesssim (1+t)^{-\frac{s}{2}-\frac{1}{2}}\Big(\|\nabla u_0^\mu\|_{H^2}+\|\nabla f_0^\mu\|_{L^2_v(H^2)}+\|u^\mu\|_{\dot{B}^{-s}_{2,\infty}}+\|f^\mu\|_{L^2_v(\dot{B}^{-s}_{2,\infty})}\Big),
\end{align}
which is crucial for obtaining the bound \eqref{rhoH3} uniformly in time when $\frac{s}{2}+\frac{1}{2}>1$, i.e., $s>1$. Note that the $\dot{B}^{-s}_{2,\infty}$--assumption was introduced by Sohinger and Strain \cite{SS-adv-2014} for the Boltzmann equation and by Xu and Kawashima \cite{Xu} for partially dissipative hyperbolic systems. A key difference of our work is the  $\dot{B}^{-s}_{2,\infty}$--norm of the solution, which ensures the decay in  \eqref{uniformdecay}, is incorporated into enclosing {\emph{a priori}} estimates to compensate for the lack of dissipation for $\varrho$. In particular, the decay relationship \eqref{uniformdecay} is used in establishing the uniform evolution of the $\dot{B}^{-s}_{2,\infty}$ regularity (see Proposition \ref{P3.9}). The above procedure leads to the uniform {\emph{a priori}} estimates, and along with a bootstrap argument, it completes the proof of Theorem \ref{T1.1}. Building up the global existence and uniform regularity estimates with respect to $\mu$ of solutions to the NS-VFP system \eqref{rNSVFP}, we can perform the compactness argument to obtain the global existence of the Euler-VFP system \eqref{rEVFP} as the viscosity coefficient $\mu\rightarrow0$.

Finally, we establish a uniform-in-time error estimate \eqref{rate} between the solutions to the inhomogeneous NS-VFP system \eqref{rNSVFP} and the inhomogeneous Euler-VFP system \eqref{rEVFP}, based on the existence of solutions obtained in Theorems \ref{T1.1} and \ref{T1.2}. This is accomplished by performing energy estimates on the difference between the two systems and analyzing each nonlinear term involving the error variables, utilizing the uniform regularity estimates \eqref{uniform1} and \eqref{uniform2} (see Lemmas \ref{L:error1}--\ref{L:error4}). It is important to note that the bound on $\varrho^\mu-\varrho$ has to grow over time due to the nonlinear term in the error equation for $\varrho^\mu-\varrho$. To overcome this, we employ additional time-weighted regularity estimates. In particular, the uniform error estimate \eqref{rate} leads to the global-in-time strong convergence of this limiting process, which contrasts sharply with 
the pure incompressible fluid case where only local convergence results are available.

\subsection{Outline of the paper}

The rest of this paper is organized as follows. 
In Section 2, we introduce some notations, state the macro-micro decomposition, the properties of the Fokker-Planck operator \(\mathcal{L}\), and the functional spaces utilized throughout the paper. 
Section 3 is devoted to the proof of Theorem \ref{T1.1} concerning the global existence and uniform regularity estimates of strong solutions for the Cauchy problem of the inhomogeneous NS-VFP system \eqref{rNSVFP}. We first derive the energy estimates for \(f\) and \(u\). 
Using these estimates, we then establish the decay rates associated with the negative Besov evolution for \(f^\mu\) and \(u^\mu\). Subsequently, using a refined energy method, we obtain the uniform evolution of negative Besov norms for \(f^\mu\) and \(u^\mu\). With these decay rates in place, we further derive the \(H^3\) and \(\dot{B}^{-s}_{2,\infty}\) estimates for \(\varrho^\mu\).  In Section 4, we investigate the vanishing viscosity limit of the inhomogeneous NS-VFP system \eqref{rNSVFP} toward the inhomogeneous Euler-VFP system \eqref{rEVFP}  and derive a global-in-time convergence rate, i.e., Theorems \ref{T1.2}-\ref{T1.3}. Finally, in the appendix, we collect several useful technical lemmas.

\medskip
\section{Preliminaries}
In this section, we present some notations, definitions, and fundamental facts that are frequently used throughout this paper.

\subsection{Notations}
The symbol $C$ denotes a generic positive constant  independent of time $t$,
 and it may vary from one instance to another.
And $a\lesssim b$ implies that $a\leq Cb$.  For simplicity,  we
set $\|(g, h)\|_X:=\|g\|_X +\|h\|_X$ for some Banach spaces $X$, where $g = g(x)$ and $h = h(x)$ belong to $X$.
We use $\langle\cdot,\cdot\rangle$ to represent the inner
product over the Hilbert space $L^2_v=L^2(\mathbb{R}^3_v)$, i.e.,
\begin{align*}
\langle g,h\rangle:=\int_{\mathbb{R}^3}g(v)h(v)\mathrm{d}v,\quad g,h\in L_v^2,    
\end{align*}
with its corresponding norm $\|\cdot\|_{L^2}$. Let the weight function be defined as  $\nu(v)= 1+|v|^2$, and denote $\|\cdot\|_{L^2_{v,\nu}}$ by
\begin{align*}
\|g\|_{L^2_{v,\nu}}^2:=\int_{\mathbb{R}^3}\big(|\nabla_v g(v)|^2+\nu(v)|g(v)|^2\big){\rm d}v,
\quad g=g(v).
\end{align*}
When involving the spatial variable, we write $L^2_{x,v}=L^2(\mathbb{R}^3_x\times\mathbb{R}^3_v)$ associated with its norm $\|\cdot\|_{L^2_{x,v}}$. The norm in the space $L^p=L^p(\mathbb{R}^3_x)$ (resp. $L^2_v(L^p)=L^2(\mathbb{R}^3_v;L^p(\mathbb{R}^3_x))$) is denoted by $\|\cdot\|_{L^p}$ (resp. $\|\cdot\|_{L^2_v(L^p_v)}$).

For the two multi-indices $\alpha=(\alpha_1,\alpha_2,\alpha_3)$ and $\beta=(\beta_1,\beta_2,\beta_3)$, we consistently use the notation 
\begin{align*}
\partial^{\alpha}_{\beta}:=\partial^{\alpha_1}_{x_1}\partial^{\alpha_2}_{x_2}\partial^{\alpha_3}_{x_3}\partial^{\beta_1}_{v_1}\partial^{\beta_2}_{v_2}\partial^{\beta_3}_{v_3},    
\end{align*}
where $x=(x_1,x_2,x_3)$ denotes the spatial variable  
and $v=(v_1,v_2,v_3)$ denotes the microscopic velocity of the particles.
For any $s\geq0$, the norm in the Sobolev space $H^s=H^s(\mathbb{R}^3_x)$ is endowed by the norm:
\begin{align*}
\|g\|_{H^s}:=\sum_{|\alpha|\leq s}\|\partial^{\alpha}g\|_{L^2}.  
\end{align*}
We also write the norm of $L^2_v(H^s)=L^2(\mathbb{R}^3_v;H^s(\mathbb{R}^3_x))$ by 
$$
\|g\|_{L^2_v(H^s)}:=\sum_{|\alpha|\leq s}\|\partial^{\alpha}g\|_{L^2_{x,v}}.
$$

Next, let's recall the definition of a homogeneous Besov space. Let $\varphi\in C_0^{\infty}(\mathbb{R}^3)$ be a smooth function such that $\varphi(\xi)=1$ for $|\xi|\leq 1$
and $\varphi(\xi)=0$ for $|\xi|\geq 2$. Let $\phi(\xi)=\varphi(\xi)-\varphi(2\xi)$ and $\phi_j(\xi)=\phi(2^{-j}\xi)$ for $j\in \mathbb{Z}$.
By construction, we have $\sum _{j\in\mathbb{Z}}\phi_j(\xi)=1$ for all
$\xi\ne 0$.
Define  $\dot{\Delta}_j f:=\mathcal{F}^{-1}(\phi_j)*f$, where $\mathcal{F}^{-1}$ denotes the inverse  Fourier transform. Then, we have the Littlewood-Paley decomposition 
$$
u=\sum_{j\in\mathbb{Z}}\dot{\Delta}_{j}u,
$$
which holds in $\mathcal{S}'(\mathbb{R}^3)$ modulo polynomials.
 For any $s\in\mathbb{R}$ and $1\leq p,r\leq \infty$, we define the
 homogeneous Besov space  $\dot{B}^s_{p,r}(\mathbb{R}^3)$
with norm $\|\cdot\|_{\dot{B}^s_{p,r}}$ as follows:
\begin{align*}
\|g\|_{\dot{B}^s_{p,r}}:=\bigg(\sum_{j\in\mathbb{Z}}2^{rsj}\|\dot{\Delta}_j g\|_{L^p}^r \bigg)^{\frac{1}{r}}, 
\end{align*}
and if $r=\infty$, then
\begin{align*}
\|g\|_{\dot{B}^s_{p,\infty}}:=\sup_{j\in\mathbb{Z}}2^{sj}\|\dot{\Delta}_j g\|_{L^p}.   
\end{align*}
For more details on Besov spaces, one can refer to \cite{BCD-Book-2011}.

We shall use the notation $L_v^2(\dot{B}^s_{p,r})$ to denote
the space $L^2(\mathbb{R}^3_v;\dot{B}^s_{p,r})$  equipped with the norm
\begin{align*}
\|g\|_{L_v^2(\dot{B}^s_{p,r})}=\Big(\int_{\mathbb{R}^3} \|g\|_{\dot{B}^s_{p,r}}^2\mathrm{d}v \Big)^{\frac{1}{2}}.
\end{align*}
In particular, we have $\dot{H}^s=\dot{B}^{s}_{2,2}$. Besides, we set the following weighted norm:
\begin{align*}
\|g\|_{L^2_{v,\nu}(\dot{B}^{-s}_{2,\infty})}:=\Big(\int_{\mathbb{R}^3}\Big(\|\nabla_v g\|_{\dot{B}^{-s}_{2,\infty}}^2+\nu(v)\|g\|_{\dot{B}^{-s}_{2,\infty}}^2\Big)\mathrm{d}v \Big)^{\frac{1}{2}}.  
\end{align*}

\subsection{Macro-micro decomposition and properties of \texorpdfstring{\(\mathcal{L}\)}{L}}

We introduce the macro-micro decomposition for the distribution function, which was introduced in \cite{Gy-IUMJ-2004} for the Boltzmann equation and developed by \cite{DF-jmp-2010} for the Fokker-Planck equation (see also  \cite{DL-krm-2013,CDM-krm-2011}).

The function $f(t,x,v)$ is decomposed into
its macroscopic (fluid) component $\mathbf{P}f$ and its microscopic component $\{\mathbf{I}-\mathbf{P}\}f$, such that 
\begin{align}\label{mmd}
f=\mathbf{P}f+\{\mathbf{I}-\mathbf{P}\}f . 
\end{align} 
Here, $\mathbf{P}$ represents the velocity orthogonal projection operator defined by
\begin{align*}
\mathbf{P}: L^2\rightarrow {\rm Span}\big\{\sqrt{M}, v_1\sqrt{M}, v_2\sqrt{M},v_3\sqrt{M}         \big\},    
\end{align*}
and can be further decomposed as
\begin{align*}
\mathbf{P}:= \mathbf{P}_0\oplus \mathbf{P}_1. 
\end{align*}
Define 
\begin{align*}
a_f:=\int_{\mathbb{R}^3}  f \sqrt{M}\,dv,\quad b_f:=\int_{\mathbb{R}^3}  f v\sqrt{M}\,dv.
\end{align*}
Then, we have $\mathbf{P}f=\mathbf{P}_0 f+\mathbf{P}_1 f$ with $\mathbf{P}_0 f:=a^f \sqrt{M}$ and  $\mathbf{P_1}f:=b^fv\sqrt{M}$. 

It is well-known that the Fokker-Planck operator $\mathcal{L}$ is self-adjoint and possesses the following coercive property 
\begin{gather*}
  \left\langle \mathcal{L} f, f\right\rangle=-\int_{\mathbb{R}^3}\left|\nabla_v\left(\frac{f}{\sqrt{M}}\right)\right|^2 M \mathrm{d} v.
  \end{gather*}
 Define
\begin{gather*}
 \operatorname{Ker} \mathcal{L}:=\operatorname{Span}\big\{\sqrt{M}\big\}, \quad 
 \operatorname { Range } \mathcal{L}:=\operatorname{Span}\big\{\sqrt{M}\big\}^{\perp} .
\end{gather*}
Then there is a constant $\lambda_0>0$ such that
\begin{align}
-\left\langle \mathcal{L} f, f\right\rangle \geq \lambda_0 \|\{\mathbf{I}-\mathbf{P}\} f \|_{L^2_{\nu,v}}^2+|b^f|^2 .\label{H2.1}
\end{align}

\medskip
\section{Proof of Theorem \ref{T1.1}}

Throughout this section, we simplify the notations of   solutions to the Cauchy 
problem \eqref{rNSVFP}--\eqref{NJKI-10} by omitting the superscript ``$\mu$". The main task of this section is to give the proof 
of Theorem \ref{T1.1}.
{{To}} achieve this, we first assume  that 
for $1<s\leq\frac{3}{2}$, the
following inequality holds:
\begin{align}\label{G3.1}
\sup_{t\in[0,T)} \Big\{\|(\varrho,u)\|_{H^3\cap \dot{B}_{2,\infty}^{-s}}+\|f\|_{L_v^2(H^3\cap\dot{B}_{2,\infty}^{-s})}\Big\} \leq \sigma, 
\end{align}
for some time $T>0$. Here $\sigma\in(0,1)$ is a universal constant independent of $T$ and $\mu$,  which will be chosen later.
By leveraging the embedding $H^2\hookrightarrow L^{\infty}$ and 
the inequality \eqref{G3.1},
we discover that there exists a generic constant $\sigma_0\in(0,1)$ such that 
\begin{align}\label{G3.4}
\frac{1}{2}\leq \varrho+1 \leq \frac{3}{2}\quad \text{if}\quad \sigma\leq \sigma_0.
\end{align}
 
To show Theorem \ref{T1.1}, the crucial part  
is to establish the uniform {\emph{a priori}} energy estimates for the solution to the NS-VFP system \eqref{rNSVFP}. 
\begin{prop}\label{P3.1}
Let $0<\mu<1$, and let $(\varrho,u,P,f)$ be a smooth solution to the Cauchy problem for \eqref{rNSVFP} with the initial data $(\varrho_0,u_0,f_0)$ on $[0,T)$ with some $T>0$. Assume that \eqref{G3.1} holds with a generic constant $\sigma\in(0,\sigma_0)$. Then, $(\varrho,u,P,f)$ has the uniform estimate:
\begin{align}\label{apropri}
& \|(\varrho,u)(t)\|_{H^3\cap \dot{B}_{2,\infty}^{-s}}^2+\|f(t)\|_{L_{v}^2(H^3\cap \dot{B}_{2,\infty}^{-s})}^2
\nonumber\\
&\quad+\mu\int_0^t\|\nabla u(\tau)\|_{H^3}^2{\rm d}\tau+\int_0^t \mathcal{D}(u,P,f)(\tau) {\rm d}\tau \nonumber\\
&\qquad\leq C^* \Big(\|(\varrho_0,u_0)\|_{H^3\cap \dot{B}_{2,\infty}^{-s}}^2+\|f_0\|_{L_{v}^2(H^3\cap \dot{B}_{2,\infty}^{-s} ) }^2\Big),
\end{align}
for any $t\in[0,T)$. Here $\mathcal{D}(u,P,f)$ is defined by \eqref{D}, and $C^*$ is a positive constant uniform with respect to $T$ and $\mu$. 
\end{prop}

The proof of Proposition \ref{P3.1} consists of Propositions
\ref{P3.6}, \ref{P3.9} and \ref{P3.10}, which will be presented in Subsections 3.1--3.3 below.

\subsection{Energy estimates of $\texorpdfstring{H^3}{H3}$ norms}

Before analyzing $\varrho$, $u$ and $f$, we  establish some essential estimates for the pressure function \(P(t,x)\) 
to handle the nonlinear term \(\frac{\varrho}{1+\varrho}\nabla P\) arising in \(\eqref{rNSVFP}_{2}\).
\begin{lem}\label{L3.1} 
Under the conditions of Proposition \ref{P3.1}, we have
\begin{align}
\|\nabla P\|_{L^2}\leq\,& C\big(\|\nabla a\|_{L^{2}}+\|\nabla u\|_{H^1}+\|b-u\|_{L^2}+\mu\|\nabla^2 u\|_{L^{2}}\big), \label{G3.2}\\
\|\nabla^2 P\|_{L^2}\leq \,&C\big( \|\nabla^2 u\|_{H^1}+\|\nabla^2 a\|_{L^2}+\|\nabla (b-u)\|_{L^2}+ \mu\|\nabla^3 u\|_{L^2}\big),\label{G3.6}\\
 \|\nabla^3  P\|_{L^2}\leq \,&C\big(\|\nabla^2 u\|_{H^{1}}+\|\nabla^2 a\|_{L^2}+\|\nabla (b-u)\|_{H^1}+\mu\|\nabla^3 u\|_{H^1}\big),\label{NJKG3.6}
\end{align}
for any $0 \leq t < T$. Here $C>0$ is a constant independent of $T$ and $\mu$. 
\end{lem}
\begin{proof}
Multiplying \eqref{rNSVFP}$_2$ by $\nabla P$, then integrating the result over $\mathbb{R}^3$ and using 
the constraint ${\rm div}\, u=0$, we deduce from  the assumption \eqref{G3.1} and Lemma \ref{LA.1} that
\begin{align}\label{G3.3}
\int_{\mathbb{R}^3} \frac{|\nabla P|^{2}}{1+\varrho}{\rm d}x
=\,&-\int_{\mathbb{R}^3} u\cdot\nabla u\cdot\nabla P{\rm d}x
+\int_{\mathbb{R}^3}\frac{\mu\Delta u\cdot\nabla P}{1+\varrho}{\rm d}x+\int_{\mathbb{R}^3} 
{(b-u-au)\cdot\nabla P}{\rm d}x\nonumber\\
\leq\,& C\big(\|u\|_{L^3}\|\nabla u\|_{L^6}+\mu\|\Delta u\|_{L^2}+\|b-u\|_{L^2}+\|a\|_{L^6}\|u\|_{L^3}\big)\|\nabla P\|_{L^2}\nonumber\\
\leq\,&C\sigma\|\nabla P\|_{L^2}^2+C\sigma\big(\|\nabla^2 u\|_{L^2}^2+\mu^2\|\nabla^2 u\|_{L^2}^2+\|\nabla a\|_{L^2}^2\big)+\frac{1}{4}\|\nabla P\|_{L^2}^2\nonumber\\
&+C(\|b-u\|_{L^2}^2+\|\nabla u\|_{L^2}^2).
\end{align}
Combining \eqref{G3.4} with \eqref{G3.3} and using the fact that $\sigma\leq 1$, we obtain \eqref{G3.2}.

Next, we continue to prove \eqref{G3.6} and \eqref{NJKG3.6}.
Applying $\partial^{\alpha}$  with $1\leq |\alpha|\leq 2$ to \eqref{rNSVFP}$_2$ and
 then multiplying the result by $\partial^{\alpha}\nabla P$,
using the assumption \eqref{G3.1} and integrating over $\mathbb{R}^3$, we arrive at 
\begin{align}\label{G3.7}
&\int_{\mathbb{R}^3}\frac{|\partial^\alpha\nabla P|^2}{1+\varrho}{\rm d}x\nonumber\\
=\,&-\sum_{0\leq|\beta|<|\alpha|}C_{\alpha,\beta}\int_{\mathbb{R}^3}\partial^{\alpha-\beta}\Big(\frac{1}{1+\varrho}\Big)\partial^\beta\nabla P\cdot\partial^\alpha \nabla P{\rm d}x
-\int_{\mathbb{R}^3}\partial^\alpha(u\cdot\nabla u)\cdot\partial^{\alpha}\nabla P{\rm d}x \nonumber\\
&+\mu\sum_{0\leq|\beta|<\alpha}C_{\alpha,\beta}\int_{\mathbb{R}^3}\partial^{\alpha-\beta}\Big(\frac{1}{1+\varrho}\Big)\partial^{\beta}\Delta u\cdot\partial^\alpha\nabla P{\rm d}x+\mu\int_{\mathbb{R}^3}\frac{\partial^\alpha\Delta u\cdot\partial^\alpha\nabla P}{1+\varrho}{\rm d}x\nonumber\\
&+\int_{\mathbb{R}^3}\partial^\alpha(b-u-au)\cdot\partial^\alpha\nabla P\mathrm{d}x\nonumber\\
\equiv:\,&\sum_{j=1}^5 I_j,
\end{align}
where $C_{\alpha,\beta}>0$ are   constants  depending  only on $\alpha$ and 
$\beta$.

Now we handle $I_j(j=1,\dots, 5)$ one by one. For the term $I_1$, when $|\alpha|=1$, according to Lemma \ref{LA.1} and the assumption \eqref{G3.1}, we have
\begin{align}\label{G3.8}
|I_1|
\leq\,& C\Big\|\partial^\alpha\Big(\frac{1}{1+\varrho} \Big) \Big\|_{L^3}\|\nabla P\|_{L^6}\|\nabla^2 P\|_{L^2}\nonumber\\
\leq\,& C\|\varrho\|_{H^3}\|\nabla^2 P\|_{L^2}^2\nonumber\\
\leq\,& C\sigma \|\nabla^2 P\|_{L^2}^2.
\end{align}
When $|\alpha|=2$, applying Lemma \ref{LA.1} guarantees that 
\begin{align}\label{G3.9}
|I_1|
\leq\,& C\Big\|\partial^\alpha\Big(\frac{1}{1+\varrho} \Big) \Big\|_{L^3}\|\nabla P\|_{L^6}\|\partial^\alpha\nabla P\|_{L^2}+C\Big\|\partial^{\alpha-\beta}\Big(\frac{1}{1+\varrho} \Big) \Big\|_{L^3}\|\partial^\beta\nabla P\|_{L^6}\|\partial^\alpha\nabla P\|_{L^2}\nonumber\\
\leq\,& C\|\varrho\|_{H^3}\big(\|\nabla^2 P\|_{L^2}^2+\|\nabla^3 P\|_{L^2}^2\big)\nonumber\\
\leq\,&C\sigma\|\nabla^2 P\|_{L^2}^2+C\sigma\|\nabla^3 P\|_{L^2}^2.
\end{align}
Combining \eqref{G3.8} and \eqref{G3.9}, we obtain
\begin{align}\label{G3.10}
|I_1| \leq    C\sigma\|\nabla^2 P\|_{L^2}^2+C\sigma\|\nabla^3 P\|_{L^2}^2.
\end{align}

For the term  $I_3$, when $|\alpha|=1$, by using the assumption \eqref{G3.1}  and Lemma \ref{LA.1}, we arrive at
\begin{align}\label{G3.11}
|I_3|
\leq\,&C\Big\|\partial^\alpha\Big(\frac{1}{1+\varrho} \Big) \Big\|_{L^3}\mu\|\Delta u\|_{L^6}\|\partial^\alpha\nabla P\|_{L^2}\nonumber\\
\leq\,&C\|\varrho\|_{H^3}\mu\|\nabla^3 u\|_{L^2}\|\nabla^2 P\|_{L^2}\nonumber\\
\leq\,&C\sigma\|\nabla^2 P\|_{L^2}^2+C\sigma \mu^2\|\nabla^3 u\|_{L^2}^2.
\end{align}
When $|\alpha|=2$, using the assumption \eqref{G3.1}  and Lemma \ref{LA.1} gives
\begin{align}\label{G3.12}
|I_3|\leq\,& C\mu\int_{\mathbb{R}^3}\Big|\partial^\alpha\Big(\frac{1}{1+\varrho}       \Big)\Delta u\cdot\partial^\alpha\nabla P    \Big|{\rm d}x\nonumber\\
  &+C \mu\int_{\mathbb{R}^3}\Big|\partial^{\alpha-\beta}\Big(\frac{1}{1+\varrho}       \Big)\partial^\beta\Delta u\cdot\partial^\alpha\nabla P    \Big|{\rm d}x \nonumber\\
\leq\,& C\mu\Big\|\partial^\alpha\Big(\frac{1}{1+\varrho} \Big) \Big\|_{L^3}\|\Delta u\|_{L^6}\|\partial^\alpha\nabla P\|_{L^2}\nonumber\\
&+C\mu\Big\|\partial^{\alpha-\beta}\Big(\frac{1}{1+\varrho} \Big) \Big\|_{L^3}\|\partial^\beta\Delta u\|_{L^6}\|\partial^\alpha\nabla P\|_{L^2}\nonumber\\
\leq\,& C\mu\|\varrho\|_{H^3}\|\nabla^3 u\|_{H^1}\|\nabla^3 P\|_{L^2}\nonumber\\
\leq\,& C\sigma\|\nabla^3 P\|_{L^2}^2+C\sigma \mu^2\|\nabla^3 u\|_{H^1}^2.
\end{align}
Combining \eqref{G3.11} and \eqref{G3.12}, we have
\begin{align}\label{G3.13}
|I_3| \leq    C\sigma\|\nabla^2 P\|_{L^2}^2+C\sigma\|\nabla^3 P\|_{L^2}^2+C\sigma \mu^2\|\nabla^3 u\|_{H^1}^2.
\end{align}

For the terms $I_2, I_4$ and $I_5$, it follows from the assumption \eqref{G3.1} and Lemma \ref{LA.1} that
\begin{align}\label{G3.14}
|I_2|\leq\,& \|\partial^\alpha(u\cdot\nabla u)\|_{L^2}\|\partial^\alpha\nabla P\|_{L^2}\nonumber\\
\leq\,& C\|\nabla u\|_{H^1}\|\nabla^2  u\|_{H^{1}}\|\partial^\alpha \nabla P\|_{L^2}\nonumber\\
\leq\,& C\sigma\|\partial^\alpha\nabla P\|_{L^2}^2+C\sigma\|\nabla^2  u\|_{H^1}^2,\\\label{G3.15}
|I_4|\leq\,&\mu \|\partial^\alpha\Delta u\|_{L^2}\|\partial^\alpha\nabla P\|_{L^2}\nonumber\\
\leq\,&\frac{1}{8}\|\partial^\alpha\nabla P\|_{L^2}^2+C\mu^2 \|\partial^{\alpha}\nabla^2 u\|_{L^2}^2,
\end{align}
and
\begin{align}\label{G3.16}
|I_5|\leq\,& \|\partial^\alpha\nabla P\|_{L^2}\|\partial^\alpha (b-u-au)\|_{L^2}\nonumber\\
\leq\,&\frac{1}{8}\|\partial^\alpha\nabla P\|_{L^2}^2+C\|\partial^\alpha(b-u)\|_{L^2}^2+C\| a\|_{L^3}^2\|\nabla u\|_{L^6}^2+C\|u\|_{L^3}^2\|\nabla a\|_{L^6}^2\nonumber\\
&+C\|a\|_{L^{\infty}}^2\|\nabla^2 u\|_{L^2}^2+C\|u\|_{L^{\infty}}^2\|\nabla^2 a\|_{L^2}^2+C\|\nabla u\|_{L^6}^2\|\nabla a\|_{L^3}^2\nonumber\\
\leq\,&\frac{1}{8}\|\partial^\alpha\nabla P\|_{L^2}^2+C\big(\|\nabla^2 u\|_{L^{2}}^2+\sigma^2\|\nabla^2 a\|_{L^2}^2+\|\partial^{\alpha} (b-u)\|_{L^2}^2\big).
\end{align}
For $|\alpha|=1$, by substituting the estimates \eqref{G3.8}, \eqref{G3.11}, \eqref{G3.14}--\eqref{G3.16} into \eqref{G3.7}, we derive \eqref{G3.6}. For $|\alpha|=2$, by incorporating the estimates \eqref{G3.10}, \eqref{G3.13}, \eqref{G3.14}--\eqref{G3.16}, and \eqref{G3.6} into \eqref{G3.7}, we ultimately obtain \eqref{NJKG3.6}.
\end{proof}

With Lemma \ref{L3.1} in hand, we are ready to derive the $L^2$ energy inequality.

\begin{lem}\label{L3.3}
Under the conditions of Proposition \ref{P3.1}, there exists a positive constant $\lambda_1$ such that
\begin{align}\label{G3.17}
&\frac{\rm d}{{\rm d}t}\big(\|u\|_{L^2}^{2}+\|f\|_{L_{x,v}^2}^{2}\big)+\lambda_1\big(\|b-u\|_{L^2}^{2}
+{\mu}\|\nabla u\|_{L^2}^2+\|\{\mathbf{I}-\mathbf{P}\}f\|_{L^2_{v,\nu}(L^2)}^{2}\big)\nonumber\\
&\quad\leq C\sigma\big(\|\nabla(a,b)\|_{L^2}^2+\|\nabla(b-u)\|_{L^2}^2+\|\nabla u\|_{H^1}^2+\mu\|\nabla^2 u\|_{L^2}^2\big),
\end{align}
for any $t\in[0,T)$.
\end{lem}
\begin{proof}
Multiplying \eqref{rNSVFP}$_2$  and \eqref{rNSVFP}$_4$ by $u$ and $f$ respectively,
then integrating over $\mathbb{R}^3$ and $\mathbb{R}^3\times \mathbb{R}^3$, and adding them up, we obtain
\begin{align}\label{G3.18}
&\frac{1}{2}\frac{{\rm d}}{{\rm d}t}\big( \|u\|_{L^2}^{2} +\|f\|_{L_{x,v}^2}^{2}  \big) +\|u-b\|_{L^2}^2+\int_{\mathbb{R}^3}(\varrho+1)\langle-\mathcal{L}\{\mathbf{I}-\mathbf{P}\}f,f\rangle{\rm d}x+\mu\int_{\mathbb{R}^3}\frac{|\nabla u|^2}{1+\varrho}{\rm d}x\nonumber\\
=\,&-\mu\int_{\mathbb{R}^3}\nabla\Big( \frac{1}{1+\varrho}    \Big)\cdot\nabla u\cdot u{\rm d}x +\int_{\mathbb{R}^3}\frac{\varrho}{1+\varrho}\nabla P\cdot u{\rm d}x +\frac{1}{2}\int_{\mathbb{R}^3}u\cdot\langle vf,f\rangle{\rm d}x             \nonumber\\
&-\int_{\mathbb{R}^3}a|u|^2{\rm d}x+\int_{\mathbb{R}^3}\varrho\Big\langle \mathcal{L}\mathbf{P}f-u\cdot\nabla_v f+\frac{1}{2}u\cdot vf+u\cdot v\sqrt{M},f\Big\rangle{\rm d}x\nonumber\\
\equiv:\,&\sum_{j=1}^{5}L_j.
\end{align}
By \eqref{G3.1}  and Lemma \ref{LA.1}, it
follows that
\begin{align}\label{G3.19}
|L_1|\leq \,&\Big\|\nabla\Big(\frac{1}{1+\varrho}  \Big) \Big\|_{L^3}\|\nabla u\|_{L^2}\|u\|_{L^6}
\leq C\|\varrho\|_{H^3}\|\nabla u\|_{L^2}^2\leq\,C\sigma \|\nabla u\|_{L^2}^2,\\\label{G3.20}
|L_2|\leq \, & \|\varrho\|_{H^3}\|\nabla P\|_{L^2}\|u\|_{L^6}\leq C\sigma\|\nabla P\|_{L^2}\|\nabla u\|_{L^2}.
\end{align}
Thanks to the estimate of $\nabla P$ in Lemma \ref{L3.1}, we get
\begin{align}\label{G3.21}
|L_2|\leq C\sigma \big(\|\nabla u\|_{H^1}^2+\mu^2\|\nabla u\|_{H^1}^2+\|\nabla a\|_{L^2}^2+\|b-u\|_{L^2}^2\big).   
\end{align}
To address the lower-order nonlinear terms $L_3$ and $L_4$, we employ the key cancellation mechanism observed in \cite{CDM-krm-2011}. By  the macro-micro  decomposition \eqref{mmd}
  and $\mathbf{P}f=a^f \sqrt{M}+b^fv\sqrt{M}$, we have
\begin{align}\label{G3.22}
|L_3+L_4|
=\Big|\,&\frac{1}{2}\int_{\mathbb{R}^3}u\cdot\langle v\{\mathbf{I}-\mathbf{P}\}f,\{\mathbf{I}-\mathbf{P}\}f\rangle\mathrm{d}x+\int_{\mathbb{R}^3}u\cdot\langle v\mathbf{P}f,\{\mathbf{I}-\mathbf{P}\}f\rangle\mathrm{d}x\nonumber\\
&\quad+\int_{\mathbb{R}^3}au\cdot(b-u)\mathrm{d}x\Big|\nonumber\\
\leq\,&C\|u\|_{L^{\infty}}\|\{\mathbf{I}-\mathbf{P}\}f\|_{L^2_{v,\nu}(L^2)}^2+C\|u\|_{L^3}\|(a,b)\|_{L^6}\|\{\mathbf{I}-\mathbf{P}\}f\|_{L^2_{v,\nu}(L^2)}\nonumber\\
&\quad+C\|b-u\|_{L^2}\|u\|_{L^3}\|a\|_{L^6}\nonumber\\
\leq\,&C\sigma\big(\|b-u\|_{L^2}^2+\|\nabla(a,b)\|_{L^2}^2+ \|\{\mathbf{I}-\mathbf{P}\}f\|_{L^2_{v,\nu}(L^2)}^2        \big).   
\end{align}
For the remaining term $L_{5}$, it follows from the assumption \eqref{G3.1}, Lemma \ref{LA.1} and $\mathbf{P}f=a^f \sqrt{M}+b^fv\sqrt{M}$ that
\begin{align}\label{G3.23}
|L_{5}|=\,&\frac{1}{2}\int_{\mathbb{R}^3}\varrho u\cdot\langle vf,f\rangle\mathrm{d}x+\int_{\mathbb{R}^3}\varrho(u-b)\cdot b\mathrm{d}x\nonumber\\
\leq\,& \frac{1}{2}\int_{\mathbb{R}^3}\varrho u\cdot\langle v\{\mathbf{I}-\mathbf{P}\}f,\{\mathbf{I}-\mathbf{P}\}f\rangle\mathrm{d}x+\int_{\mathbb{R}^3}\varrho u\cdot\langle v\mathbf{P}f,\{\mathbf{I}-\mathbf{P}\}f\rangle\mathrm{d}x\nonumber\\
&+\int_{\mathbb{R}^3}\varrho u \cdot ab\mathrm{d}x+C\|\varrho\|_{L^3}\|u-b\|_{L^2}\|b\|_{L^6}\nonumber\\
\leq\,&C\|\varrho\|_{H^3}\|u\|_{H^3}\|\{\mathbf{I}-\mathbf{P}\}f\|_{L^2_{v,\nu}(L^2)}^2+C\|\varrho\|_{H^3}\|u\|_{H^3}\|\nabla(a,b)\|_{L^2}^2\nonumber\\
&+C\|\varrho\|_{H^3}\|u\|_{H^3}\|(a,b)\|_{L^6}\|\{\mathbf{I}-\mathbf{P}\}f\|_{L^2_{v,\nu}(L^2)}+ C\|\varrho\|_{H^3}\|u-b\|_{L^2}\|\nabla (a,b)\|_{L^2}\nonumber\\
\leq\,&C\sigma \big( \|u-b\|_{L^2}^2+\|\nabla(a,b)\|_{L^2}^2
+\|\{\mathbf{I}-\mathbf{P}\}f\|_{L^2_{v,\nu}(L^2)}^2\big).    
\end{align}
Using  \eqref{H2.1} and \eqref{G3.4}, substituting \eqref{G3.19}--\eqref{G3.23} into \eqref{G3.18},
and then using the   smallness of  $\sigma$ in \eqref{G3.1}, we obtain
\eqref{G3.17}.
\end{proof}

Then, we establish the energy inequality of the derivatives for $u$ and $f$. 

\begin{lem}\label{L3.4}
Under the conditions of Proposition \ref{P3.1}, there exists a positive constant $\lambda_2$ such that
\begin{align}\label{G3.24}
&\frac{\rm d}{{\rm d}t}\sum_{1\leq|\alpha|\leq 3}\big( \|\partial^\alpha u\|_{L^2}^{2}+\|\partial^\alpha f\|_{L_{x,v}^2}^{2}\big)\nonumber\\
&\quad+\lambda_2\sum_{1\leq |\alpha|\leq3}\big(\|\partial^\alpha(b-u)\|_{L^2}^{2}
+{\mu}\|\nabla\partial^\alpha u\|_{L^2}^2+\|\partial^\alpha\{\mathbf{I}-\mathbf{P}\}f\|_{L^2_{v,\nu}(L^2)}^{2}\big)\nonumber\\
&\quad\quad\leq C\sigma(\|\nabla(a,b)\|_{H^2}^2+\|\nabla u\|_{H^2}^2+\mu \|\nabla u\|_{L^2}^2),
\end{align}
for any $t\in[0,T)$ .
\end{lem}
\begin{proof}
Applying $\partial^\alpha$ with $1\leq |\alpha|\leq 3$ to
\eqref{rNSVFP}$_2$ and \eqref{rNSVFP}$_4$, and multiplying the resulting identities by $\partial^\alpha u$ and
 $ \partial^\alpha f$ respectively, and then integrating and summing the outcomes, we have
\begin{align}\label{G3.25}
&\frac12\frac {\rm d}{{\rm d}t}\big(\|\partial^{\alpha}u\|_{L^2}^{2}+\|\partial^{\alpha}f\|_{L_{x,v}^2}^{2}\big)
+\|\partial^{\alpha}(b-u)\|_{L^2}^{2}
 \nonumber\\
&+\int_{\mathbb{R}^3}\left\langle-\mathcal{L}\{\mathbf{I}-\mathbf{P}\}\partial^{\alpha}f,
\partial^{\alpha}f\right\rangle {\rm d}x+\mu \int_{\mathbb{R}^3}\frac{|\nabla\partial^\alpha u|^2}{1+\varrho}  \mathrm{d}x                  \nonumber\\
=\,& -\mu\int_{\mathbb{R}^3}\nabla\Big(\frac{1}{1+\varrho}\Big)\nabla\partial^\alpha u\cdot \partial^{\alpha} u   \mathrm{d}x+\mu\sum_{0\leq|\beta|<|\alpha|}C_{\alpha,\beta}\int_{\mathbb{R}^3}\partial^{\alpha-\beta}\Big(\frac{1}{1+\varrho}\Big)\partial^\beta\Delta u\cdot\partial^\alpha u \mathrm{d}x\nonumber\\
&+\frac{1}{2}\int_{\mathbb{R}^3}
\big\langle\partial^{\alpha}\big((1+\varrho)u\cdot vf\big),\partial^{\alpha}f\big\rangle {\rm d}x+\int_{\mathbb{R}^3}\langle[-\partial^{\alpha},u\cdot\nabla_{v}]f,\partial^{\alpha}f\rangle {\rm d}x\nonumber\\
&+\int_{\mathbb{R}^3}
\big\langle\partial^{\alpha}\big(\varrho(\mathcal{L}f-u\cdot\nabla_v f+u\cdot v\sqrt{M})\big),\partial^{\alpha}f\big\rangle {\rm d}x+\int_{\mathbb{R}^3}\nabla\Big( \frac{1}{1+\varrho}  \Big)\partial^\alpha P\cdot\partial^\alpha u{\rm d}x\nonumber\\
&-\sum_{0\leq|\beta|<|\alpha|}C_{\alpha,\beta}\int_{\mathbb{R}^3}\partial^{\alpha-\beta}\Big( 
 \frac{1}{1+\varrho}  \Big)\partial^\beta\nabla P\cdot\partial^\alpha u{\rm d}x-\int_{\mathbb{R}^3}\partial^{\alpha}(au)\cdot\partial^{\alpha}u\mathrm{d}x\nonumber\\
&+\int_{\mathbb{R}^3}[-\partial^{\alpha},u\cdot\nabla]u\cdot\partial^{\alpha}u {\rm d}x   \nonumber\\
\equiv:\,&\sum_{k=1}^{9} J_k. 
\end{align}
By leveraging \eqref{G3.1} and Lemma \ref{LA.1},
we reach
\begin{align}\label{G3.26}
 |J_1|\leq\,& C\mu\Big\|\nabla\Big(\frac{1}{1+\varrho}  \Big)  \Big\|_{L^3}\|\nabla\partial^\alpha u\|_{L^2}\|\partial^\alpha u\|_{L^6}\nonumber\\
 \leq\,& C\mu\|\varrho\|_{H^3}\|\nabla\partial^\alpha u\|_{L^2}^2\nonumber\\
 \leq\,& C\sigma \mu\|\nabla^2 u\|_{H^2}^2,
 \end{align}
 and
 \begin{align}\label{G3.27}
|J_2|\leq\,& C\mu
\left\{\begin{aligned} 
&\Big\|\partial^{\alpha}\Big(\frac{1}{1+\varrho}\Big)\Big\|_{L^2}\|\Delta u\|_{L^{\infty}}\|\partial^{\alpha}u\|_{L^2}
&\quad(|\beta|=0)\\ 
&\Big\|\partial^{\alpha-\beta}\Big(\frac{1}{1+\varrho}\Big)\Big\|_{L^{3}}\|\partial^{\beta}\Delta u\|_{L^{6}}
\|\partial^{\alpha}u\|_{L^2}&\quad(|\beta|=1)\\ 
&\Big\|\partial^{\alpha-\beta}\Big(\frac{1}{1+\varrho}\Big)\Big\|_{L^{3}}\|\partial^{\beta}\Delta u\|_{L^2}
\|\partial^{\alpha}u\|_{L^6}&\quad(|\beta|= 2)
\end{aligned}\right.\nonumber\\
\leq\,&C \mu\|\varrho\|_{H^3}\|\nabla ^2u\|_{H^2}^2\nonumber\\
\leq\,& C\sigma \mu\|\nabla^2 u\|_{H^2}^2.
\end{align}
Note that 
 $\|v\mathbf{P}f\|_{L_{x,v}^2}+\|\nabla_{v}\mathbf{P}f\|_{L_{x,v}^2}\lesssim \|a\|_{L^2}+\|b\|_{L^2}$. 
 It follows from the macro-micro decomposition
\eqref{mmd}, the assumption \eqref{G3.1} and Lemma \ref{LA.1} that
\begin{align}\label{G3.28}
|J_3|\leq\,& C\big\|\nabla\big((1+\varrho)u\big)\big\|_{H^2}\|\nabla f\|_{L_v^2(H^2)}\|v\partial^\alpha f\|_{L_{x,v}^2}  \nonumber\\
\leq\,& C(1+\|\varrho\|_{H^3})\|\nabla u\|_{H^2}\big(\|\partial^\alpha\{\mathbf{I}-\mathbf{P}\}f\|_{L^2_{v,\nu}(L^2)}^2+\|\nabla (a,b)\|_{H^2}^2\big)\nonumber\\
\leq\,& C\sigma\big(\|\partial^\alpha\{\mathbf{I}-\mathbf{P}\}f\|_{L^2_{v,\nu}(L^2)}^2+\|\nabla (a,b)\|_{H^2}^2\big).
\end{align}
From Lemma \ref{LA.2}, it follows that
\begin{align}\label{G3.29}
|J_4|\leq\,& C\|\partial^\alpha f\|_{L_{x,v}^2}\big(\|\nabla u\|_{L^{\infty}}\|\nabla_v f\|_{L_v^2(H_x^2)}+\|\nabla_v f\|_{L_v^2(L^{\infty})}\|\partial^\alpha u\|_{L^2}\big) \nonumber\\
\leq\,& C\|f\|_{L_v^2(H_x^3)}\big(\|\partial^\alpha\{\mathbf{I}-\mathbf{P}\}f\|_{L^2_{v,\nu}(L^2)}^2+\|\nabla (a,b)\|_{H^2}^2\big)\nonumber\\
\leq\,& C\sigma\big(\|\partial^\alpha\{\mathbf{I}-\mathbf{P}\}f\|_{L^2_{v,\nu}(L^2)}^2+\|\nabla (a,b)\|_{H^2}^2\big).
\end{align}
For the term $J_5$, we can apply integration by parts, the assumption   \eqref{G3.1}, and  Lemmas \ref{LA.1}--\ref{LA.2} to obtain
\begin{align}\label{G3.30}
|J_5|\leq \,& C\bigg\|\partial^\alpha\bigg(\varrho\sqrt{M}\nabla_v\Big(\frac{f}{\sqrt{M}}  \Big)    \bigg)\bigg\|_{L_{x,v}^2}\bigg\|\partial^\alpha\bigg(\sqrt{M}\nabla_v\Big(\frac{f}{\sqrt{M}}  \Big)    \bigg)\bigg\|_{L_{x,v}^2}\nonumber\\
&+C\|\partial^\alpha(\varrho u f)\|_{L_{x,v}^2}\|\nabla_v\partial^\alpha f\|_{L_{x,v}^2}
+C\|\partial^\alpha(\varrho u)\|_{L^2}\|\partial^\alpha b\|_{L^2}\nonumber\\
\leq\,& C(\|\varrho\|_{H^3}+\|\varrho\|_{H^3}\|u\|_{H^3})\Big(\|\nabla (a,b,u)\|_{H^2}^2
+\sum_{1\leq |\alpha|\leq3}\|\partial^\alpha\{\mathbf{I}-\mathbf{P}\}f\|_{L^2_{v,\nu}(L^2)}^{2}\Big)\nonumber\\
\leq\,& C\sigma \Big(\|\nabla (a,b,u)\|_{H^2}^2
+\sum_{1\leq |\alpha|\leq3}\|\partial^\alpha\{\mathbf{I}-\mathbf{P}\}f\|_{L^2_{v,\nu}(L^2)}^{2}\Big).
\end{align}
For the term $J_6$, by utilizing the estimates of $\nabla P$ derived in
Lemma \ref{L3.1}, we obtain
\begin{align}\label{G3.31}
|J_6|\leq\,& C\Big\|\nabla\Big(\frac{1}{1+\varrho}     \Big)\Big\|_{L^3}\|\nabla P\|_{L^6}\|\nabla u\|_{L^2}+C\Big\|\nabla\Big(\frac{1}{1+\varrho}     \Big)\Big\|_{L^{\infty}}
\|\nabla^2 P\|_{H^1}\|\nabla u\|_{H^2}\nonumber\\
\leq\,& C\|\varrho\|_{H^3}\|\nabla^2 P\|_{H^1}\|\nabla u\|_{H^2}\nonumber\\
\leq\,& C\sigma\big( \mu\|\nabla u\|_{H^3}^2+\|\nabla u\|_{H^2}^2+\|\nabla (a,b)\|_{H^2}^2+\|\nabla (b-u)\|_{H^1}^2\big).
\end{align}
Similar computations lead to
\begin{align}\label{G3.32}
|J_7|\leq\,& C
\left\{\begin{aligned}
&\Big\|\partial^{\alpha}\Big(\frac{1}{1+\varrho}\Big)\Big\|_{L^2}\|\nabla P\|_{L^{6}}\|\partial^{\alpha}u\|_{L^3}
&\quad(|\beta|=0)\\
&\Big\|\partial^{\alpha-\beta}\Big(\frac{1}{1+\varrho}\Big)\Big\|_{L^{3}}\|\partial^{\beta}\nabla P\|_{L^{6}}
\|\partial^{\alpha}u\|_{L^2}&\quad(|\beta|=1)\\
&\Big\|\partial^{\alpha-\beta}\Big(\frac{1}{1+\varrho}\Big)\Big\|_{L^{\infty}}\|\partial^{\beta}\nabla P\|_{L^2}
\|\partial^{\alpha}u\|_{L^2}&\quad(|\beta|= 2)\end{aligned}\right.\nonumber\\
\leq\,&C \|\varrho\|_{H^3}\|\nabla^2 P\|_{H^1}\|\nabla u\|_{H^2}\nonumber\\
\leq\,& C\sigma\big(\mu\|\nabla u\|_{H^3}^2 +\|\nabla u\|_{H^2}^2+\|\nabla(a,b)\|_{H^2}^2+\|\nabla(b-u)\|_{H^1}^2\big),
\end{align}
and
\begin{align}\label{H3.35}
|J_8+J_9|&\leq\, C\|\partial^\alpha u\|_{L^2}\|\nabla a\|_{H^2}\|\nabla u\|_{H^2}+C\|\partial^\alpha u\|_{L^2}\|\nabla u\|_{L^{\infty}}\|\partial^\alpha u\|_{L^2}\nonumber\\
&\leq\, C\|u\|_{H^3}\|\nabla u\|_{H^2}^2+C\|u\|_{H^2}\|\nabla u\|_{H^2}^2\nonumber\\
&\leq\, C\sigma \|\nabla u\|_{H^2}^2.
\end{align}
Substituting the estimates \eqref{G3.26}--\eqref{H3.35} into \eqref{G3.25},
we ultimately obtain \eqref{G3.24}.
\end{proof}

To absorb the terms on the right-hand side of \eqref{G3.17} and \eqref{G3.23}, the next step in the energy estimates is to derive the energy dissipation rate $\|\nabla (a,b)\|_{H^2}$.

\begin{lem}\label{L3.5}
Under the conditions of Proposition \ref{P3.1},  there exists a uniform constant $\lambda_3>0$ and a functional $\mathcal{E}_{0}$ satisfying $|\mathcal{E}_{0}|\lesssim \|f\|_{L^2_v(H^3)}^2$  such that
\begin{align}\label{G3.40}
\frac{\rm d}{{\rm d}t}\mathcal{E}_{0}(t)+\lambda_3\|\nabla(a,b)\|_{H^{2}}^{2}
\leq&\, C\big(\|\{\mathbf{I}-\mathbf{P}\}f\|_{L_{v}^{2}(H^{3})}^{2}+\|b-u\|_{H^{2}}^{2}\big), 
\end{align} 
for any $t\in[0,T)$.
\end{lem}

\begin{proof}
To capture the dissipation of $a$ and $b$, we shall  make full use of the interactions of $v\cdot\nabla \mathbf{P}f$ and $\mathcal{L}f$ as in \cite{CDM-krm-2011,Gy-IUMJ-2004}. To that end, from $\eqref{rNSVFP}_4$, we proceed by the following 
macroscopical equations involving $a$ and $b$:
\begin{equation}\label{G3.36}
\left\{\begin{aligned}
&\partial_{t}a+{\rm div}\, b=0,\\
&\partial_{t} b_i+\partial_{i} a+\sum_{j=1}^3\partial_j\Gamma_{ij}(\{\mathbf{I}-\mathbf{P}\}f)=(1+\varrho)(u_i-b_i)+(1+\varrho)u_ia,  \\
&\partial_{i}b_j+\partial_j b_i-(1+\varrho)(u_ib_j+u_jb_i)=-\partial_t \Gamma_{ij}(\{\mathbf{I}-\mathbf{P}\}f)+\Gamma_{ij}(\mathfrak{l}+\mathfrak{r}+\mathfrak{s}),  
 \end{aligned}
 \right.
\end{equation}
for $1\leq i,j\leq 3$, 
where $\Gamma_{ij}(\cdot)$ is the  moment functional
\begin{align}\label{G3.35}
\Gamma_{ij}(\cdot):=\langle(v_iv_j-1)\sqrt{M},\cdot\rangle,
\end{align}
and $\mathfrak{l}$, $\mathfrak{r}$ and $\mathfrak{s}$ are given by
\begin{align*}
\mathfrak{l}\,&:=\mathcal{L}\{\mathbf{I}-\mathbf{P}\}f-v\cdot\nabla\{\mathbf{I}-\mathbf{P}\}f, \nonumber \\
\mathfrak{r}\,&:=-u\cdot\nabla_v\{\mathbf{I}-\mathbf{P}\}f+\frac{1}{2}u\cdot v\{\mathbf{I}-\mathbf{P}\}f,  \nonumber\\
\mathfrak{s}\,&:=\frac{\varrho}{\sqrt{M}}\nabla_v\cdot\Big(\nabla_v\big(\sqrt{M}\{\mathbf{I}
-\mathbf{P}\}f\big)+v\sqrt{M}\{\mathbf{I}-\mathbf{P}\}f-u\sqrt{M}\{\mathbf{I}-\mathbf{P}\}f   \Big).
\end{align*}

We first address the dissipation estimates of $a$. From  \eqref{G3.36}$_1$ and \eqref{G3.36}$_2$, one has
\begin{align}\label{GH3.26}
&\frac{{\rm d}}{{\rm d}t}\sum_{i=1}^3\int_{\mathbb{R}^3}\partial^\alpha\partial_i a\partial^\alpha b_i\mathrm{d}x+ \| \partial^\alpha\nabla a\|_{L^2}^2-\|\partial^\alpha{\rm div}\, b\|_{L^2}^2\nonumber \\
&\quad =\sum_{i=1}^3\int_{\mathbb{R}^3}\partial^\alpha\partial_i a\partial^\alpha\bigg(-\sum_{j=1}^3\partial_j\Gamma_{ij}(\{\mathbf{I}-\mathbf{P}\}f) 
+(1+\varrho)(u_i-b_i)+(1+\varrho)u_ia        \bigg)\mathrm{d}x.
\end{align}
Here the nonlinear term on the right-hand side of \eqref{GH3.26} can be controlled by  
\begin{align}\label{GH3.28}
 &\sum_{i=1}^3\int_{\mathbb{R}^3}\partial^\alpha\partial_i a\partial^\alpha\bigg(-\sum_{j=1}^3\partial_j\Gamma_{ij}(\{\mathbf{I}-\mathbf{P}\}f) 
+(1+\varrho)(u_i-b_i)+(1+\varrho)u_ia        \bigg)\mathrm{d}x\nonumber\\
\leq\,&\frac{1}{4}\|\nabla a\|_{H^2}^2+C\|\nabla\{\mathbf{I}-\mathbf{P}\}f\|_{L_v^2(H^2)}^2+C(1+\|\varrho\|_{H^3})^2\|u-b\|_{H^2}^2\nonumber\\
&+C(1+\|\varrho\|_{H^3})^2\|u\|_{H^2}^2\|\nabla a\|_{H^2}^2\nonumber\\
\leq\,&\Big(\frac{1}{4}+C\sigma\Big)\|\nabla a\|_{H^2}^2+C\|u-b\|_{H^2}^2+C\|\{\mathbf{I}-\mathbf{P}\}f\|_{L_v^2(H^3)}^2.
\end{align}
where we have  used the assumption \eqref{G3.1}, Lemma \ref{LA.1}  and
 the fact that $\Gamma_{ij}(\cdot)$ can absorb any velocity derivative
and velocity weight.

Thanks to the identity $$\sum_{i,j=1}^3\|\partial^\alpha(\partial_ib_j+\partial_jb_i)\|_{L^2}^2=2\|\nabla\partial^\alpha b\|_{L^2}^2+2\|{\rm div}\,\partial^\alpha b\|_{L^2}^2,$$ 
the dissipation of $b$ can be derived from \eqref{G3.36}$_3$:
\begin{align}\label{GH3.21}
&\frac{{\rm d}}{{\rm d}t}\sum_{i,j=1}^3\int_{\mathbb{R}^3}\partial^\alpha(\partial_ib_j+\partial_jb_i)\partial^\alpha\Gamma_{ij}(\{\mathbf{I}-\mathbf{P}\}f)\mathrm{d}x+2\|\nabla\partial^\alpha b\|_{L^2}^2+2\|{\rm div}\,\partial^\alpha b\|_{L^2}^2\nonumber\\
&\quad =\sum_{i,j=1}^3\int_{\mathbb{R}^3}\partial^\alpha(\partial_i\partial_t b_j+\partial_j\partial_t b_i)\partial^\alpha\Gamma_{ij}(\{\mathbf{I}-\mathbf{P}\}f)\mathrm{d}x\nonumber\\
&\qquad+\sum_{i,j=1}^3\int_{\mathbb{R}^3}\partial^\alpha(\partial_ib_j+\partial_jb_i)\partial^\alpha
\big((1+\varrho)(u_ib_j+u_jb_i)+\Gamma_{ij}(\mathfrak{l}+\mathfrak{r}+\mathfrak{s})   \big)\mathrm{d}x.
\end{align}
It follows from the assumption \eqref{G3.1} and Lemma \ref{LA.1}  that
\begin{align}\label{GH3.22}
&\sum_{i,j=1}^3\int_{\mathbb{R}^3}\partial^\alpha(\partial_i\partial_t b_j+\partial_j\partial_t b_i)\partial^\alpha\Gamma_{ij}(\{\mathbf{I}-\mathbf{P}\}f)\mathrm{d}x \nonumber\\ 
=\,& -2\sum_{i,j=1}^3\int_{\mathbb{R}^3}\partial^\alpha\partial_t b_i
\partial^\alpha\partial_j\Gamma_{ij}(\{\mathbf{I}-\mathbf{P}\}f)\mathrm{d}x\nonumber\\
=\,& 2\sum_{i,j=1}^3\int_{\mathbb{R}^3}\partial^\alpha\bigg(\partial_i a+\sum_{m=1}^3\partial_m\Gamma_{im}(\{\mathbf{I}-\mathbf{P}\}f)       \bigg)\partial^\alpha\partial_j\Gamma_{ij}(\{\mathbf{I}-\mathbf{P}\}f)\mathrm{d}x\nonumber\\
&-2\sum_{i,j=1}^3\int_{\mathbb{R}^3}\partial^\alpha\big((1+\varrho)(u_i-b_i)+(1+\varrho)u_i a  \big)\partial^\alpha\partial_j\Gamma_{ij}(\{\mathbf{I}-\mathbf{P}\}f)\mathrm{d}x\nonumber\\
\leq\,& \frac{1}{4}\|\nabla a\|_{H^2}^2+C\|\nabla \{\mathbf{I}-\mathbf{P}\}f \|_{L_v^2(H^2)}^2+C(1+\|\varrho\|_{H^{3}})^2\|u-b\|_{H^2}^2\nonumber\\
&+C(1+\|\varrho\|_{H^3})^2\|u\|_{H^2}^2\|\nabla a\|_{H^2}^2\nonumber\\
\leq\,&\Big(\frac{1}{4}+C\sigma\Big)\|\nabla a\|_{H^2}^2+C\|\nabla \{\mathbf{I}-\mathbf{P}\}f \|_{L_v^2(H^2)}^2+C\|u-b\|_{H^2}^2.
\end{align}
By a similar computation to that  in \eqref{GH3.28}, one has
\begin{align}\label{GH3.23}
&\sum_{i,j=1}^3\int_{\mathbb{R}^3}\partial^\alpha(\partial_ib_j+\partial_jb_i)
\partial^\alpha\big((1+\varrho)(u_ib_j+u_jb_i)+\Gamma_{ij}(\mathfrak{l}+\mathfrak{r}+\mathfrak{s})      \big)\mathrm{d}x\nonumber\\
\quad \leq\, &\frac{1}{2}\sum_{i,j=1}^3\|\partial^\alpha(\partial_ib_j+\partial_jb_i)\|_{L^2}^2+C\sum_{i,j=1}^3
\big\|\partial^\alpha\big((1+\varrho)(u_ib_j+u_jb_i)\big)\big\|_{L^2}^2\nonumber\\
\qquad& +C\sum_{i,j=1}^3\big(\|\partial^\alpha\Gamma_{ij}(\mathfrak{l})\|_{L ^2}^2
+\|\partial^\alpha\Gamma_{ij}(\mathfrak{r})\|_{L ^2}^2+\|\partial^\alpha\Gamma_{ij}(\mathfrak{s})\|_{L ^2}^2                  \big)\nonumber\\
 \quad \leq\,& C\|(1+\|\nabla\varrho\|_{H^2})^2\|u\otimes b\|_{H^2}^2+C\| \{\mathbf{I}-\mathbf{P}\}f \|_{L_v^2(H^3)}^2\nonumber\\
 \qquad&+C\|u\|_{H^2}^2\|\nabla  \{\mathbf{I}-\mathbf{P}\}f \|_{L_v^2(H^2)}^2+C(1+\|\varrho\|_{H^3})^2\|\nabla u\|_{H^2}^2\|\nabla  \{\mathbf{I}-\mathbf{P}\}f \|_{L_v^2(H^2)}^2\nonumber\\
\quad\leq\, & C\sigma\|\nabla b\|_{H^2}^2+C\sigma\| \{\mathbf{I}-\mathbf{P}\}f \|_{L_v^2(H^3)}^2,
\end{align}
where we  have used the assumption \eqref{G3.1}, Lemma \ref{LA.1}  and
 the fact that $\Gamma_{ij}(\cdot)$ can absorb any velocity derivative
and velocity weight.

We then denote the temporal functional $\mathcal{E}_0(t)$ as
\begin{align}\label{G3.39}
\mathcal{E}_0(t):=\sum_{|\alpha|\leq 2}\sum_{i,j=1}^3\int_{\mathbb{R}^3}\partial^\alpha(\partial_ib_j+\partial_jb_i)\partial^{\alpha}\Gamma_{ij}(\{\mathbf{I}-\mathbf{P}\}f)\mathrm{d}x-\sum_{|\alpha|\leq 2}\int_{\mathbb{R}^3}\partial^\alpha a\partial^{\alpha}{\rm div}\, b\mathrm{d}x,
\end{align}
which clearly fulfills the inequality $|\mathcal{E}_0(t)|\lesssim  \|f\|_{L^2_v(H^3)}^2$. 
Putting  \eqref{GH3.26}--\eqref{GH3.23} together, we get the desired  \eqref{G3.40}.
\end{proof}

Then, we establish the uniform $H^3$ estimates of $u$ and $f$ as follows.

\begin{prop}\label{P3.6}
Under the assumption \eqref{G3.1} on $[0,T)$ with some time $T>0$, we have
\begin{align}\label{H3.43}
& \| u(t)\|_{H^3}^2+\|f(t)\|_{L_{v}^2(H^3)}^2 +\mu\int_0^t \|\nabla u(\tau)\|_{H^3}^2 {\rm d}\tau\nonumber\\
&\quad +\int_0^t\Big(\|(b-u)(\tau)\|_{H^3}^2
+\|\nabla(a,b)(\tau)\|_{H^2}^2+\|\nabla P(\tau)\|_{H^2}^2+\|\{\mathbf{I}-\mathbf{P}\}f(\tau)\|_{L^2_{v,\nu}(H^3)}^2
\Big){\rm d}\tau \nonumber\\
\,&\qquad \leq C\big(  \|(\varrho_0,u_0)\|_{H^3}^2+\|f_0\|_{L^2_{v}(H^3)}^2\big),
\end{align}
for any $t\in [0,T)$.
\end{prop}
\begin{proof}
We now define the temporal energy functional $\mathcal{E}(t)$ as follows:
\begin{align*} 
\mathcal{E}(t):=\,&\|u(t)\|_{H^3}^2+\|f(t)\|_{{L}_{v}^2(H^3)}^2
+\tau_{1}\mathcal{E}_0(t).
\end{align*}
Note that the dissipation term of $u$ can be derived from the coupling of $b-u$ and $f$:
\begin{align}\label{G3.480}
    \|\nabla u\|_{H^2}^2\leq \|b-u\|_{H^3}^2+\|\nabla b\|_{H^2}^2\leq \mathcal{D}(t).
\end{align}
Here and below in this section, we denote $\mathcal{D}(t)=\mathcal{D}(u,f,P)(t)$ as 
defined by \eqref{D}.  Making use of \eqref{G3.1}, \eqref{G3.480} and Lemmas \ref{L3.1}--\ref{L3.5} and choosing a uniform constant $0<\tau_1\ll 1$, we have 
\begin{align}\label{G3.45}
\mathcal{E}(t)\backsim \|u\|_{H^3}^2+\|f\|_{L^2_{v}(H^3)}^2.
\end{align}
and the Lyapunov inequality
\begin{align}\label{G3.46}
\frac{{\rm d}}{{\rm d}t}\mathcal{E}(t)+\lambda_4\big(\mathcal{D}(t)+\mu\|\nabla u(t)\|_{H^{3}}^{2}\big)\leq 0,    
\end{align}
for a positive constant $\lambda_4$ satisfying $0<\lambda_4\leq \frac{1}{2}\min\{\lambda_1,\lambda_2,\tau_1\lambda_3 \}$.
Thus, \eqref{G3.46} gives rise to 
\begin{align}\label{G3.47}
\mathcal{E}(t)+\lambda_4\int_0^t\mathcal{D}(s){\rm d}s\leq \mathcal{E}(0),
\end{align}
for any $0\leq t< T$. This, together with \eqref{G3.45}, yields \eqref{H3.43}.
\end{proof}

\subsection{Decay rates associated with the negative Besov evolution}

To address the variable $\varrho$, our idea is to establish suitable decay rates of $(u,f)$.  Motivated by \cite{GW-CPDE-2012}, using Lemma \ref{LA.5}, we have, for $l=0,1,2$,
\begin{align}\label{G3.48}
\|\nabla^l u\|_{L^2}\lesssim \|u\|_{\dot{B}^{-s}_{2,\infty}}^{\frac{1}{1+l+s}}\|\nabla^{l+1}u\|_{L^2}^{\frac{l+s}{1+l+s}},
\end{align}
and
\begin{align}\label{G3.49}
\|\nabla^l f\|_{L_{x,v}^2}\lesssim \|f\|_{L_v^2(\dot{B}^{-s}_{2,\infty})}^{\frac{1}{1+l+s}}\|\nabla^{l+1}f\|_{L_{x,v}^2}^{\frac{l+s}{1+l+s}},
\end{align}
for any $s>0$. Combining \eqref{G3.480}, \eqref{G3.48} and \eqref{G3.49}, we arrive at
\begin{align*} 
\mathcal{E}(t)\lesssim \mathcal{D}(t)^{\frac{s}{1+s}}\Big(\|u\|_{\dot{B}^{-s}_{2,\infty}}+ \|f\|_{L_v^2(\dot{B}^{-s}_{2,\infty})}\Big)^{\frac{2}{1+s}},
\end{align*}
which implies that  
\begin{align}\label{GG3.50}
\mathcal{E}(t)^{1+\frac{1}{s}}\Big(\|u\|_{\dot{B}^{-s}_{2,\infty}}+ \|f\|_{L_v^2(\dot{B}^{-s}_{2,\infty})}+\epsilon\Big)^{-\frac{2}{s}}\lesssim \mathcal{D}(t). 
\end{align}

Consequently, putting \eqref{GG3.50} into \eqref{G3.47}, we derive the following inequality:
\begin{align}\label{G3.50}
\frac{{\rm d}}{{\rm d}t}\mathcal{E}(t)+ \lambda_5 \Big(\epsilon+\|u\|_{\dot{B}^{-s}_{2,\infty}}+ \|f\|_{L_v^2(\dot{B}^{-s}_{2,\infty})}+\|u\|_{H^3}
+\|f\|_{L^2_v(H^3)}\Big)^{-\frac{2}{s}}\mathcal{E}(t)^{1+\frac{1}{s}}\leq 0,
\end{align}
for some uniform constant $\lambda_5 > 0$ and any $\epsilon > 0$.
Setting $\epsilon\rightarrow  0$ in \eqref{G3.50}, we have 
\begin{align}\label{G3.51}
 \mathcal{E}(t)\leq\,& C \bigg\{\mathcal{E}(0)^{-\frac{2}{s}}+\frac{\lambda_5 t}{s}{\Big(\|u\|_{\dot{B}^{-s}_{2,\infty}}+\|f\|_{L_v^2(\dot{B}^{-s}_{2,\infty})}+ \|u\|_{H^3}+\|f\|_{L^2_v(H^3)}\Big)^{-\frac{2}{s}}}                 \bigg\}^{-s} \nonumber\\  
 \leq\,& C \Big(\mathcal{E}(0)+\|u\|_{\dot{B}^{-s}_{2,\infty}}^2
 +\|f\|_{L_v^2(\dot{B}^{-s}_{2,\infty})}^2+\|u\|_{H^3}^2+\|f\|_{L^2_v(H^3)}^2\Big)(1+t)^{-s}\nonumber\\
  \leq\,& C \Big(\|u\|_{\dot{B}^{-s}_{2,\infty}}^2
  +\|f\|_{L_v^2(\dot{B}^{-s}_{2,\infty})}^2+\|u_0\|_{H^3}^2+\|f_0\|_{L^2_v(H^3)}^2\Big)(1+t)^{-s}.
\end{align}
Thus,  from \eqref{G3.51}, we can deduce that
\begin{align}\label{G3.54}
\|u\|_{H^3}+\|f\|_{L_v^2(H^3)}\leq  C \Big(\|u\|_{\dot{B}^{-s}_{2,\infty}}+\|f\|_{L_v^2(\dot{B}^{-s}_{2,\infty})}
+\|u_0\|_{H^3}+\|f_0\|_{L^2_v(H^3)}\Big)(1+t)^{-\frac{s}{2}},
\end{align}
where $1<s\leq\frac{3}{2}$. However, the decay rate presented in \eqref{G3.54} is not fast enough as $\frac{s}{2}<1$ even in the endpoint $s=\frac{3}{2}$ associated with the $L^1$ assumption.

To improve the decay rate and establish the $L^1$ time integrability for $\|\nabla u\|_{H^2}$ in \eqref{rhoH3}, we establish a higher-order Lyapunov inequality, which relies on the introduction of a new temporal energy functional $\mathcal{E}_1(t)$ 
and its corresponding dissipation rate $\mathcal{D}_1(t)$:
\begin{align*} 
\mathcal{E}_1(t):=\,&\|\nabla u(t)\|_{H^2}^2+\|\nabla f(t)\|_{{L}_{v}^2(H^2)}^2
+\tau_{2}\mathcal{E}_0^{\prime}(t),\nonumber\\
\mathcal{D}_1(t):=\,&\|\nabla(b-u)(t)\|_{H^2}^2+\|\nabla^2(a,b)(t)\|_{H^1}^2+{\mu}\|\nabla^2 u(t)\|_{H^{2}}^{2}\\
& +\|\nabla^2 P(t)\|_{H^1}^2+\|\nabla \{\mathbf{I}-\mathbf{P}\}f(t)\|_{L^2_{v,\nu}(H^2)}^2,
\end{align*}
where the temporal functional $\mathcal{E}_0^{\prime}(t)$ is given by
\begin{align*} 
\mathcal{E}_0^{\prime}(t):=\sum_{1\leq|\alpha|\leq 2}\sum_{i,j=1}^3\int_{\mathbb{R}^3}
\partial^\alpha(\partial_ib_j+\partial_jb_i)\partial^{\alpha}\Gamma_{ij}(\{\mathbf{I}-\mathbf{P}\}f)\mathrm{d}x
-\sum_{1\leq|\alpha|\leq 2}\int_{\mathbb{R}^3}\partial^\alpha a\partial^{\alpha}{\rm div}\, b\mathrm{d}x, 
\end{align*}
where the uniform constant $0<\tau_2\ll 1$  is small enough, and $\Gamma_{i,j}(\cdot)$ is given by \eqref{G3.35}.

\begin{lem}\label{L3.8}
Let $1<s\leq\frac{3}{2}$. Under the conditions of Proposition \ref{P3.1}, there exists a universal constant $\lambda_6>0$ such that
\begin{align}\label{G3.81}
 \frac{{\rm d}}{{\rm d}t}\mathcal{E}_1(t)+\lambda_6\mathcal{D}_1(t)\leq 0. 
\end{align}
Consequently, we have
\begin{align}\label{G3.82}
 \sqrt{\mathcal{E}_1(t)}&\sim \|\nabla u\|_{H^2}+\|\nabla f\|_{L_v^2(H^2)}\nonumber\\
 &\leq C\big(\|\nabla u_0\|_{H^2}+\|\nabla f_0\|_{L^2_v(H^2)}+\|u\|_{\dot{B}^{-s}_{2,\infty}}
 +\|f\|_{L_v^2(\dot{B}^{-s}_{2,\infty})}\big)(1+t)^{-\frac{s}{2}-\frac{1}{2}},
\end{align}
for any $t\in[0,T)$. Here $C>0$ is independent of $\mu$ and $T$.
\end{lem}

\begin{proof}
By a direct calculation similar to that in Lemma \ref{L3.5}, we can easily obtain
\begin{align}\label{G3.57}
\frac{\rm d}{{\rm d}t}\mathcal{E}^{\prime}_{0}(t)+\lambda_7\|\nabla^2(a,b)\|_{H^{1}}^{2}
\leq&\, C\big(\|\nabla\{\mathbf{I}-\mathbf{P}\}f\|_{L_{v}^{2}(H^{2})}^{2}+\|\nabla(b-u)\|_{H^{2}}^{2}\big), 
\end{align} 
Now we claim the following derivative estimate:
\begin{align}\label{G3.59}
&\frac{\rm d}{{\rm d}t}\sum_{1\leq|\alpha|\leq 3}\big(\|\partial^\alpha u\|_{L^2}^{2}+\|\partial^\alpha f\|_{L_{x,v}^2}^{2}\big)\nonumber\\
& \quad+\lambda_8\sum_{1\leq |\alpha|\leq3}\big(\|\partial^\alpha(b-u)\|_{L^2}^{2}
+{\mu}\|\nabla\partial^\alpha u\|_{L^2}^2+\|\partial^\alpha\{\mathbf{I}-\mathbf{P}\}f\|_{L^2_{v,\nu}(L^2)}^{2}\big)\nonumber\\
&\qquad \leq C\sigma\|\nabla^2(a,b)\|_{H^1}^2.
\end{align}
Here $\lambda_7,\lambda_8>0$ are uniform constants. 

Thus, in accordance with \eqref{G3.57} and \eqref{G3.59}, the Lyapunov inequality \eqref{G3.81} is shown. Arguing similarly to those in \eqref{G3.480}--\eqref{G3.54}, we infer from \eqref{G3.81} that \eqref{G3.82} holds. 

Now we justify the claim \eqref{G3.59}, which essentially relies on an elaborate analysis of the nonlinear terms  $J_3,...,J_9$ in \eqref{G3.25}. Indeed, we need to control $J_3,...,J_9$ by higher-order dissipation terms such that these can be absorbed by $\mathcal{D}_1(t)$. This can be achieved by taking advantage of the $\dot{B}^{-s}_{2,\infty}$ regularity of the solution.

Notice that 
$$
J_3=\frac{1}{2}\int_{\mathbb{R}^3}
\big\langle\partial^{\alpha}(u\cdot vf),\partial^{\alpha}f\big\rangle {\rm d}x+\frac{1}{2}\int_{\mathbb{R}^3}
\big\langle\partial^{\alpha}(\varrho u\cdot vf),\partial^{\alpha}f\big\rangle {\rm d}x=:J_{3_1}+J_{3_2}.
$$
We now analyze $J_{3_1}$ and $J_{3_2}$ separately. When $|\alpha|=2,3$, one has
\begin{align*}
|J_{3_1}|\leq\,& C\|\partial^\alpha(uf)\|_{L_{x,v}^2}\|v\partial^\alpha f\|_{L_{x,v}^2}\nonumber\\
\leq\, & C\big(\|\partial^\alpha u\|_{L^2}\|f\|_{H^3_{x,v}}+\|\nabla u\|_{L^6}\|f\|_{H^3_{x,v}}+\|\nabla^2 u\|_{L^2}\|f\|_{H^3_{x,v}}+\|\partial^\alpha f\|_{L^2_{x,v}}\|u\|_{H^3}\big)\nonumber\\
\,&\times \big(\|\partial^\alpha\{\mathbf{I}-\mathbf{P}\}f\|_{L^2_{v,\nu}(L^2)}+\|\nabla^2 (a,b)\|_{H^1}\big)\nonumber\\
\leq\,& C\big(\|u\|_{H^3}+\|f\|_{L^2_v(H^3)}\big)\big(\|\nabla^2 u\|_{H^1}^2+\|\nabla^2 (a,b)\|_{H^1}^2+\|\nabla\{\mathbf{I}-\mathbf{P}\}f\|_{L^2_{v,\nu}(H^2)}^2 \big).
\end{align*}
The lower-order case $|\alpha|=1$ requires an interpolation argument. It follows from Lemmas \ref{LA.1} and \ref{LA.5} that
\begin{align*}
|J_{3_1}|\leq\,& C\big(\|\nabla u\|_{L^2}+\|\nabla f\|_{L_{x,v}^2}\big)\big(\|u\|_{L^3}+\|f\|_{L_v^2(L^3)}\big)\|\nabla(v\partial^\alpha f)\|_{L_v^2(L^6)}  \nonumber\\
\leq\,& C\Big(\|\nabla u\|_{L^2}^{\frac{3}{2}}+\|\nabla f\|_{L_{x,v}^2}^{\frac{3}{2}}\Big)\Big(\| u\|_{L^2}^{\frac{1}{2}}+\| f\|_{L_{x,v}^2}^{\frac{1}{2}}\Big)\big(\|\nabla^2 (a,b)\|_{H^1}+\|\nabla\{\mathbf{I}-\mathbf{P}\}f\|_{L^2_{v,\nu}(H^2)}   \big)\nonumber\\
\leq\,& C\Big(\|\nabla^2 u\|_{L^2}^{\frac{3}{2}\times\frac{1+s}{2+s}}+\|\nabla^2 f\|_{L_{x,v}^2}^{\frac{3}{2}\times\frac{1+s}{2+s}}\Big)
\Big(\| u\|_{\dot{B}^{-s}_{2,\infty}}^{\frac{3}{2}\times\frac{1}{2+s}}+\|\nabla^2 f\|_{L_{v}^2(\dot{B}^{-s}_{2,\infty})}^{\frac{3}{2}\times\frac{1}{2+s}}\Big)\Big(\| u\|_{L^2}^{\frac{1}{2}}+\| f\|_{L_{x,v}^2}^{\frac{1}{2}}\Big)\nonumber\\
\,&\times \big(\|\nabla^2 (a,b)\|_{H^1}+\|\nabla\{\mathbf{I}-\mathbf{P}\}f\|_{L^2_{v,\nu}(H^2)}             \big)\nonumber\\ 
\leq\,& C\Big(\|u\|_{H^3\cap\dot{B}^{-s}_{2,\infty}}+\|f\|_{L^2_v(H^3\cap\dot{B}^{-s}_{2,\infty})}\Big)
\big(\|\nabla^2 (a,b)\|_{H^1}^2+\|\nabla\{\mathbf{I}-\mathbf{P}\}f\|_{L^2_{v,\nu}(H^2)}^2+\|\nabla^2 u\|_{H^1}^2              \big).
\end{align*}
From the above estimates of $J_{3_1}$ and \eqref{G3.1}, for $|\alpha|=1,2,3$, it yields 
\begin{align*}
|J_{3_1}|\leq C    \sigma\big(\|\nabla^2 u\|_{H^1}^2+\|\nabla^2 (a,b)\|_{H^1}^2+\|\nabla\{\mathbf{I}-\mathbf{P}\}f\|_{L^2_{v,\nu}(H^2)}^2   \big).
\end{align*}
Now we turn to $J_{3_2}$. For $|\alpha|=1,2$, as $|v\partial^\alpha f|\lesssim |\partial^\alpha a|+|\partial^\alpha b|+\|\partial^\alpha\{\mathbf{I-P}\}f\|_{L^2_{v,\nu}}$, it holds
\begin{align*}
|J_{3_2}|\leq\, & C\|\varrho\|_{H^3}\big(\|u\|_{L^6}\|f\|_{L_v^2(L^6)}\|v\partial^\alpha f\|_{L_v^2(L^6)}+\|\nabla u\|_{L^6}\|f\|_{L_v^2(L^6)}\|v\partial^\alpha f\|_{L_v^2(L^6)}\big) \nonumber\\
&+C\|\varrho\|_{H^3}\big(\|u\|_{L_v^2(L^6)}\|\nabla f\|_{L_v^2(L^6)}\|v\partial^\alpha f\|_{L_v^2(L^6)}+\|\nabla u\|_{H^1}\|\nabla^2f\|_{L_v^2(H^1)}\|v\partial^\alpha f\|_{L_v^2(L^6)}\big)\nonumber\\
&+C\|\varrho\|_{H^3}\|\nabla^2 u\|_{H^1}\|\nabla f\|_{L_v^2(H^1)}\|v\partial^\alpha f\|_{L_v^2(L^6)}\nonumber\\
\leq\,& C\sigma \big( \|\nabla u\|_{L^2}^2+\|\nabla f\|_{L_{x,v}^2}^2   \big)\|\nabla(v\partial^\alpha f)\|_{L_{x,v}^2}+C\sigma\|\nabla^2 (a,b)\|_{H^1}^2             \nonumber\\
&+C\sigma\|\nabla^2 u\|_{H^1}^2+C\sigma\|\nabla\{\mathbf{I}-\mathbf{P}\}f\|_{L^2_{v,\nu}(H^2)}^2\nonumber\\
\leq\,& C\sigma\Big(\|\nabla^2 u\|_{L^2}^{2\times\frac{1+s}{2+s}}+\|\nabla^2 f\|_{L_{x,v}^2}^{2\times\frac{1+s}{2+s}}\Big)
\Big(\| u\|_{\dot{B}^{-s}_{2,\infty}}^{2\times\frac{1}{2+s}}+\|\nabla^2 f\|_{L_{v}^2(\dot{B}^{-s}_{2,\infty})}^{2\times\frac{1}{2+s}}\Big)\nonumber\\
&\times \big(\|\nabla^2 (a,b)\|_{H^1}+\|\nabla\{\mathbf{I}-\mathbf{P}\}f\|_{L^2_{v,\nu}(H^2)}             \big)\nonumber\\
&+C\sigma\big(\|\nabla^2 u\|_{H^1}^2+\|\nabla^2 (a,b)\|_{H^1}^2+\|\nabla\{\mathbf{I}-\mathbf{P}\}f\|_{L^2_{v,\nu}(H^2)}^2 \big)\nonumber\\
\leq\,& C\sigma\big(\|\nabla^2 u\|_{H^1}^2+\|\nabla^2 (a,b)\|_{H^1}^2+\|\nabla\{\mathbf{I}-\mathbf{P}\}f\|_{L^2_{v,\nu}(H^2)}^2 \big).
\end{align*}
Similarly, for $|\alpha|=3$, we  have
\begin{align*}
|J_{3_2}|\leq \,&C\big(\|\nabla^3\varrho\|_{L^2}\|u\|_{L^{\infty}}\|f\|_{L_v^2(L^{\infty})} +C\|\varrho\|_{H^3}\|u\|_{L^6}\|\nabla f\|_{L_v^2(L^{6})}\big)\|v\partial^\alpha f\|_{L_{x,v}^2}\nonumber\\
&+ C\|\varrho\|_{H^3}\big(\|\nabla u\|_{L^6}\|f\|_{L_v^2(L^{6})}\|v\partial^\alpha f\|_{L_{x,v}^2}+\|\nabla u\|_{H^1}\|\nabla^2f\|_{L_v^2(H^{1})}\|v\partial^\alpha f\|_{L_{x,v}^2}\big)  \nonumber\\
&+C\|\varrho\|_{H^3}\|\nabla^2 u\|_{H^1}\|\nabla f\|_{L_v^2(H^{1})}\|v\partial^\alpha f\|_{L_{x,v}^2}\nonumber\\
\leq\,& C\sigma\big(\|\nabla^2 u\|_{H^1}^2+\|\nabla^2 (a,b)\|_{H^1}^2+\|\nabla\{\mathbf{I}-\mathbf{P}\}f\|_{L^2_{v,\nu}(H^2)}^2 \big).
\end{align*}
By combining the above estimates for $J_{3_1}$ and  $J_{3_2}$, we arrive at
\begin{align}\label{G3.63}
|J_3|\leq C\sigma\big(\|\nabla^2 u\|_{H^1}^2+\|\nabla^2 (a,b)\|_{H^1}^2+\|\nabla\{\mathbf{I}-\mathbf{P}\}f\|_{L^2_{v,\nu}(H^2)}^2  \big).    
\end{align}
For the term  $J_4$, we split the proof into the cases  $|\alpha|=1$ and $|\alpha|=2,3$. As for $|\alpha|=1$, we employ the $\dot{B}^{-s}_{2,\infty}$ regularity to obtain
\begin{align*}
|J_4|\leq\,& C\|\nabla u\|_{L^3}\|\{\mathbf{I}-\mathbf{P}\}\nabla_vf\|_{L_v^2(L^6)}\|\nabla f\|_{L_{x,v}^2}+C\|\nabla u\|_{L^2}\|(a,b)\|_{L^3} \|\nabla f\|_{L_v^2(L^6)} \nonumber\\
\leq\,& C\|\nabla u\|_{L^2}^{\frac{1}{2}}\|\nabla^2 u\|_{L^2}^{\frac{1}{2}}\|\nabla\{\mathbf{I}-\mathbf{P}\}f\|_{L^2_{v,\nu}(L^2)}\|\nabla f\|_{L_{x,v}^2}\nonumber\\
&+C
\|(a,b)\|_{L^2}^{\frac{1}{2}}\|\nabla(a,b)\|_{L^2}^{\frac{1}{2}}\|\nabla u\|_{L^2}\|\nabla^2 f\|_{L_{x,v}^2}\nonumber\\
\leq\,& C\|\nabla^2 u\|_{L^2}^{\frac{1}{2}\times \frac{1+s}{2+s}}\|u\|_{\dot{B}_{2,\infty}^{-s}}^{\frac{1}{2}\times\frac{1}{2+s}}
\|\nabla\{\mathbf{I}-\mathbf{P}\}f\|_{L^2_{v,\nu}(L^2)}\|\nabla^2 f\|_{L_v^2(\dot{B}^{-s}_{2,\infty})}^{\frac{1+s}{2+s}}
\|f\|_{L_v^2(\dot{B}^{-s}_{2,\infty})}^{\frac{1}{2+s}}\nonumber\\
&+C\|(a,b)\|_{L^2}^{\frac{1}{2}}\|\nabla^2(a,b)\|_{L^2}^{\frac{1}{2}\times\frac{1+s}{2+s}}
\|(a,b)\|_{\dot{B}^{-s}_{2,\infty}}^{\frac{1}{2}\times\frac{1}{2+s}}\|\nabla^2 u\|_{L^2}^{ \frac{1+s}{2+s}}\|u\|_{\dot{B}_{2,\infty}^{-s}}^{\frac{1}{2+s}}\|\nabla^2 f\|_{L_{x,v}^2}\nonumber\\
\leq\,& C\sigma\big(\|\nabla^2 u\|_{H^1}^2+\|\nabla^2 (a,b)\|_{H^1}^2+\|\nabla\{\mathbf{I}-\mathbf{P}\}f\|_{L^2_{v,\nu}(H^2)}^2   \big).
\end{align*}
The higher-order  cases $|\alpha|=2,3$ can  easily addressed similar to that in  \eqref{G3.27}.
Consequently, for $|\alpha|=1,2,3$, we have
\begin{align}\label{G3.65}
|J_4|\leq  C\sigma\big(\|\nabla^2 u\|_{H^1}^2+\|\nabla^2 (a,b)\|_{H^1}^2+\|\nabla\{\mathbf{I}-\mathbf{P}\}f\|_{L^2_{v,\nu}(H^2)}^2   \big).    
\end{align}
For the term  $J_6$, similarly to \eqref{G3.31}, we obtain
\begin{align}\label{G3.66}
|J_6|\leq C\sigma\big( \|\nabla^2 u\|_{H^1}^2+\mu\|\nabla^2 u\|_{H^2}^2+\|\nabla^2(a,b)\|_{H^1}^2+\|\nabla (b-u)\|_{H^2}^2\big),\quad |\alpha|=2,3.
\end{align}
When $|\alpha|=1$, using the pressure estimates \eqref{G3.6} and \eqref{NJKG3.6}, one also has
\begin{align*}
|J_6|\leq\,& C\|\varrho\|_{L^3}\big(\|\nabla^2 P\|_{L^2}+\|\nabla P\|_{L^6}\big)
\big(\|\nabla u\|_{L^6}+\|\nabla^2 u\|_{L^2}\big) \nonumber\\
\leq\,& C\sigma\|\nabla^2 P\|_{L^2}\|\nabla^2 u\|_{L^2}\nonumber\\
\leq\,&C\sigma\big( \|\nabla^2 u\|_{H^1}^2+\mu \|\nabla^2 u\|_{H^2}^2+\|\nabla^2(a,b)\|_{H^1}^2+\|\nabla (b-u)\|_{H^1}^2\big).\quad |\alpha|=1.
\end{align*}
Similarly, we derive
\begin{align}\label{G3.680}
|J_7|\leq\,&C\|\varrho\|_{H^3}\|\nabla^2 P\|_{H^1}\|\nabla^2 u\|_{H^1}\nonumber\\
\leq\,&C\sigma\big(\|\nabla^2 u\|_{H^1}^2+\mu \|\nabla^2 u\|_{H^2}^2+\|\nabla^2 (a,b)\|_{H^1}^2+\|\nabla (b-u)\|_{H^2}^2  \big),\quad \quad |\alpha|=2,3,
\end{align}
and 
\begin{align}\label{G3.68}
|J_7|\leq \,& C\|\varrho\|_{L^2}\|\nabla^2 P\|_{L^2}\|\nabla u\|_{L^{\infty}}+C\|\varrho\|_{L^3}\|\nabla P\|_{L^6}\|\nabla^2 u\|_{L^2}\nonumber\\
\leq\,& C\sigma\|\nabla^2 P\|_{L^2}\|\nabla^2 u\|_{H^1}\nonumber\\
\leq\,& C\sigma\big(\|\nabla^2 u\|_{H^1}^2+\mu\|\nabla^2 u\|_{H^2}^2+\|\nabla^2 (a,b)\|_{H^1}^2+\|\nabla (b-u)\|_{H^1}^2     \big),\quad |\alpha|=1.
\end{align}
For the term $J_8$, we have
\begin{align}\label{G3.69}
|J_8|\leq\,& C\|\partial^\alpha u\|_{L^2}\big(\|\nabla^2(a,u)\|_{H^1}\|(a,u)\|_{L^{\infty}}+\|\nabla^2(a,u)\|_{L^2}\|(a,u)\|_{H^3}\big)   \nonumber\\
\leq\,& C\sigma\big(\|\nabla^2 u\|_{H^1}^2+\|\nabla^2(a,b)\|_{H^1}^2\big),\quad |\alpha|=2,3.
\end{align}
Also, it follows from the assumption \eqref{G3.1}, Lemma \ref{LA.1} and
Lemma \ref{LA.5} that
\begin{align}\label{G3.70}
|J_8|\leq\,& C\|(a,u)\|_{L^3}\|\nabla u\|_{L^6}\|\nabla(a,u)\|_{L^2}\nonumber\\
\leq\,& C\|(a,u)\|_{L^2}^{\frac{1}{2}}\|\nabla(a,u)\|_{L^2}^{\frac{3}{2}}
\|\nabla^2 u\|_{L^2}\nonumber\\
\leq\,& C\|(a,u)\|_{L^2}^{\frac{1}{2}}
\|\nabla^2(a,u)\|_{L^2}^{\frac{3}{2}\times\frac{1+s}{2+s}}\|\nabla(a,u)\|_{\dot{B}_{2,\infty}^{-s}}^{\frac{3}{2}
\times\frac{1}{2+s}}\|\nabla^2 u\|_{L^2}\nonumber\\
\leq\,& C\sigma\big(\|\nabla^2 u\|_{H^1}^2+\|\nabla^2(a,b)\|_{H^1}^2\big),\quad |\alpha|=1.
\end{align}
For the term $J_9$, the commutator estimate in Lemma \ref{LA.2} as well as Lemmas \ref{LA.1} and \ref{LA.3} 
implies that 
\begin{align}
|J_9|&\leq C\|\partial^\alpha u\|_{L^2}\|\nabla u\|_{L^{\infty}}\|\partial^\alpha u\|_{L^2}\nonumber\\
&\leq C\|\nabla u\|_{H^2}^2\|\nabla^2u\|_{H^1}\nonumber\\
&\leq C\|u\|_{H^3} \|\nabla^2 u\|_{H^1}^2 \leq C\sigma \|\nabla u\|_{H^2}^2.\label{J9}
\end{align}
Finally, we address the remaining term $J_5$, which constitutes the most critical step in this process.
Now, we shall decompose  $J_5$ into three distinct terms.
\begin{align*}
J_5=\,&\int_{\mathbb{R}^3}\langle\partial^\alpha(\varrho uf),\nabla_v\partial^\alpha f\rangle\mathrm{d}x+\int_{\mathbb{R}^3}\langle\partial^\alpha(\varrho \mathcal{L}f),\partial^\alpha f\rangle\mathrm{d}x +\int_{\mathbb{R}^3}\partial^\alpha(\varrho u)\cdot\partial^\alpha b\mathrm{d}x  \nonumber\\
=\,&\int_{\mathbb{R}^3}\langle\partial^\alpha(\varrho uf),\nabla_v\partial^\alpha f\rangle\mathrm{d}x+\int_{\mathbb{R}^3}\partial^\alpha\big(\varrho( u-b)\big)\cdot\partial^\alpha b\mathrm{d}x +\int_{\mathbb{R}^3}\langle\partial^\alpha(\varrho \mathcal{L}\{\mathbf{I}-\mathbf{P}\}f),\partial^\alpha f\rangle\mathrm{d}x  \nonumber\\
=:\,&J_{5_1}+J_{5_2}+J_{5_3}.
\end{align*}
Analogous to the derivation of $J_3$, we can similarly deduce that
\begin{align*} 
|J_{5_1}|\leq  C\sigma\big(\|\nabla^2 u\|_{H^1}^2+\|\nabla^2 (a,b)\|_{H^1}^2+\|\nabla\{\mathbf{I}-\mathbf{P}\}f\|_{L^2_{v,\nu}(H^2)}^2  \big).  
\end{align*}
When $|\alpha|=2,3$, we have
\begin{align*}
|J_{5_2}|\leq C\|\varrho\|_{H^3}\|\nabla (u-b)\|_{H^2}\|\nabla^2 b\|_{H^1}  
\leq C\sigma\Big(\|\nabla^2 (a,b)\|_{H^1}^2+\sum_{1\leq|\alpha|\leq 3}\|\partial^\alpha(b-u)\|_{L^2}^2\Big).
\end{align*}
While for  $|\alpha|=1$, applying integration by parts, we get
\begin{align*} 
|J_{5_2}|\leq\,& C\|\varrho\|_{H^3}(\|u-b\|_{L^6}\|\nabla^2 b\|_{L^2}+\|\nabla(u-b)\|_{L^2}\|\nabla b\|_{L^6})\nonumber\\
\leq\,& C\sigma \Big(\|\nabla^2 (a,b)\|_{H^1}^2+\sum_{1\leq|\alpha|\leq 3}\|\partial^\alpha(b-u)\|_{L^2}^2\Big).
\end{align*}
Therefore, it is evident that
\begin{align*}
|J_{5_2}|  \leq C\sigma\Big(\|\nabla^2 (a,b)\|_{H^1}^2+\sum_{1\leq|\alpha|\leq 3}\|\partial^\alpha(b-u)\|_{L^2}^2\Big).  
\end{align*}
For the term $J_{5_3}$, we shall utilize the   macro-micro decomposition
\eqref{mmd} and  the self-adjoint property of $\mathcal{L}$. First, 
\begin{align*}
J_{5_3}&=-\int_{\mathbb{R}^3}   \partial^\alpha(\varrho\{\mathbf{I}-\mathbf{P}  \}f) \cdot\partial^\alpha b {\rm d}x+\int_{\mathbb{R}^3}\langle\partial^\alpha(\varrho\{\mathbf{I}-\mathbf{P}  \}f),\partial^\alpha\mathcal{L}\{\mathbf{I}-\mathbf{P}\}f\rangle\mathrm{d}x\nonumber\\
&=:J_{5_{31}}+J_{5_{32}}.
\end{align*}
When $|\alpha|=2,3$, we have
\begin{align*}
|J_{5_{31}}|\leq C\|\varrho\|_{H^3}\|\nabla\{\mathbf{I}-\mathbf{P}\}f\|_{L_v^2(H^2)}\|\partial^\alpha b\|_{L^2}\leq C\sigma\big(\|\nabla^2 (a,b)\|_{H^1}^2+\|\nabla\{\mathbf{I}-\mathbf{P}\}f\|_{L^2_{v,\nu}(H^2)}^2              \big).
\end{align*}
When $|\alpha|=1$, by integration by parts, the fact $\dot{B}^{-s}_{2,\infty}\cap H^3\subset\dot{H}^{-1}$ for $1<s\leq\frac{3}{2}$ and Lemma \ref{LA.7} gives
\begin{align*}
|J_{5_{31}}|\leq\,& C\|\varrho\|_{\dot{H}^{-1}}\|\nabla\{\mathbf{I}-\mathbf{P}\}f\cdot\nabla b\|_{L_v^2(\dot{H}^1)}\nonumber\\
\leq\, &C\Big(\|\varrho\|_{\dot{B}^{-s}_{2,\infty}}+\|\varrho\|_{H^{3}}\Big)
\big(\|\nabla^2\{\mathbf{I}-\mathbf{P}\}f\|_{L_v^2(L^3)}\|\nabla b\|_{L^6}+\|\nabla\{\mathbf{I}-\mathbf{P}\}f\|_{L_v^2(L^6)}\|\nabla^2 b\|_{L^3}\big)\nonumber\\
\leq\,&C\sigma\big(\|\nabla^2 (a,b)\|_{H^1}^2+\|\nabla\{\mathbf{I}-\mathbf{P}\}f\|_{L^2_{v,\nu}(H^2)}^2              \big).
\end{align*}
For the term $J_{5_{32}}$, the case $|\alpha|=2,3$ can be easily addressed by 
\begin{align*}
|J_{5_{32}}|\leq C\|\varrho\|_{H^3}\|\partial^\alpha\{\mathbf{I}-\mathbf{P}  \}f\|_{L_v^2(L^{\infty})}\|\partial^\alpha\{\mathbf{I}-\mathbf{P}  \}f\|_{L_{x,v}^2}\leq C\sigma   \|\nabla\{\mathbf{I}-\mathbf{P}\}f\|_{L^2_{v,\nu}(H^2)}^2.  
\end{align*}
When $|\alpha|=1$, similar to the case $|\alpha|=1$ for $J_{5_{31}}$, from integration by parts, interpolation inequality and Lemma \ref{LA.7}, one gets
\begin{align*}
|J_{5_{32}}|\leq\,& C  \|\varrho\|_{\dot{H}^{-1}} 
\|\nabla\{\mathbf{I}-\mathbf{P}\}f\|_{L^2_{v,\nu}(H^2)}^2\nonumber\\
\leq\,& C\Big(\|\varrho\|_{\dot{B}^{-s}_{2,\infty}}+\|\varrho\|_{H^{3}}\Big)
\|\nabla\{\mathbf{I}-\mathbf{P}\}f\|_{L^2_{v,\nu}(H^2)}^2\nonumber\\
\leq\,& C\sigma\|\nabla\{\mathbf{I}-\mathbf{P}\}f\|_{L^2_{v,\nu}(H^2)}^2.
\end{align*}
Therefore, we end up
with
\begin{align}\label{G3.80}
|J_5|\leq C\sigma\big(\|\nabla^2 u\|_{H^1}^2+\|\nabla^2 (a,b)\|_{H^1}^2+\|\nabla\{\mathbf{I}-\mathbf{P}\}f\|_{L^2_{v,\nu}(H^2)}^2 \big).
\end{align}
Owing to \eqref{G3.25}--\eqref{G3.27},  \eqref{G3.63}--\eqref{J9}, and \eqref{G3.80}, we prove the claim \eqref{G3.59} and finish the proof of Lemma \ref{L3.8}.
\end{proof}



\subsection{Evolution of negative Besov norms}

In this subsection, we analyze the uniform evolution of the negative Besov norms associated with  $(u,f)$ required in \eqref{G3.54} and \eqref{G3.82}, based on the Littlewood-Paley theorem. 
\begin{prop}\label{P3.9}
Under the conditions of Proposition \ref{P3.1}, we have 
\begin{align}\label{G3.85}
&  
\|u(t)\|_{\dot{B}^{-s}_{2,\infty}}^2+\|f(t)\|_{L_v^2(\dot{B}^{-s}_{2,\infty})}^2 \nonumber\\
&\quad +\int_0^t \Big(\|(b-u)(\tau)\|_{\dot{B}^{-s}_{2,\infty}}^2+\mu\|\nabla u(\tau)\|_{\dot{B}^{-s}_{2,\infty}}^2
+\|\{\mathbf{I}-\mathbf{P}\}f(\tau)\|_{L^2_{v,\nu}({\dot{B}_{2,\infty}^{-s}})}^{2}\Big){\rm d}\tau\nonumber\\
\,&\qquad \leq C\Big(  \|u_0\|_{H^3\cap\dot{B}^{-s}_{2,\infty}}^2+ \|f_0\|_{L_{v}^2(H^3\cap\dot{B}^{-s}_{2,\infty})}^2\Big),
\end{align}
for any $t\in[0,T)$.
\end{prop}

\begin{proof}
Applying $\dot{\Delta}_j$ to \eqref{rNSVFP}$_2$--\eqref{rNSVFP}$_4$ yields
\begin{equation}\label{3.474}
\left\{
\begin{aligned}
&\partial_{t} \dot{\Delta}_j 
 u +\nabla \dot{\Delta}_j  P +\dot{\Delta}_j (u -b^{f })-\mu\Delta\dot\Delta_j u\\\
 &\quad =\dot{\Delta}_j\Big(-u \cdot\nabla u -\frac{\mu \varrho  }{1+\varrho }\Delta u +\frac{\varrho }{1+\varrho }\nabla P {-au }\Big),\\
&{\rm div}\, \dot{\Delta}_j  u =0,\\
&\partial_{t}\dot{\Delta}_j f +v\cdot\nabla_x \dot{\Delta}_j f -\dot{\Delta}_j u \cdot v\sqrt{M}-\mathcal{L} \dot{\Delta}_jf \\
&\quad =\dot{\Delta}_j\Big({\frac{1}{2}}u\cdot vf -u \cdot\nabla_{v}f +\varrho \Big(\mathcal{L}f -u \cdot\nabla_v f +\frac{1}{2}u \cdot vf +u \cdot v\sqrt{M}\Big)\Big).
\end{aligned}
\right.
\end{equation}
Taking the inner products of $\eqref{3.474}_1$ and $\eqref{3.474}_3$ by $\dot{\Delta}_j u$ and $\dot{\Delta}_j f$, respectively, we have
\begin{align}\label{G3.86}
&\frac{1}{2}\frac{{\rm d}}{{\rm d}t}\big(\|\dot{\Delta}_j u\|_{L^2}^2+\|\dot{\Delta}_j f\|_{L^2_{x,v}}^2   \big)+\lambda_0\|\dot{\Delta}_j\{\mathbf{I}-\mathbf{P}\}f\|_{L^2_{v,\nu}(L^2)}^2+\|\dot{\Delta}_j(u-b)\|_{L^2}^2
+\mu\|\nabla\dot{\Delta}_j u\|_{L^2}^2\nonumber\\
\,&\quad \leq-\mu\int_{\mathbb{R}^3}\dot{\Delta}_j\Big(\frac{\varrho}{1+\varrho}\Delta u\Big)\cdot\dot{\Delta}_ju\mathrm{d}x-\int_{\mathbb{R}^3}\dot{\Delta}_j(u\cdot\nabla u)\cdot\dot{\Delta}_ju\mathrm{d}x-\int_{\mathbb{R}^3}\dot{\Delta}_j\Big(\frac{\varrho}{1+\varrho}\nabla P\Big)\cdot\dot{\Delta}_ju\mathrm{d}x\nonumber\\
&\qquad -\int_{\mathbb{R}^3}\dot{\Delta}_j(au)\cdot\dot{\Delta}_j(u-b)\mathrm{d}x
+\int_{\mathbb{R}^3}\big\langle\dot{\Delta}_j\big((1+\varrho)u\cdot v b\cdot v\sqrt{M}\big),\dot{\Delta}_j\{\mathbf{I}-\mathbf{P}\}f\big\rangle\mathrm{d}x
\nonumber\\
&\qquad +\int_{\mathbb{R}^3}\dot{\Delta}_j(\varrho
u a)\cdot\dot{\Delta}_jb\mathrm{d}x
+\int_{\mathbb{R}^3}\Big\langle\dot{\Delta}_j\Big((1+\varrho)\Big(\frac{1}{2}u\cdot v\{\mathbf{I}-\mathbf{P}\}f
-u\cdot\nabla_v\{\mathbf{I}-\mathbf{P}\}f\Big)\Big),\dot{\Delta}_j f
\Big\rangle\mathrm{d}x
\nonumber\\
&\qquad +\int_{\mathbb{R}^3}\dot{\Delta}_j (\varrho(u-b
) )\cdot\dot{\Delta}_jb\mathrm{d}x
+\int_{\mathbb{R}^3}\langle\dot{\Delta}_j(\varrho\{\mathbf{I}-\mathbf{P}\}f),\dot{\Delta}_j
\mathcal{L}\{\mathbf{I}-\mathbf{P}\}f\rangle\mathrm{d}x,
\end{align}
where we have used the key facts that
\begin{align*}
& -\int_{\mathbb{R}^3}\dot{\Delta}_j(au)\cdot\dot{\Delta}_ju\mathrm{d}x
+\int_{\mathbb{R}^3}\Big\langle\dot{\Delta}_j\Big((1+\varrho)\Big({\frac{1}{2}}u\cdot vf -u \cdot\nabla_{v}f\Big)\Big), \dot{\Delta}_jf\Big\rangle {\rm d}x\\
&\quad =-\int_{\mathbb{R}^3}\dot{\Delta}_j(au)\cdot\dot{\Delta}_j(u-b)\mathrm{d}x
+\int_{\mathbb{R}^3}\big\langle\dot{\Delta}_j\big(u\cdot v b\cdot v\sqrt{M}\big),\dot{\Delta}_j\{\mathbf{I}-\mathbf{P}\}f\big\rangle\mathrm{d}x\\
&\qquad+\int_{\mathbb{R}^3}\Big\langle\dot{\Delta}_j\Big((1+\varrho)\Big(\frac{1}{2}u\cdot v\{\mathbf{I}-\mathbf{P}\}f
-u\cdot\nabla_v\{\mathbf{I}-\mathbf{P}\}f\Big)\Big),\dot{\Delta}_j f
\Big\rangle\mathrm{d}x,
\end{align*}
and
\begin{align*}
&\int_{\mathbb{R}^3}\langle\dot{\Delta}_j (\varrho\mathcal{L}f) \cdot\dot{\Delta}_jf\rangle\mathrm{d}x+\int_{\mathbb{R}^3}\langle\dot{\Delta}_j (\varrho u\cdot v\sqrt{M}) \cdot\dot{\Delta}_jf\rangle\mathrm{d}x\nonumber\\
=\,& \int_{\mathbb{R}^3}\langle\dot{\Delta}_j(\varrho\{\mathbf{I}-\mathbf{P}\}f),\dot{\Delta}_j\mathcal{L} f\rangle\mathrm{d}x+\int_{\mathbb{R}^3}\langle\dot{\Delta}_j(\varrho \mathcal{L}\mathbf{P} f),\dot{\Delta}_j f\rangle\mathrm{d}x+\int_{\mathbb{R}^3}\dot{\Delta}_j (\varrho u)\cdot\dot{\Delta}_jb\mathrm{d}x   \nonumber\\
=\,& \int_{\mathbb{R}^3}\langle\dot{\Delta}_j(\varrho\{\mathbf{I}-\mathbf{P}\}f),\dot{\Delta}_j
\mathcal{L}\{\mathbf{I}-\mathbf{P}\} f\rangle\mathrm{d}x+\int_{\mathbb{R}^3}\langle
\dot{\Delta}_j(\varrho\{\mathbf{I}-\mathbf{P}\}f),\dot{\Delta}_j\mathcal{L} \mathbf{P} f\rangle\mathrm{d}x\nonumber\\
&+\int_{\mathbb{R}^3}\dot{\Delta}_j\big(\varrho(u-b
)\big)\cdot\dot{\Delta}_jb\mathrm{d}x\nonumber\\
=\,&\int_{\mathbb{R}^3}\dot{\Delta}_j\big(\varrho(u-b
)\big)\cdot\dot{\Delta}_jb\mathrm{d}x+\int_{\mathbb{R}^3}
\langle\dot{\Delta}_j(\varrho\{\mathbf{I}-\mathbf{P}\}f),\dot{\Delta}_j
\mathcal{L}\{\mathbf{I}-\mathbf{P}\}f\rangle\mathrm{d}x.
\end{align*}
Here, we have used the basic fact that $\mathcal{L}\{\mathbf{I}-\mathbf{P}\}f=\{\mathbf{I}-\mathbf{P}\}\mathcal{L}f$ (see  \cite[p. 122]{DF-jmp-2010}).
We then multiply \eqref{G3.86}  by $2^{-2sj}$ and take the supremum over $j\in\mathbb{Z}$ 
to obtain
\begin{align}\label{G3.87}
& \|u(t)\|_{\dot{B}_{2,\infty}^{-s}}^{2}+\|f(t)\|_{L_{v}^2(\dot{B}_{2,\infty}^{-s})}^{2} +\int_0^t \Big(\|(b-u)(\tau)\|_{\dot{B}^{-s}_{2,\infty}}^2+ \mu\|\nabla u(\tau)\|_{\dot{B}^{-s}_{2,\infty}}^2\Big){\rm d}\tau\nonumber\\
&+\int_0^t  \|\{\mathbf{I}-\mathbf{P}\}f(\tau)\|_{L^2_{v,\nu}({\dot{B}_{2,\infty}^{-s}})}^{2} {\rm d}\tau\nonumber\\ 
\lesssim\,&   \int_0^t\Big\|\frac{\varrho}{1+\varrho}\Delta u\Big\|_{\dot{B}_{2,\infty}^{-s}}\|u\|_{\dot{B}_{2,\infty}^{-s}}{\rm d}\tau+\int_0^t\|u\cdot\nabla u\|_{\dot{B}_{2,\infty}^{-s}}\|u\|_{\dot{B}_{2,\infty}^{-s}}{\rm d}\tau\nonumber\\
&+\int_0^t\Big\|\frac{\varrho}{1+\varrho}\nabla P\Big\|_{\dot{B}_{2,\infty}^{-s}}\|u\|_{\dot{B}_{2,\infty}^{-s}}{\rm d}\tau+\int_0^t\|au\|_{\dot{B}_{2,\infty}^{-s}}\|u-b\|_{\dot{B}_{2,\infty}^{-s}}{\rm d}\tau\nonumber\\
&+\int_0^t\|(\varrho+1)u\cdot b\|_{\dot{B}_{2,\infty}^{-s}}\|\{\mathbf{I}-\mathbf{P}\}f\|_{L^2_{v,\nu}({\dot{B}_{2,\infty}^{-s}})}{\rm d}\tau\nonumber\\
&+\int_0^t {\Big \| \|\varrho \{\mathbf{I}-\mathbf{P}\}f \|_{\dot{B}_{2,\infty}^{-s}} \Big\|_{L^2_{v,\nu} } \|\{\mathbf{I}-\mathbf{P}\}f\|_{L^2_{v,\nu}({\dot{B}_{2,\infty}^{-s}})}{\rm d}\tau}\nonumber\\
&+ {\int_0^t \Big\| \|(\varrho+1)u \{\mathbf{I}-\mathbf{P}\}f\|_{\dot{B}_{2,\infty}^{-s}} \Big\|_{L^2_{v,\nu} }\|f\|_{L_v^2({\dot{B}_{2,\infty}^{-s}})}{\rm d}\tau}\nonumber\\
&+\int_0^t\Big(\|\varrho(u-b)\|_{\dot{B}_{2,\infty}^{-s}}+\|\varrho au\|_{\dot{B}_{2,\infty}^{-s}}\Big) \|b\|_{\dot{B}^{-s}_{2,\infty}} {\rm d}\tau+\|u_0\|_{\dot{B}_{2,\infty}^{-s}}^{2}+\|f_0\|_{L_{v}^2(\dot{B}_{2,\infty}^{-s})}^{2} \nonumber\\
=:\,&\sum_{j=1}^8 K_j+\|u_0\|_{\dot{B}_{2,\infty}^{-s}}^{2}+\|f_0\|_{L_{v}^2(\dot{B}_{2,\infty}^{-s})}^{2}.
\end{align}
To address the first term $K_1$, we require the restriction $\mu<1$. Note that $\frac{1
}{2}+\frac{s}{3}\leq 1$ and $2\leq\frac{3}{s}<6$  due to $1<s\leq \frac{3}{2}$. Thus, using Lemma \ref{LA.4} and the assumption \eqref{G3.1}, we have 
\begin{align}\label{G3.88}
K_1\leq\,& C\mu\|\varrho\|_{L_t^{\infty}(L^\frac{3}{s})}\|\nabla^2 u\|_{L_t^1(L^2)}\|u\|_{L_t^{\infty}(\dot{B}_{2,\infty}^{-s})}\nonumber\\
\leq\,& C\|\varrho\|_{L_t^{\infty}(H^3)}\|\nabla^2 u\|_{L_t^1(L^2)}\|u\|_{L_t^{\infty}(\dot{B}_{2,\infty}^{-s})}\nonumber\\
\leq\,& C\sigma \|\nabla^2 u\|_{L_t^1(L^2)}\|u\|_{L_t^{\infty}(\dot{B}_{2,\infty}^{-s})}.
\end{align}
A similar calculation  gives rise to
\begin{align}\label{G3.89}
K_2\leq\,& C\int_0^t\|u
\|_{L^\frac{3}{s}}\|\nabla u\|_{L^2}\|u\|_{\dot{B}_{2,\infty}^{-s}}{\rm d}\tau\nonumber\\
\leq\,& C\|u
\|_{L_t^{\infty}(H^3)}\|\nabla u\|_{L_t^1(L^2)}\|u\|_{L_t^{\infty}(\dot{B}_{2,\infty}^{-s})}\nonumber\\
\leq\,& C\sigma \|\nabla u\|_{L_t^1(L^2)}\|u\|_{L_t^{\infty}(\dot{B}_{2,\infty}^{-s})}.
\end{align}

To estimate the term $K_3$,  we use the estimates of $\nabla P$ in \eqref{G3.6}--\eqref{NJKG3.6}.
Therefore, it follows from the assumption \eqref{G3.1}  and Lemma \ref{LA.7} that
\begin{align}\label{G3.91}
K_3\leq\,& C\|\varrho\|_{L_t^{\infty}(\dot{B}^{-s+\frac{1}{2}}_{2,1})}\|\nabla P\|_{L_t^1(\dot{B}^{1}_{2,\infty})}\|u\|_{L_t^{\infty}(\dot{B}_{2,\infty}^{-s})}\nonumber\\
\leq \,& C\Big(\|\varrho\|_{L_t^{\infty}(\dot{B}_{2,\infty}^{-s})}
+\|\varrho\|_{L_t^{\infty}(\dot{B}_{2,\infty}^{-s+2})}\Big)\|\nabla^2 P\|_{L_t^1(\dot{B}_{2,2}^{0})}\|u\|_{L_t^{\infty}(\dot{B}_{2,\infty}^{-s})}\nonumber\\
\leq\,& C\Big(\|\varrho\|_{L_t^{\infty}(\dot{B}_{2,\infty}^{-s})}
+\|\varrho\|_{L_t^{\infty}(H^3)}\Big)\|\nabla^2 P\|_{L_t^1(L^2)}\|u\|_{L_t^{\infty}(\dot{B}_{2,\infty}^{-s})}\nonumber\\
\leq\,& C\sigma \big(\|\nabla u\|_{L_t^1(H^2)}+\|\nabla (a,b)\|_{L_t^1(H^2)}\big)\|u\|_{L_t^{\infty}(\dot{B}_{2,\infty}^{-s})}.
\end{align}
Note that $\|\cdot\|_{\dot{B}^{\frac{3}{2}}_{2,1}}\lesssim \|\nabla\cdot\|_{H^1}$ due to Lemma \ref{LA.7}. For the terms $K_4$, $K_5$ and $K_6$, one deduces from  \eqref{G3.1} and Lemma \ref{LA.7} that
\begin{align}\label{G3.92}
K_4\leq\,& C\|a\|_{L_t^{2}(\dot{B}^{\frac{3}{2}}_{2,1})}\|u\|_{L_t^{\infty}(\dot{B}_{2,\infty}^{-s})}
\|u-b\|_{L_t^{2}(\dot{B}_{2,\infty}^{-s})}\nonumber\\
\leq\,& C\|\nabla a\|_{L_t^{2}(H^1)}\|u-b\|_{L_t^{2}(\dot{B}_{2,\infty}^{-s})}\|u\|_{L_t^{\infty}(\dot{B}_{2,\infty}^{-s})}\nonumber\\
\leq\,& C\sigma \|\nabla a\|_{L_t^1(H^1)}\|u\|_{L_t^{\infty}(\dot{B}_{2,\infty}^{-s})}+C\sigma
\|u-b\|_{L_t^{2}(\dot{B}_{2,\infty}^{-s})}^2,\\\label{G3.93}
K_5\leq\,& C\|(\varrho+1)u\|_{L_t^{\infty}(\dot{B}_{2,\infty}^{-s})}
\|b\|_{L_t^{2}(\dot{B}_{2,1}^{\frac{3}{2}})}
\|\{\mathbf{I}-\mathbf{P}\}f\|_{L_t^{2}(L^2_{v,\nu}({\dot{B}_{2,\infty}^{-s})})}\nonumber\\
\leq\,& C\Big(1+\|\varrho\|_{L_t^{\infty}(\dot{B}_{2,1}^{\frac{3}{2}})}\Big)\|\nabla b\|_{L_t^2(H^1)}\|\{\mathbf{I}-\mathbf{P}\}f\|_{L_t^2(L^2_{v,\nu}
({\dot{B}_{2,\infty}^{-s}}))}\|u\|_{L_t^{\infty}(\dot{B}_{2,\infty}^{-s})}\nonumber\\
\leq\,& C\sigma \|\nabla b\|_{L_t^1(H^1)}\|u\|_{L_t^{\infty}(\dot{B}_{2,\infty}^{-s})}+C\sigma
\|\{\mathbf{I}-\mathbf{P}\}f\|_{L_t^2(L^2_{v,\nu}({\dot{B}_{2,\infty}^{-s}}))}^2,
\end{align}
and
\begin{align}\label{G3.94}
K_6\leq\,& C\|\varrho\|_{L_t^{\infty}(\dot{B}_{2,1}^{\frac{3}{2}})}
\|\{\mathbf{I}-\mathbf{P}\}f\|_{L_t^{2}(L^2_{v,\nu}({\dot{B}_{2,\infty}^{-s}}))}^2\leq C\sigma\|\{\mathbf{I}-\mathbf{P}\}f\|_{L_t^2(L^2_{v,\nu}({\dot{B}_{2,\infty}^{-s}}))}^2.
 \end{align}
Similarly, we have
\begin{align}\label{G3.95}
K_7\leq\,& C\|(\varrho+1)u\|_{L_t^{2}(\dot{B}_{2,1}^{\frac{3}{2}})}
\|\{\mathbf{I}-\mathbf{P}\}f\|_{L_t^2(L^2_{v,\nu}({\dot{B}_{2,\infty}^{-s}}))} \|f\|_{L_t^{\infty}(L_v^2({\dot{B}_{2,\infty}^{-s}}))}  \nonumber\\
\leq\,& C\big(1+\|\varrho\|_{L_t^{\infty}(H^3)}\big)\|\nabla u\|_{L_t^2(H^1)}\|\{\mathbf{I}-\mathbf{P}\}
f\|_{L_t^{2}(L^2_{v,\nu}({\dot{B}_{2,\infty}^{-s}}))}\|f\|_{L_t^{\infty}
(L_v^2({\dot{B}_{2,\infty}^{-s}}))}\nonumber\\
\leq\,& C\sigma \|\nabla u\|_{L_t^1(H^1)}\|f\|_{L_t^{\infty}(L_v^2({\dot{B}_{2,\infty}^{-s}}))}+C\sigma
\|\{\mathbf{I}-\mathbf{P}\}f\|_{L_t^{2}(L^2_{v,\nu}({\dot{B}_{2,\infty}^{-s}}))}^2,
\end{align}
and
\begin{align}
\label{G3.96}
K_8\leq\,& C\|\varrho\|_{L_t^{\infty}(\dot{B}_{2,\infty}^{-s})}\Big(\|u-b\|_{L_t^{1}
(\dot{B}^{\frac{3}{2}}_{2,1})}
+\|a\|_{L^2_t(\dot{B}^{\frac{3}{2}}_{2,1})}\|u\|_{L^2_t(\dot{B}^{\frac{3}{2}}_{2,1})}\Big) \|b\|_{L_t^{\infty}(\dot{B}_{2,\infty}^{-s})}\nonumber\\
\leq\,& C\|\varrho\|_{L_t^{\infty}(\dot{B}_{2,\infty}^{-s})}\big(\|\nabla (u,b)\|_{L_t^{1}(H^{1})}+\|\nabla a\|_{L^2_t(H^2)}\|\nabla u\|_{L_t^{2}(H^{1})}\big) \|f\|_{L_t^{\infty}(L^2_v(\dot{B}_{2,\infty}^{-s}))}\nonumber\\
\leq\,& C\sigma \|\nabla (a,b,u)\|_{L_t^{1}(H^{1})}
\|f\|_{L_t^{\infty}(L^2_v(\dot{B}_{2,\infty}^{-s}))}.
\end{align}
Substituting \eqref{G3.88}--\eqref{G3.89} and \eqref{G3.91}--\eqref{G3.96}
into \eqref{G3.87} and using  the smallness assumption   \eqref{G3.1}, we eventually obtain 
\begin{align*} 
&\|u\|_{\dot{B}^{-s}_{2,\infty}}^2+\|f\|_{L_v^2(\dot{B}^{-s}_{2,\infty})}^2
+\int_0^t \Big(\|(b-u)(\tau)\|_{\dot{B}^{-s}_{2,\infty}}^2
+\|\{\mathbf{I}-\mathbf{P}\}f(\tau)\|_{L^2_{v,\nu}({\dot{B}_{2,\infty}^{-s}})}^{2}\Big){\rm d}\tau\nonumber\\
&\quad  +\int_0^t \mu\|\nabla u(\tau)\|_{\dot{B}^{-s}_{2,\infty}}^2{\rm d}\tau\nonumber\\
\,&\qquad \leq \|u_0\|_{\dot{B}^{-s}_{2,\infty}}^2+\|f_0\|_{L_v^2(\dot{B}^{-s}_{2,\infty})}^2\nonumber\\
&\quad\qquad +C\sigma\sup_{\tau\in[0,t]}\Big(\|f(\tau)\|_{L_v^2(\dot{B}^{-s}_{2,\infty})}
+\|u(\tau)\|_{\dot{B}^{-s}_{2,\infty}}\Big) \int_0^t
\|\nabla (a,b,u)(\tau)\|_{H^2}{\rm d}\tau,
\end{align*}
which, together with the decay estimate \eqref{G3.82}, leads to  
\begin{align*} 
&\|u\|_{\dot{B}^{-s}_{2,\infty}}^2+\|f\|_{L_v^2(\dot{B}^{-s}_{2,\infty})}^2\nonumber\\
&\quad+\int_0^t \big(\|(b-u)(\tau)\|_{\dot{B}^{-s}_{2,\infty}}^2+\mu\|\nabla u(\tau)\|_{\dot{B}^{-s}_{2,\infty}}^2
+\|\{\mathbf{I}-\mathbf{P}\}f(\tau)\|_{L^2_{v,\nu}({\dot{B}_{2,\infty}^{-s}})}^{2}\big){\rm d}\tau\nonumber\\
\,&\qquad \leq C\sigma(\|u_0\|_{H^3}+\|f_0\|_{L_v^2({H^3})})
\sup_{0\leq\tau\leq t}\big(\|u(\tau)\|_{\dot{B}^{-s}_{2,\infty}}
+\|f(\tau)\|_{L_v^2(\dot{B}^{-s}_{2,\infty})}\big)\int_{0}^t(1+\tau)^{-\frac{s}{2}-\frac{1}{2}}{\rm d}\tau\nonumber\\
&\qquad \quad +C\sigma\sup_{0\leq\tau\leq t}\Big(\|u(\tau)\|_{\dot{B}^{-s}_{2,\infty}}
+\|f(\tau)\|_{L_v^2(\dot{B}^{-s}_{2,\infty})}\Big)^2\int_{0}^t(1+\tau)^{-\frac{s}{2}-\frac{1}{2}}{\rm d}\tau\nonumber\\
&\qquad\quad +C\Big(\|u_0\|_{\dot{B}^{-s}_{2,\infty}}^2+\|f_0\|_{L_v^2(\dot{B}^{-s}_{2,\infty})}^2\Big).
\end{align*}
Thanks to the condition $1<s\leq\frac{3}{2}$, this yields
\eqref{G3.85}.
\end{proof}

\smallskip
From \eqref{G3.85}, by combining \eqref{G3.54} and  \eqref{G3.82}, we   further obtain
\begin{align}
 \|u\|_{H^3}+\|f\|_{L_v^2(H^3)}\leq\,&  C \Big(\|u_0\|_{H^3\cap\dot{B}^{-s}_{2,\infty}}
 +\|f_0\|_{L_{v}^2(H^3\cap\dot{B}^{-s}_{2,\infty})}\Big)(1+t)^{-\frac{s}{2}},\label{G3.99-9}\\
 \|\nabla u\|_{H^2}+\|\nabla f\|_{L_v^2(H^2)}\leq\,&  C \Big(\|u_0\|_{H^3\cap\dot{B}^{-s}_{2,\infty}}
 +\|f_0\|_{L_{v}^2(H^3\cap\dot{B}^{-s}_{2,\infty})}\Big)(1+t)^{-\frac{s}{2}-\frac{1}{2}},\label{G3.100}
\end{align}
for all $t\geq 0$,
where $1<s\leq\frac{3}{2}$.

\smallskip

\subsection{Estimate of the non-dissipative component}

Finally, we give the estimate of $\varrho$ using the theorem of the transport equation and the decay estimates in \eqref{G3.100}.

\begin{prop}\label{P3.10}
Under the assumption \eqref{G3.1} on $[0,T)$ with some time $T>0$, we have
\begin{align}\label{G4.1}
 \|\varrho\|_{H^3\cap\dot{B}^{-s}_{2,\infty}}\leq C \|\varrho_0\|_{H^3\cap\dot{B}^{-s}_{2,\infty}},    
\end{align}
for any $0\leq t< T$.
\end{prop}
\begin{proof}
Applying $\partial^\alpha $ with $|\alpha|\leq 3$ to \eqref{rNSVFP}$_1$, 
multiplying the result by $\partial^\alpha \varrho$,   then integrating over
$\mathbb{R}^3$, and using the constraint ${\rm div}\, u=0$ and Lemma \ref{LA.2}, one has 
\begin{align}\label{G4.2}
 \frac{1}{2}\frac{{\rm d}}{{\rm d}t} \|\varrho\|_{H^3}^2=\,&-\sum_{|\alpha|\leq 3} \int_{\mathbb{R}^3}\partial^\alpha(u\cdot \nabla \varrho) \partial^\alpha\varrho{\rm d}x\nonumber\\
  =\,&-\sum_{|\alpha|\leq 3} \int_{\mathbb{R}^3}[\partial^\alpha, u\cdot \nabla] \varrho\partial^\alpha\varrho{\rm d}x\nonumber \\
  \leq\,&C\sum_{|\alpha|\leq 3}\|[\partial^\alpha, u\cdot \nabla] \varrho\|_{L^2}\|\partial^\alpha \rho\|_{L^2}    \nonumber\\
 \leq\,& C\sum_{|\alpha|\leq 3} \|\partial^\alpha\varrho\|_{L^2} \|\nabla \varrho\|_{H^2}\|\nabla u\|_{H^2}\nonumber\\
 \leq\,& C\|\varrho\|_{H^3}^2\|\nabla u\|_{H^2}.
\end{align} 
With the aid of Gronwall's inequality and \eqref{G3.100}, we obtain from \eqref{G4.2} that
\begin{align}\label{G4.3}
 \|\varrho\|_{H^3}\leq\,& C\exp\Big\{\int_0^t \|\nabla u(\tau)\|_{H^2}{\rm d}\tau\Big\}   \|\varrho_0\|_{H^3}\nonumber\\
 \leq\,&C\exp\Big\{C\big(\|u_0\|_{H^3\cap\dot{B}^{-s}_{2,\infty}}
 +\|f_0\|_{L_{v}^2(H^3\cap\dot{B}^{-s}_{2,\infty})}\big)\int_0^t (1+\tau)^{-\frac{s}{2}-\frac{1}{2}}{\rm d}\tau\Big\}\|\varrho_0\|_{H^3}\nonumber\\
 \leq\,& C\|\varrho_0\|_{H^3}.
\end{align}

Next, applying $\dot{\Delta}_j $ to \eqref{rNSVFP}$_1$, multiplying the result by $\dot{\Delta}_j \varrho$,  then integrating over $\mathbb{R}^3$ and using ${\rm div}\, u=0$, one reaches
\begin{align}\label{G4.4}
\frac{1}{2}\frac{{\rm d}}{{\rm d}t}\|\dot{\Delta}_j \varrho\|_{L^2}^2=\,&-\int_{\mathbb{R}^3}\dot{\Delta}_j \varrho   \dot{\Delta}_j (u\cdot\nabla \varrho){\rm d}x\nonumber\\
=\,&-\int_{\mathbb{R}^3}\dot{\Delta}_j \varrho [\dot{\Delta}_j,u\cdot\nabla]\varrho {\rm d}x\nonumber\\
\lesssim\,&\|\dot{\Delta}_j\varrho\|_{L^2}\|[\dot{\Delta}_j,u\cdot\nabla]\varrho\|_{L^2}.
\end{align}
 Multiplying \eqref{G4.4}  by $2^{-2sj}$, taking
the supremum over $j\in\mathbb{Z}$ and leveraging Lemma \ref{LA.8}, we infer that
\begin{align*} 
\frac{{\rm d}}{{\rm d}t}\|\varrho\|_{\dot{B}^{-s}_{2,\infty}}^2\leq C\|\varrho\|_{\dot{B}^{-s}_{2,\infty}}^2\|u\|_{\dot{B}^{\frac{5}{2}}_{2,\infty}}\leq C\|\varrho\|_{\dot{B}^{-s}_{2,\infty}}^2\|\nabla u\|_{H^2},  
\end{align*}
which leads to
\begin{align}\label{G4.6}
 \|\varrho\|_{\dot{B}^{-s}_{2,\infty}}\leq\,& C\exp\Big\{C\int_0^t \|\nabla u(\tau)\|_{H^2}{\rm d}\tau\Big\}   \|\varrho_0\|_{\dot{B}^{-s}_{2,\infty}}\nonumber\\
 \leq\,&C\exp\Big\{C\big(\|u_0\|_{H^3\cap\dot{B}^{-s}_{2,\infty}}
 +\|f_0\|_{L_{v}^2(H^3\cap\dot{B}^{-s}_{2,\infty})}\big)\int_0^t (1+\tau)^{-\frac{s}{2}-\frac{1}{2}}{\rm d}\tau\Big\}\|\varrho_0\|_{\dot{B}^{-s}_{2,\infty}}\nonumber\\
 \leq\,& C\|\varrho_0\|_{\dot{B}^{-s}_{2,\infty}}.
\end{align}
Hence, \eqref{G4.3} and \eqref{G4.6} give rise to \eqref{G4.1}.    
\end{proof}

\subsection{Proof of Theorem \ref{T1.1}}
In this section, we  shall integrate  the energy estimates, exploit decay rates
of solutions and negative Besov estimates that  established in Section 3 to
prove Theorem \ref{T1.1}. 
\begin{prop}\label{PP311}
For any $1<s\leq\frac{3}{2}$, there exists a constant $\eps_0^*>0$ and a time $T^{*}>0$ such that if 
\begin{align*}
  &\|(\varrho_0, u_0) \|_{H^3\cap\dot{B}_{2,\infty}^{-s}}^2
  +\|f_0\|_{L_v^2(H^3\cap\dot{B}_{2,\infty}^{-s})}^2\leq \eps_0^*, 
\end{align*}
 then 
there exists a unique solution $(\varrho,u,P,f)$ to the Cauchy problem \eqref{rNSVFP}--\eqref{NJKI-10} on $[0,T^{*})$ satisfying
\begin{align*}
& \|(\varrho,u)(t)\|_{H^3\cap \dot{B}_{2,\infty}^{-s}}^2+\|f(t)\|_{L_{v}^2(H^3\cap \dot{B}_{2,\infty}^{-s})}^2\\
&\quad+\int_0^t\big(\mu\|\nabla u\|_{H^3}^2+\|\nabla P\|_{H^2}^2+\|\{\mathbf{I}-\mathbf{P}\}f(\tau)\|_{L^2_{v,\nu}(H^3)}^2
\big){\rm d}\tau\\
&\qquad\leq 2\|(\varrho_0, u_0) \|_{H^3\cap\dot{B}_{2,\infty}^{-s}}^2+2\|f_0\|_{L_v^2(H^3\cap\dot{B}_{2,\infty}^{-s})}^2,
\end{align*}
for $t\in [0,T^{*})$. Moreover, $\|(\varrho,u)\|_{H^3\cap\dot{B}_{2,\infty}^{-s}}^2+\|f\|_{L_{v}^2(H^3\cap\dot{B}_{2,\infty}^{-s} )}^2$ is continuous
over $[0,T^{*})$.  
\end{prop}
The proof of Proposition \ref{PP311} follows by a standard iteration argument and local-in-time energy estimates 
obtained above (cf. \cite{Gy-IUMJ-2004,LMW-SIAM-2017}). For brevity, we omit the details here.

With the aid of Proposition \ref{PP311}, we continue to prove Theorem \ref{T1.1}.
Let $\eqref{small1}$ hold with $\eps_0\leq \eps_0^*$ to be chosen later.
Denote 
\begin{align*}
T^{**}=\sup\bigg\{ T\in [0,T^{*})\,\bigg|
 \begin{array}{l}
  \eqref{rNSVFP}\!-\!\eqref{NJKI-10}~\textrm{admits a solution}~ (\varrho,u,P,f) ~ \textrm{on} ~ [0,T) ~ \textrm{satisfying} \\
 \|f(t)\|_{L_{v}^2(H^3\cap\dot{B}^{-s}_{2,\infty})}
 +\|(\varrho,u)(t)\|_{H^3\cap\dot{B}^{-s}_{2,\infty})}\leq \sigma ~ \textrm{for} ~ t\in [0,T).  \\
 \end{array}
\bigg\}.
\end{align*}
Here, the generic constant $\sigma$ is given by Proposition \ref{P3.1}. According to Proposition \ref{PP311}, $T^{**}\in(0,T^*]$ is well-defined. It thus follows from the uniform estimates established in Proposition \ref{P3.1} that the uniform regularity estimate \eqref{apropri} holds true for any $t\in (0,T^{**})$.
Consequently, choosing 
$$
\eps_0:=\frac{1}{2}\min\bigg\{\eps_0^*, \frac{\sigma^2}{ C^*}\bigg\},
$$
 we arrive at
\begin{align*}
&\|f(t)\|_{L_{v}^2(H^3\cap\dot{B}^{-s}_{2,\infty}))}^2+\|(\varrho,u)(t)\|_{H^3\cap\dot{B}^{-s}_{2,\infty})}^2\\
&\quad \leq C^{*}\Big(\|(\rho_0, u_0) \|_{H^3\cap\dot{B}_{2,\infty}^{-s}}^2+\|f_0\|_{L_v^2(H^3\cap\dot{B}_{2,\infty}^{-s})}^2\Big)
\leq \frac{1}{2}\sigma^2,
\end{align*}
for any $t\in (0,T^{**})$. Thus, the standard continuity argument implies that  $T^{**}=T^*$.

We now claim that $T^{**}=T^{*}=+\infty$. Otherwise, assume $T^{*}<+\infty$. According to the local existence theory and the uniform estimate obtained above, the solution is in fact defined on at least some interval $[0,T^{*}+\eta]$ for some $\eta>0$. 
This contradicts the definition of $T^{*}$. Hence, $(\varrho,u,P,f)$ is indeed a global solution to the Cauchy problem \eqref{rNSVFP}--\eqref{NJKI-10}. In particular, Proposition \ref{P3.1} ensures that \eqref{uniform1} is fulfilled. 
From Lemma \ref{L3.1}, \eqref{G3.100} and $\mu<1$, we have
\begin{align*} 
\|\nabla P(t)\|_{L^2}\lesssim\,& (1+\mu)\|u(t)\|_{H^3}+\|f(t)\|_{L_v^2(H^3)}\lesssim  \eps_0^{\frac{1}{2}}(1+t)^{- \frac{s}{2}},  \nonumber\\
\|\nabla^2 P(t)\|_{L^2}\lesssim\,& (1+\mu)\|\nabla u(t)\|_{H^2}+\|\nabla f (t)\|_{L_v^2(H^2)}\lesssim \eps_0^{\frac{1}{2}}(1+t)^{- \frac{s+1}{2}},
\end{align*}
which, combined with \eqref{G3.99-9} and \eqref{G3.100} leads to  \eqref{1.13}. Finally, invoking the maximum principle (cf. \cite{Gy-IUMJ-2004}), we have $F = M + \sqrt{M}f \geq 0$. This completes the proof of Theorem \ref{T1.1}.  \hfill $\Box$


\medskip 
\section{Justification of the Global Inviscid Limit}

In this section, we first establish the global existence of the Euler-VFP system \eqref{rEVFP} in Theorem \ref{T1.2} via the 
vanishing viscosity limit of the NS-VFP system \eqref{rNSVFP}. Then we show that 
this inviscid limit is global-in-time and prove Theorem \ref{T1.3}.

\subsection{Proof of Theorem \ref{T1.2}}

This section is devoted to the proof of Theorem \ref{T1.2} 
regarding the convergence rate of the vanishing viscosity limit. 

Let $(\varrho_0,u_0,f_0)$ satisfy \eqref{small2}. By density, one can choose a sequence $\{(\varrho_0^\mu,u_0^\mu,f_0^\mu)\}_{\mu\in(0,1)}$ such that $(\varrho_0^\mu,u_0^\mu)\rightarrow (\varrho_0,u_0)$ in $H^3\cap\dot{B}^{-s}_{2,\infty}$ and $f_0^\mu\rightarrow f_0$ in $L^2_v(H^3\cap\dot{B}^{-s}_{2,\infty})$. Therefore, recall $\eps_0>0$ given by \eqref{small1}, and let $\mu_0>0$ be such that
\begin{align*}
\|(\varrho_0^\mu-\varrho_0, u_0^\mu-u_0)\|_{H^3\cap\dot{B}^{-s}_{2,\infty}}^2+\|f_0^\mu-f_0\|_{L^2_v(H^3\cap \dot{B}^{-s}_{2,\infty})}^2<\frac{\eps_0}{2}\quad\text{for}\quad 0<\mu<\mu_0.
\end{align*}
This, together with \eqref{small2}, implies that 
\begin{align*}
\mathcal{E}_0^\mu=\|(\varrho_0^\mu, u_0^\mu)\|_{H^3\cap\dot{B}^{-s}_{2,\infty}}^2+\|f_0^\mu\|_{L^2_v(H^3\cap \dot{B}^{-s}_{2,\infty})}^2<\frac{\eps_0}{2}+\mathcal{E}_0\leq \frac{\eps_0}{2}+\eps_1\leq \eps_0,
\end{align*}
provided that we let $\eps_1\leq \frac{\eps_0}{2}$. According to Theorem \ref{T1.1}, for any $0<\mu<\mu_0$, we can construct a unique global solution $(\varrho^\mu,u^\mu,P^\mu,f^\mu)$ for the Cauchy problem of the inhomogeneous NS-VFP system \eqref{rNSVFP}, which satisfies the regularity estimates \eqref{uniform1} and \eqref{1.13}, independent of the viscosity $\mu$ and the time. As a direct consequence of \eqref{uniform1}, there exists a limit $(\varrho,u,P,f)$ such that as $\mu\rightarrow 0$, we have
\begin{equation*}
\begin{aligned}
(\varrho^\mu,u^\mu) &\rightharpoonup (\varrho,u)\quad  \,\, \text{weakly$^*$~ ~~in ~~} L^{\infty}(\mathbb{R}_{+};H^3),\\
\nabla P^\mu &\rightharpoonup \nabla P \quad \quad \text{weakly~ in ~~} L^{2}(\mathbb{R}_{+};H^2),\\
f^\mu &\rightharpoonup f\quad \quad \quad  \,  \text{weakly$^*$~ in ~~} L^{\infty}(\mathbb{R}_{+};L^2_v(H^3))\cap  L^2(\mathbb{R}_{+};L^2_{v,\nu}(H^3)),\\
\end{aligned}
\end{equation*}
up to a subsequence. In order to justify the convergence of   nonlinear terms in the sense of distributions, we need to analyze the time derivatives $\partial_t \varrho^\mu$ and $\partial_t u^\mu$. In fact, it follows from \eqref{rNSVFP} and \eqref{uniform1} that
\begin{align*}
\|\partial_t\varrho^\mu\|_{L^{\infty}_t(H^2)}\leq \|u^\mu \cdot \nabla \varrho^\mu\|_{L^{\infty}_t(H^2)}\leq C\|u^\mu\|_{L^{\infty}_t(H^2)}\|\varrho^\mu\|_{L^{\infty}_t(H^3)}\leq C,
\end{align*}
and
\begin{align*}
\|\partial_t u^\mu\|_{L^{\infty}_t(H^1)}&\leq \|u^\mu\cdot\nabla u^\mu\|_{L^{\infty}_t(H^1)}+\|\nabla P^\mu\|_{L^{\infty}_t(H^1)}+\|u^\mu-b^{f^\mu}\|_{L^{\infty}_t(H^1)}\\
&\quad+ \mu\Big\| \frac{ \varrho^\mu }{1+\varrho^\mu}\Delta u^\mu\Big\|_{L^{\infty}_t(H^1)}+\Big\|\frac{\varrho^\mu}{1+\varrho^\mu}\nabla P^\mu\Big\|_{L^{\infty}_t(H^1)}+\|a^{\mu}u^\mu\|_{L^{\infty}_t(H^1)}\\
&\leq \|u^\mu\|_{L^{\infty}_t(H^2)}\|\nabla u^\mu\|_{L^{\infty}_t(H^2)}+(1+\|\varrho^\mu\|_{L^{\infty}_t(H^2)}) \|\nabla P^\mu\|_{L^{\infty}_t(H^1)}\\
&\quad+\|a^{\mu}\|_{L^{\infty}_t(H^2)} \|\nabla u^\mu\|_{L^{\infty}_t(H^2)} +\mu (1+\|\varrho^\mu\|_{L^{\infty}_t(H^2)}) \|\nabla u^{\mu}\|_{L^{\infty}_t(H^2)}\leq C.
\end{align*}
Here, we have used the smallness of $\mu$.
We then conclude from the Aubin-Lions lemma, up to a subsequence, that 
\begin{align*}
(\varrho^\mu,u^\mu) & \rightarrow (\varrho,u) \quad  \,\,  \,\, \text{strongly~~ in ~~} C_{\rm loc}(\mathbb{R}_{+};H^2_{\rm loc}).
\end{align*}
Thus, it is straightforward to check that $(\rho,u,P,F)$ with $\rho=1+\varrho$ and $F=M+\sqrt{M}f$ solves the system \eqref{I-2} (or equivalently, $(\varrho,u,P,f)$ satisfies the system \eqref{rEVFP}) in the sense of distribution.

Finally, in view of \eqref{uniform1} and \eqref{1.13}, Fatou's property ensures that for $t\geq0$,
\begin{align*} 
&\|(\varrho^\mu, u^\mu)(t)\|_{H^3}^2+\|f^\mu(t)\|_{L_{v}^2(H^3)}^2+ \int_0^t \mathcal{D}(u^\mu,P^\mu,f^\mu)(\tau) {\rm d}\tau \nonumber\\
\quad \lesssim   & \liminf_{\mu\to 0} \Big(\|(\varrho^\mu, u^\mu)(t)\|_{H^3}^2+\|f^\mu(t)\|_{L_{v}^2(H^3)}^2+\int_0^t \mathcal{D}(u^\mu,P^\mu,f^\mu)(\tau) {\rm d}\tau\Big)\\
\quad \lesssim   & \lim_{\mu\to 0}\mathcal{E}_0^\mu=\mathcal{E}_0,
\end{align*}
where $\mathcal{D}(u^\mu,P^\mu,f^\mu)$ is defined by \eqref{D}. Similar to the arguments in the proof of Theorem \ref{T1.1}, one has
\begin{equation*}
\begin{aligned}
&\| u(t)\|_{L^{2}}+\|f(t)\|_{L_v^2(L^{2})}+\|\nabla P(t)\|_{L^2}\lesssim \eps_1^{\frac{1}{2}} (1+t)^{-\frac{s}{2}},\\
&\| \nabla u(t)\|_{L^{2}}+\|\nabla f(t)\|_{L_v^2(L^2)}+\|\nabla^2 P(t)\|_{L^2}\lesssim \eps_1^{\frac{1}{2}}  (1+t)^{-\frac{s}{2}-\frac{1}{2}} ,         \\
&\| \nabla^2 u (t)\|_{H^{1}}+\|\nabla^2 f(t)\|_{L_v^2(H^1)}+\|\nabla^3 P(t)\|_{L^2}\lesssim \eps_1^{\frac{1}{2}}   (1+t)^{-\frac{s}{2}-\frac{1}{2}}.
\end{aligned}
\end{equation*}
Thus, \eqref{convergence0}--\eqref{G1.7} are verified. The proof of Theorem \ref{T1.2} is completed. \hfill $\Box$

\subsection{Proof of Theorem \ref{T1.3}}

Now we are in a position to establish the global-in-time convergence of the inviscid limit.  
For $\mu\in(0,1)$, let $(\varrho^\mu,u^\mu,f^\mu)$ be the global solution to the Cauchy problem \eqref{rNSVFP}--\eqref{NJKI-10} obtained in Theorem \ref{T1.1}, and denote $a^\mu$ and $b^\mu$ the corresponding moments of $f^\mu$ (see \eqref{ab}). The equations satisfied by $(\varrho^\mu,u^\mu,f^\mu,P^\mu)$ are written as
\begin{equation}\left\{
\begin{aligned}\label{I-3mu} 
&\partial_{t}\varrho^\mu+u^\mu\cdot \nabla\varrho^\mu=0,\\ 
&\partial_{t}u^\mu+u^\mu\cdot\nabla u^\mu+\nabla P^\mu+u^\mu-b^\mu=\frac{\mu\Delta u^\mu}{1+\varrho^\mu}+\frac{\varrho^\mu\nabla P^\mu}{1+\varrho^\mu}-a^\mu u^\mu,\\
&{\rm div}\, u^\mu=0,\\
&\partial_{t}f^\mu+v\cdot\nabla_x f^\mu+u^\mu\cdot\nabla_{v}f^\mu-{\frac{1}{2}}u^\mu\cdot vf^\mu-u^\mu\cdot v\sqrt{M}\\
&\quad = \mathcal{L}f^\mu+\varrho^\mu\Big(\mathcal{L}f^\mu-u^\mu\cdot\nabla_v f^\mu+\frac{1}{2}u^\mu\cdot vf^\mu+u^\mu\cdot v\sqrt{M}\Big). 
\end{aligned}\right.
\end{equation}
For the global solution $(\varrho,u,P,f)$ given in Theorem \ref{T1.2}, it also satisfies the system 
\begin{equation}\left\{
\begin{aligned}\label{I-3*} 
&\partial_{t}\varrho +u \cdot \nabla\varrho =0,\\ 
&\partial_{t}u +u \cdot\nabla u +\nabla P +u -b =\frac{\varrho \nabla P }{1+\varrho }-a  u ,\\
&{\rm div}\, u =0,\\
&\partial_{t}f +v\cdot\nabla_x f +u \cdot\nabla_{v}f -{\frac{1}{2}}u \cdot vf -u \cdot v\sqrt{M}\\
&\quad = \mathcal{L}f +\varrho \Big(\mathcal{L}f -u \cdot\nabla_v f +\frac{1}{2}u \cdot vf +u \cdot v\sqrt{M}\Big). 
\end{aligned}\right.
\end{equation}
Define the differences
\begin{equation*}
\begin{aligned}
&(\widetilde{\varrho}^\mu,\widetilde{u}^\mu,\widetilde{P}^\mu,\widetilde{f}^\mu):=(\varrho^\mu-\varrho ,u^\mu-u , P^\mu-P , f^\mu-f ),
\end{aligned}
\end{equation*}
and
\begin{equation*}
\begin{aligned}
(\widetilde{a}^\mu,\widetilde{b}^\mu):=(a^\mu-a , b^\mu-b ).
\end{aligned}
\end{equation*}
By \eqref{I-3mu} and \eqref{I-3*}, we find that $(\widetilde{a}^\mu,\widetilde{b}^\mu)$ satisfies 
\begin{equation}\label{errorEQ}
\left\{
\begin{aligned}
&\partial_t \widetilde{\varrho}^\mu+ u^\mu\cdot \nabla \widetilde{\varrho}^\mu=\widetilde{F}_1,\\
&\partial_{t}\widetilde{u}^\mu+u^\mu\cdot\nabla \widetilde{u}^\mu+\nabla \widetilde{P}^\mu+\widetilde{u}^\mu-\widetilde{b}^\mu=-\widetilde{a}^\mu u^\mu+\frac{\mu\Delta u^\mu}{1+\varrho^\mu}+\widetilde{F}_2+\widetilde{F}_3,\\
&{\rm div}\,\widetilde{u}^\mu=0,\\
&\partial_{t}\widetilde{f}^\mu+v\cdot\nabla_x \widetilde{f}^\mu+ (1+\varrho)u^\mu\cdot\nabla_{v}\widetilde{f}^\mu\\
& \quad +\frac{1}{2}(1+\varrho) {u}^\mu\cdot v \widetilde{f} ^{\mu} -\widetilde{u}^\mu\cdot v\sqrt{M}-\mathcal{L}\widetilde{f}^\mu=\widetilde{F}_4+\widetilde{F}_5,
\end{aligned}\right.
\end{equation}
where $\widetilde{F}_i$, $i=1,2,3,4$, are given by
\begin{equation*}
\begin{aligned}
\widetilde{F}_1:=&\,-\widetilde{u}^\mu\cdot \nabla \varrho ,\\
\widetilde{F}_2:=&\,\Big(\frac{\varrho^\mu}{1+\varrho^\mu}-\frac{\varrho }{1+\varrho }\Big) \nabla P^\mu,\\
\widetilde{F}_3:=&\,-\widetilde{u}^\mu\cdot \nabla u -a^{f} \widetilde{u}^\mu+\frac{\varrho }{1+\varrho }\nabla \widetilde{P}^\mu ,\\
\widetilde{F}_4:=&\,(\varrho^\mu-\varrho )\Big(-u^\mu\cdot\nabla_v f^\mu+\frac{1}{2}u^\mu\cdot vf^\mu\Big) -(\varrho+1)\Big(\widetilde{u}^\mu\cdot\nabla_v f+\frac{1}{2} \widetilde{u}^{\mu}\cdot v {f}   \Big),\\
\widetilde{F}_5:=&\,(\varrho^\mu-\varrho )(\mathcal{L}f^\mu+u^\mu\cdot v\sqrt{M})+\varrho (\mathcal{L}\widetilde{f}^\mu+\widetilde{u}^\mu\cdot v\sqrt{M}).
\end{aligned}
\end{equation*}


We first establish the estimate of $\widetilde{\varrho}^\mu$ due to the transport nature for $\eqref{errorEQ}_1$.
\begin{lem}\label{L:error1}
It holds for all $t\geq0$ that
\begin{align}\label{widerho}
\|\widetilde{\varrho}^\mu(t)\|_{H^1}\leq C\mu+C\eps_1^{\frac{1}{2}} t^{\frac{1}{2}}\|\nabla \widetilde{u}^\mu\|_{L^2_t(L^2)}.
\end{align}
Here $\eps_1$ is given by \eqref{small1}.
\end{lem}
 
\begin{proof}
Performing the energy estimates on $\eqref{errorEQ}_1$, we obtain 
\begin{align*}
 \frac{1}{2}\frac{{\rm d}}{{\rm d}t} \|\widetilde{\varrho}^\mu\|_{H^1}^2
 =\,&-\sum_{0\leq |\alpha|\leq 1} \int_{\mathbb{R}^3}\big( u^{\mu}\cdot \nabla \partial^\alpha\widetilde{\varrho}^\mu+[\partial^\alpha,u^\mu\cdot\nabla]
\widetilde{\varrho}^\mu+\partial^\alpha\widetilde{F}_1\big) \partial^\alpha\widetilde{\varrho}^\mu{\rm d}x\nonumber\\
 \leq\,& C\sum_{0\leq |\alpha|\leq 1}\big(\|{\rm div}\, u^\mu\|_{L^{\infty}} \|\partial^\alpha\widetilde{\varrho}^\mu\|_{L^2}^2+\|\partial^\alpha u^\mu\|_{L^3} \|\nabla \widetilde{\varrho}^\mu\|_{L^{6}}\|\partial^\alpha\widetilde{\varrho}^\mu\|_{L^2}\big)\nonumber\\
 &  +C\|\partial^\alpha\widetilde{F}_1\|_{L^2}\|\partial^\alpha\widetilde{\varrho}^\mu\|_{L^2}\nonumber\\
  \leq\,&C\|\nabla u^\mu\|_{H^2}\|\widetilde{\varrho}^\mu\|_{H^1}^2+\|\widetilde{F}_1\|_{H^1} \|\widetilde{\varrho}^\mu\|_{H^1},
\end{align*}
where Lemma \ref{LA.2}  has been used.

For the nonlinear term $\widetilde{F}_1$, we have
\begin{align*}
\|\widetilde{F}_1\|_{H^1}&\leq  \|\widetilde{F}_1\|_{L^2}+\|\nabla\widetilde{F}_1\|_{L^2}\\
&\leq \|\nabla\varrho \|_{L^3}\|\widetilde{u}^\mu\|_{L^6}+\|\nabla\varrho\|_{L^\infty}\|\nabla \widetilde{u}^\mu\|_{L^2}+\|\nabla^2 \varrho\|_{L^3}\|\widetilde{u}^\mu\|_{L^6}\\
&\leq C\|\varrho \|_{H^3} \|\nabla \tilde{u}^\mu\|_{L^2}.
\end{align*}
It thus holds that
\begin{align*}
\|\widetilde{\varrho}^\mu\|_{H^1}\leq e^{C\int_0^t \|\nabla u^\mu\|_{H^2}  {\rm d}\tau} 
\big( \|\varrho_0^\mu-\varrho_0\|_{H^1}+C\sup_{\tau\in[0,t]} \| \varrho\|_{H^3} \,t^{\frac{1}{2}}\|\nabla \tilde{u}^\mu\|_{L^2_t(L^2)}\big).
\end{align*}
Thus, applying the uniform bounds  \eqref{uniform1}, \eqref{1.13},   \eqref{uniform2} and \eqref{a3} yields \eqref{widerho}.
\end{proof}

To overcome the time growth when addressing $\widetilde{\varrho}^\mu$ in \eqref{widerho}, we first establish some additional time-weighted integrability estimates for the NS-VFP system \eqref{rNSVFP}.

\begin{lem}\label{LL4.2}
It holds that
\begin{equation}\label{weighted:inter}
\int_0^t (1+\tau)\big(\|\nabla P^\mu\|_{H^2}^2+\|u^\mu-b^\mu\|_{H^3}^2 +\|\nabla(a^\mu,b^\mu)\|_{H^2}^2+\|\{\mathbf{I-P}\}f^\mu\|_{L^2_{v,\nu}(H^3 )}^2\big){\rm d}\tau \leq C\eps_0.
\end{equation}
\end{lem}

\begin{proof}
Multiplying the Lyapunov functional inequality \eqref{G3.46} by $(1+t) $, we obtain
\begin{align*} 
\frac{{\rm d}}{{\rm d}t}\big( (1+t)\mathcal{E}(t)\big)+\lambda_4 (1+t)\mathcal{D}(t)\leq \mathcal{E}(t).
\end{align*}
which yields
\begin{align}\label{NJK4.9}
(1+t)\mathcal{E}(t)+\lambda_4\int_0^t (1+\tau) \mathcal{D}(\tau) {\rm d}\tau\leq \mathcal{E}_0^\mu +\int_0^t \mathcal{E}(\tau){\rm d}\tau.
\end{align}
Recall
\begin{align}\label{NJK4.10}
    \int_0^t \mathcal{E}(\tau){\rm d}\tau&\leq C \mathcal{E}_0^\mu\int_0^t (1+\tau)^{-s} {\rm d}\tau  \leq C\mathcal{E}_0^\mu,
\end{align}
for $s>1$.
Combining \eqref{NJK4.9} and \eqref{NJK4.10}, it holds that
\begin{align*}
\int_0^t (1+\tau) \mathcal{D}(\tau) {\rm d}\tau\leq C\mathcal{E}_0^\mu,
\end{align*}
which, combined with \eqref{G3.2}, \eqref{G3.6}--\eqref{NJKG3.6}
 and \eqref{G3.480}, yields \eqref{weighted:inter}.
\end{proof}

Then, we are in a position to establish the $L^2$ estimate of the error $(\widetilde{u}^\mu,\widetilde{f}^\mu)$. To that end, we introduce the functional 
\begin{align}\label{4.11}
\widetilde{\mathcal{X}}^\mu(t) 
:=\,&\sup_{\tau\in[0,t]}\big(\|\widetilde{u}^\mu(\tau)\|_{H^1}^2
+\|\widetilde{f}^\mu(\tau)\|_{L^2_v(H^1)}^2\big)  
\nonumber\\
& +\int_{0}^t\big(  \|\widetilde{b}^\mu-\widetilde{u}^\mu\|_{H^1}^2
+\|\nabla(\widetilde{a}^\mu,\widetilde{b}^\mu)\|_{L^2}^2
+\|\{\mathbf{I-P}\}\widetilde{f}^{\mu}\|_{L^2_{v,\nu}(H^1)}^2 \big) {\rm d}\tau.
\end{align}
It is clear that
\begin{align}\label{4.12}
\int_0^t \big( \|\nabla \widetilde{f}^\mu\|_{L_{x,v}^2}^2+\|\nabla \widetilde{u}^\mu\|_{L^2}^2\big) {\rm d}\tau
\leq C\widetilde{\mathcal{X}}^{\mu}(t).
\end{align}

\begin{lem}\label{L:error2}
It holds that
\begin{align}\label{4.19}
&\sup_{\tau\in[0,t]} \big(\|\widetilde{u}^\mu(\tau)\|_{L^2}^{2}+\|\widetilde{f}^\mu(\tau)\|_{L_{x,v}^2}^{2}\big)\nonumber\\
&\quad +\int_0^{t}\big(\|\widetilde{b}^\mu-\widetilde{u}^\mu\|_{L^2}^{2}
+\|\nabla \widetilde{P}^{\mu}\|_{L^2}^2
+\|\{\mathbf{I}-\mathbf{P}\}\widetilde{f}^\mu\|_{L^2_{v,\nu}(L^2)}^{2}\big) {\rm d}\tau \nonumber\\
 &\qquad \leq C\mu^2+C \big(\eps_0 ^{\frac{1}{2}}+\eps_1 ^{\frac{1}{2}}\big)\widetilde{\mathcal{X}}^\mu(t),
\end{align}
where $C>0$ is a constant independent of $t$ and $\mu$.

\end{lem}

\begin{proof}
It follows from \eqref{errorEQ}$_2$--\eqref{errorEQ}$_4$ and \eqref{H2.1} that
\begin{align}\label{NJK4.14}
&\frac{1}{2}\big( \|\widetilde{u}^\mu\|_{L^2}^2+\|\widetilde{f}^\mu\|_{L_{x
,v}^{2}}^2\big)+\int_0^t\|\widetilde{b}^\mu-\widetilde{u}^\mu\|_{L^2}^{2}{\rm d}x{\rm d}\tau+\bar{\lambda}\int_0^t\|\{\mathbf{I}-\mathbf{P}\}\widetilde{f}^\mu\|_{L^2_{v,\nu}(L^2)}^{2}  {\rm d}x{\rm d}\tau\nonumber\\
\leq\,&\frac{1}{2}\int_0^t\!\int_{\mathbb{R}^3} (1+\varrho)u^\mu\cdot \langle v \widetilde{f}^\mu,\widetilde{f}^\mu\rangle {\rm d}x{\rm d}\tau-\int_0^t\!\int_{\mathbb{R}^3} \widetilde{a}^\mu 
 u^\mu \cdot\widetilde{u}^{\mu} {\rm d}x{\rm d}\tau\nonumber\\
&+\int_0^t\int_{\mathbb{R}^3} \frac{\mu\Delta u^\mu\cdot\widetilde{u}^\mu }{1+\varrho^\mu} {\rm d}x{\rm d}\tau+\int_0^t\!\int_{\mathbb{R}^3} \widetilde{F}_2 \cdot \widetilde{u}^{\mu}{\rm d}x{\rm d}\tau+\int_0^t\!\int_{\mathbb{R}^3} \widetilde{F}_3 \cdot \widetilde{u}^\mu{\rm d}x{\rm d}\tau\nonumber\\
&+\int_0^t\!\int_{\mathbb{R}^3\times\mathbb{R}^3} \widetilde{F}_4 \widetilde{f}^\mu{\rm d}x{\rm d}v{\rm d}\tau+\int_0^t\!\int_{\mathbb{R}^3\times\mathbb{R}^3} \widetilde{F}_5 \widetilde{f}^\mu{\rm d}x{\rm d}v{\rm d}\tau+\frac{1}{2}\big(\|\widetilde{u}^\mu(0)\|_{L^2}^2+\|\widetilde{f}^\mu(0)\|_{L_{x
,v}^{2}}^2\big)\nonumber\\
\equiv:\,&\sum_{j=1}^7\widetilde{I}_j+\frac{1}{2}\big(\|\widetilde{u}^\mu(0)\|_{L^2}^2
+\|\widetilde{f}^\mu(0)\|_{L_{x
,v}^{2}}^2\big).
\end{align}
Following the same lines for those in handling the terms $L_3$ and $L_4$ in Lemma \ref{L3.3}, one has 
\begin{align}\label{NJK4.15}
&|\widetilde{I}_1+\widetilde{I}_2|\nonumber\\
\leq\,& C\Big(1+\sup_{\tau\in[0,t]} \| \varrho\|_{H^3}\Big)\int_0^{t}\big( \|u^\mu\|_{L^{\infty}} \|\{\mathbf{I-P}\}\widetilde{f}^\mu\|_{L^2_{v,\nu}(L^2)}^2
+\|\widetilde{b}^\mu-\widetilde{u}^\mu\|_{L^2}\|u^\mu\|_{L^3}\|\widetilde{a}^\mu\|_{L^6}  \nonumber \\
&\quad +\|u^\mu\|_{L^3} \|(\widetilde{a}^\mu,\widetilde{b}^\mu)\|_{L^6}\|\{\mathbf{I-P}\}\widetilde{f}^\mu\|_{L^2_{v,\nu}(L^2)}\big) {\rm d}\tau+C\int_0^t\|\varrho\|_{H^3}\|u^{\mu}\|_{H^2}\|\nabla (\widetilde{a}^{\mu},\widetilde{b}^{\mu})\|_{L^2}^2{\rm d}\tau\nonumber\\
\leq\,&C  (1+\eps_1^{\frac{1}{2}} )\sup_{\tau\in[0,t]}\|u^\mu(\tau)\|_{H^2}\int_0^t \big( \|\widetilde{b}^\mu-\widetilde{u}^\mu\|_{L^2}^2+\|\nabla(\widetilde{a}^\mu,\widetilde{b}^\mu)\|_{L^2}^2
+\|\{\mathbf{I-P}\}\widetilde{f}^\mu\|_{L^2_{v,\nu}(L^2)}^2\big){\rm d}\tau \nonumber\\
\leq\,& C {\eps_0}^{\frac{1}{2}} \int_0^t \big( \|\widetilde{b}^\mu-\widetilde{u}^\mu\|_{L^2}^2+\|\nabla(\widetilde{a}^\mu,\widetilde{b}^\mu)\|_{L^2}^2
+\|\{\mathbf{I-P}\}\widetilde{f}^\mu\|_{L^2_{v,\nu}(L^2)}^2\big){\rm d}\tau.
\end{align}
Note that the term $\widetilde{I}_3$ causes the convergence rate of the inviscid limit. Indeed, it follows from the decay estimate \eqref{1.13} for $\|\nabla^2u\|_{L^2}$ and the lower bound of $\rho^\mu=1+\varrho^\mu$ that 
\begin{align}\label{NJKH4.17}
|\widetilde{I}_3|\leq \,&C\mu \int_0^t \|\nabla^2 u^\mu \|_{L^2} {\rm d}\tau \sup_{\tau\in[0,t]} \| \widetilde{u}^{\mu}\|_{L^2} \nonumber\\
\leq\, &C\mu \eps_0^{\frac{1}{2}} \int_0^t (1+\tau)^{-\frac{s}{2}-\frac{1}{2}}{\rm d}\tau \,\sup_{\tau\in[0,t]} \| \widetilde{u}^{\mu}\|_{L^2}\nonumber\\
\leq\,& C\mu^2+\frac{1}{4}\sup_{\tau\in[0,t]} \| \widetilde{u}^{\mu}\|_{L^2}^2.
\end{align}
As for the term $\widetilde{I}_4$, since $L^3 (\mathbb{R}^3)\hookrightarrow H^1(\mathbb{R}^3)$ 
and $L^6(\mathbb{R}^3)\hookrightarrow \dot{H}^1(\mathbb{R}^3)$, we make use of the error estimate \eqref{widerho} and the weighted  time integrability \eqref{weighted:inter} to obtain
\begin{align}\label{NJH4.18}
|\widetilde{I}_4|\leq\,& C\Big\|(1+\tau)^{-\frac{1}{2}}
\Big(\frac{\varrho^\mu}{1+\varrho^\mu}-\frac{\varrho}{1+\varrho}\Big)
\Big\|_{L^{\infty}_t(L^2)}\big\|(1+\tau)^{\frac{1}{2}}\nabla P^\mu\|_{L^2_t(L^3)}\big\|\widetilde{u}^\mu\|_{L^2_t(L^6)} \nonumber\\
\leq\,& C\big\|(1+\tau)^{-\frac{1}{2}}\widetilde{\varrho}^\mu\big\|_{L^{\infty}_t(L^2)}\big\|(1+\tau)^{\frac{1}{2}}\nabla P^\mu\big\|_{L^2_t(H^1)}\|\nabla\widetilde{u}^\mu\|_{L^2_t(L^2)} \nonumber\\
\leq\,& C\mu^2+C\eps_1^{\frac{1}{2}}\widetilde{\mathcal{X}}^\mu(t).  
\end{align}
The term $\widetilde{I}_5+\widetilde{I}_6$  can be analyzed by
\begin{align}\label{NJKH4.18}
|\widetilde{I}_5+\widetilde{I}_6|\leq\,& \int_0^t\big(\|\widetilde{F}_3\|_{\dot{H}^{-1}}\|\widetilde{u}^\mu\|_{\dot{H}^{1}}+\| \widetilde{F}_4\|_{L^2_v(\dot{H}^{-1})} \|\widetilde{f}^\mu\|_{L^2_{v}(\dot{H}^{1})} \big){\rm d}\tau \nonumber\\
\leq\,& \|\widetilde{F}_3\|_{L^2_t(\dot{H}^{-1})}\| \widetilde{u}^{\mu}\|_{L^2_t(\dot{H}^1)}+\| \widetilde{F}_4\|_{L^2_t(L^2_v(\dot{H}^{-1}))} \|\widetilde{f}^\mu\|_{L^2_t(L^2_v(\dot{H}^{1} ))}.
\end{align}
Then, as $1<s\leq \frac{3}{2}$, Lemma \ref{LA.7} as well as $L^2=\dot{B}^{0}_{2,2}$ ensure that
\begin{align}\label{I2I3}
\| \widetilde{F}_3\|_{L^2_t(\dot{H}^{-1})}\sim \| \widetilde{F}_3\|_{L^2_t(\dot{B}^{-1}_{2,2})}\leq\, &C\|\widetilde{u}^\mu\|_{L^2_t(\dot{B}^{1}_{2,2})}\|\nabla u\|_{L^{\infty}_t(\dot{B}^{-\frac{1}{2}}_{2,1})}
+C\|a \|_{L^{\infty}_t(\dot{B}^{-\frac{1}{2}}_{2,1})}\|\widetilde{u}^{\mu}\|_{L^2_t(\dot{B}^{1}_{2,2})} \nonumber\\
&+C\Big\|\frac{\varrho}{1+\varrho}
\Big\|_{L^{\infty}_t(\dot{B}^{\frac{1}{2}}_{2,1})}\|\nabla\widetilde{P}^\mu\|_{L^2_t(\dot{B}^{0}_{2,2})} \nonumber\\
\leq\,& C\|\nabla\widetilde{u}^\mu\|_{L^2_t(L^2)} \Big(\|u\|_{L^{\infty}_t(\dot{B}^{\frac{1}{2}}_{2,1})}+\|f\|_{L^{\infty}_t(L^2_v(\dot{B}^{-\frac{1}{2}}_{2,1}))} \Big) \nonumber\\
& +C\|\varrho\|_{L^{\infty}_t(\dot{B}^{\frac{1}{2}}_{2,1})}\|\nabla \widetilde{P}^\mu\|_{L^2_t(L^2)}.
\end{align}
We recall that $(\varrho^\mu,u^\mu,f^\mu)$ satisfies the uniform bounds \eqref{G3.85} and \eqref{G4.1}. From Fatou's property, it implies that
\begin{align}
&\|(\varrho,u)\|_{L^{\infty}_{t}(\dot{B}^{-s}_{2,\infty})}
+\|f\|_{L^{\infty}_t(L_v^2(\dot{B}^{-s}_{2,\infty}))}\leq C\eps_1^{\frac{1}{2}}.\label{lowuniform2}
\end{align}
In view of \eqref{1.13}, \eqref{uniform2}, \eqref{lowuniform2} and the interpolation property (iv) in Lemma \ref{LA.7}, one finds that 
\begin{align}\label{NJKH4.20}
&\|\varrho\|_{L^{\infty}_t(\dot{B}^{\frac{1}{2}}_{2,1})}+\|u\|_{L^{\infty}_t(\dot{B}^{\frac{1}{2}}_{2,1})}
+\|f\|_{L^{\infty}_t(L^2_v(\dot{B}^{-\frac{1}{2}}_{2,1}))}\nonumber\\
&\quad\leq C\|(\varrho,u)\|_{L^{\infty}_t(\dot{B}^{-s}_{2,\infty}\cap H^3)}+\|f\|_{L^{\infty}_t(L^2_v(\dot{B}^{-s}_{2,\infty}\cap H^3))}\leq C\eps_1^{\frac{1}{2}}.
\end{align}
Similarly to \eqref{G3.2}, it can be deduced after multiplying \eqref{errorEQ}$_2$ by $\nabla\widetilde{P}^{\mu}$ that    
\begin{align*}
\|\nabla \widetilde{P}^{\mu}\|_{L^2}^2\lesssim\,&   \|    {u}^{\mu}\cdot\nabla \widetilde{u}^{\mu} \|_{L^{2}}^2+\|\widetilde{b}^{\mu}-\widetilde{u}^{\mu}\|_{L^2}^2+\mu^2\|\nabla^2 u^{\mu}\|_{L^{2}}^2+\|\widetilde{a}^{\mu}u^{\mu}\|_{L^2}^2 \nonumber\\
&+  \Big\|\Big(\frac{\varrho^\mu}{1+\varrho^\mu}-\frac{\varrho }{1+\varrho }\Big)
\nabla P^{\mu}\Big\|_{L^2}^2+ \|    \widetilde{u}^{\mu}\cdot\nabla  {u}^{\mu} \|_{L^{2}}^2+\|a^{f}\widetilde{u}^{\mu}\|_{L^2}^2 \nonumber\\
\lesssim\,& \|u^{\mu}\|_{L^{\infty}}^2\|\nabla\widetilde{u}^{\mu}\|_{L^2}^2
+\|\widetilde{b}^{\mu}-\widetilde{u}^{\mu}\|_{L^2}^2+\|u^{\mu}\|_{L^{\infty}}^2\| \widetilde{a}^{\mu}\|_{L^2}^2+\mu^2\|\nabla^2 u^{\mu}\|_{L^2}^2\nonumber\\
&+\|\varrho^{\mu}-\varrho\|_{H^1}^2\|\nabla P^{\mu}\|_{H^1}^2+\|\widetilde{u}^{\mu}\|_{H^1}^2\|\nabla^2 u^{\mu}\|_{L^2}^2+\|\nabla a\|_{L^2}^2\|\widetilde{u}^{\mu}\|_{H^1}^2\nonumber\\
\lesssim\,& \big(\|\nabla u^{\mu}\|_{H^1}^2+\|\nabla a\|_{L^2}^2\big)\big(\|\widetilde{u}^{\mu}\|_{H^1}^2+\|\widetilde{f}^{\mu}\|_{L_v^2(H^1)}^2\big)\nonumber\\
&+\big\|(1+t)^{\frac{1}{2}}\nabla P^{\mu}\big\|_{L^2}^2\big\|(1+t)^{-\frac{1}{2}}\widetilde{\varrho}^{\mu}\big\|_{H^1}^2
+\|\widetilde{b}^{\mu}-\widetilde{u}^{\mu}\|_{L^2}^2+\mu^2\|\nabla^2 u^{\mu}\|_{L^2}^2,
\end{align*}
which, together with \eqref{uniform1} and \eqref{weighted:inter}, implies that
\begin{align}\label{NJK4.21}
\|\nabla \widetilde{P}^{\mu}\|_{L_t^2(L^2)}\leq \, & C\big(\eps_0^{\frac{1}{2}}+\eps_1^{\frac{1}{2}}\big) \big(\|\widetilde{u}^{\mu}\|_{L^{\infty}_t(H^1)}+\|\nabla\widetilde{u}^{\mu}\|_{L^{2}_t(L^2)}
+\|\widetilde{f}^{\mu}\|_{L^{\infty}_t(L_v^2(H^1))}\big)\nonumber\\
& +C\mu+\|\widetilde{b}^{\mu}-\widetilde{u}^{\mu}\|_{L^2(L^2)}.  
\end{align}
Substituting \eqref{lowuniform2}--\eqref{NJK4.21} into \eqref{I2I3}, we obtain
\begin{align}\label{sgggg}
\| \widetilde{F}_3\|_{L^2_t(\dot{H}^{-1})}&\leq C\big(\eps_0^{\frac{1}{2}}+\eps
_1^{\frac{1}{2}}\big)\widetilde{\mathcal X}^\mu(t)^{\frac{1}{2}}+C\eps_1^{\frac{1}{2}}\mu.
\end{align}
Following the same line, invoking Lemmas \ref{L:error1}--\ref{LL4.2} yields
\begin{align} \label{sgggg2}
&\| \widetilde{F}_4\|_{L^2_t(\dot{H}^{-1})}\nonumber\\
\lesssim\, & \big\|(1+\tau)^{-\frac{1}{2}}\widetilde{\varrho}^\mu\big\|_{L^{\infty}_t(\dot{B}^{\frac{1}{2}}_{2,1})}
\Big(\big\|(1+\tau)^{\frac{1}{2}}u^{\mu}\cdot\nabla_v f^{\mu}\big\|_{L^{2}_t(L_v^2(\dot{B}^{0}_{2,2}))}+\big\|(1+\tau)^{\frac{1}{2}}u^{\mu}\cdot v f^{\mu}\big\|_{L^{2}_t(L_v^2(\dot{B}^{0}_{2,2}))}  \Big)\nonumber\\
&+  \| 1+\varrho \|_{L^{\infty}_t(\dot{B}^{\frac{1}{2}}_{2,1})}\Big(\|\widetilde{u}^\mu\nabla_v f\|_{L^2_t(L_v^2(\dot{B}^{0}_{2,2}))}+\|\widetilde{u}^\mu\cdot vf\|_{L^2_t(L_v^2(\dot{B}^{0}_{2,2}))}\Big) \nonumber\\
\lesssim\,&  \big(\mu+\eps_1^{\frac{1}{2}}\widetilde{\mathcal{X}}^\mu(t)^{\frac{1}{2}}\big) \big( \big\|(1+\tau)^{\frac{1}{2}}\nabla(a^\mu,b^\mu)\big\|_{L_t^2(H^2)} +\big\|(1+\tau)^{\frac{1}{2}}\{\mathbf{I-P}\}f^\mu\big\|_{L_t^2(L^2_{v,\nu}(H^3 ))}   \big)\|u^{\mu}\|_{L_t^{\infty}(L^2)}\nonumber\\
& +\big(1+ \|\varrho\|_{L^{\infty}_t(H^3)}\big)\big( \|\nabla(a^\mu,b^\mu)\|_{L_t^2(H^2)} +\|\{\mathbf{I-P}\}f^\mu\|_{L_t^2(L^2_{v,\nu}(H^3 ))}   \big) \|\widetilde{u}^{\mu}\|_{L_t^{\infty}(H^1)}\nonumber\\
\lesssim\,& \big(\eps_0^{\frac{1}{2}}+\eps_1^{\frac{1}{2}}\big)\mu+\big(\eps_0^{\frac{1}{2}}+\eps_1^{\frac{1}{2}}\big)
\widetilde{\mathcal{X}}^\mu(t)^{\frac{1}{2}}.
\end{align} 
Therefore, substituting \eqref{sgggg}--\eqref{sgggg2} into \eqref{NJKH4.18}, we arrive at
\begin{align}\label{NJK4.26}
|\widetilde{I}_5+\widetilde{I}_6| \lesssim C\big(\eps_0^{\frac{1}{2}}+\eps_1^{\frac{1}{2}}\big) \widetilde{\mathcal{X}}^\mu(t)+(\eps_0+\eps_1)\mu^2.
\end{align}
Concerning the last term $\widetilde{I}_7$, we use    integration by parts and 
the basic fact $v\sqrt{M}\lesssim 1$ to obtain 
\begin{align}\label{NJK4.27}
|\widetilde{I}_7|\lesssim\,& \big\|(1+\tau)^{-\frac{1}{2}}\widetilde{\varrho}^{\mu}\big\|_{L_t^{\infty}(L^3)}\big( \big\|(1+\tau)^{\frac{1}{2}}\nabla(a^\mu,b^\mu)\big\|_{L_t^2(H^2)} +\big\|(1+\tau)^{\frac{1}{2}}\{\mathbf{I-P}\}f^\mu\big\|_{L_t^2(L^2_{v,\nu}(H^3 ))}    \big)\nonumber\\
&\times \big(\| \{\mathbf{I-P}\}\widetilde{f}^\mu\|_{L_t^2(L^2_{v,\nu}(H^1 ))}+\|\nabla(\widetilde{a}^{\mu},\widetilde{b}^{\mu})\|_{L_t^2(L^2)}^2\big)\nonumber\\
&+\big\|(1+\tau)^{-\frac{1}{2}}
\widetilde{\varrho}^{\mu}\big\|_{L_t^{\infty}(L^3)}\big\|(1+\tau)^{\frac{1}{2}}(u^{\mu}-b^{\mu})
\big\|_{L_t^2(L^2)}\|(\widetilde{u}^{\mu},\widetilde{b}^{\mu})\|_{L_t^2(L^6)}\nonumber\\
&+\|\varrho\|_{L_t^{\infty}(L^{\infty})}\big(\| \{\mathbf{I-P}\}
\widetilde{f}^\mu\|_{L_t^2(L^2_{v,\nu}(H^1 ))}^2+\|\widetilde{u}^{\mu}-\widetilde{b}^{\mu}\|_{L_t^2(L^2)}
\|\widetilde{b}^{\mu}\|_{L_t^2(L^6)}\big)\nonumber\\
\lesssim\,&(\eps_0+\eps_1)\big\|(1+\tau)^{-\frac{1}{2}}\widetilde{\varrho}^{\mu}
\big\|_{L_t^{\infty}(H^1)}\widetilde{\mathcal{X}}^\mu(t)^{\frac{1}{2}}+\eps_0
\big\|(1+\tau)^{-\frac{1}{2}}\widetilde{\varrho}^{\mu}\big\|_{L_t^{\infty}(H^1)}
\|(1+\tau)^{\frac{1}{2}}\nabla(\widetilde{u}^{\mu},\widetilde{b}^{\mu})\|_{L_t^2(L^2)}\nonumber\\
&+\|\varrho\|_{L_t^{\infty}(H^3)}\big( \widetilde{\mathcal{X}}^\mu(t)+\widetilde{\mathcal{X}}^\mu(t)^{\frac{1}{2}}
\|\nabla\widetilde{b}^{\mu}\|_{L_t^2(L^2)}  \big)\nonumber\\
\lesssim\,&\big(\eps_0^{\frac{1}{2}}+\eps_1^{\frac{1}{2}}\big)\widetilde{\mathcal{X}}^\mu(t)+\mu^2.
\end{align}
Inserting the estimates \eqref{NJK4.15}--\eqref{NJH4.18}, \eqref{NJK4.26}, and \eqref{NJK4.27} into \eqref{NJK4.14} and adding \eqref{NJK4.21} with a small coefficient, we derive the desired \eqref{4.19} and finish the proof of Lemma \ref{L:error2}.
\end{proof}

\begin{lem}\label{L:error3}
It holds that
\begin{align}\label{NJK4.28}
&\sup_{\tau\in[0,t]} \big(\|\nabla  \widetilde{u}^\mu(\tau)\|_{L^2}^{2}+\| \nabla \widetilde{f}^\mu(\tau)\|_{L_{x,v}}^{2}\big)+\int_0^{t}\big(\|\nabla(\widetilde{b}^\mu-\widetilde{u}^\mu)\|_{L^2}^{2}
+\|\nabla\{\mathbf{I}-\mathbf{P}\}\widetilde{f}^\mu\|_{L^2_{v,\nu}(L^2)}^{2}\big){\rm d}\tau\nonumber\\
&\quad\leq C\mu^2+C\big(\eps_0^{\frac{1}{2}}+\eps_1^{\frac{1}{2}}\big)\widetilde{\mathcal{X}}^\mu(t),
\end{align}
where $C>0$ is a constant independent of $t$ and $\mu$.
\end{lem}

\begin{proof}
Applying  the operator $\partial^{\alpha}$ with $|\alpha|=1$ to \eqref{errorEQ}$_2$ and \eqref{errorEQ}$_4$ respectively, performing $L^2$ energy estimates and using ${{\rm div}\,} \widetilde{u}^{\mu}=0$,
then employing \eqref{H2.1}, we have
\begin{align}\label{NJK4.29}
&\frac{1}{2}\big( \|\partial^{\alpha}\widetilde{u}^\mu\|_{L^2}^2+\|\partial^{\alpha}\widetilde{f}^\mu
\|_{L_{x,v}^2}^2\big)
+\int_0^t\|\partial^{\alpha}(\widetilde{b}^\mu-\widetilde{u}^\mu)\|_{L^2}^{2}{\rm d}x{\rm d}\tau+\bar{\lambda}\int_0^t\|\partial^{\alpha}\{\mathbf{I}-\mathbf{P}\}\widetilde{f}^\mu\|_{L^2_{v,\nu}(L^2)}^{2}  {\rm d}x{\rm d}\tau\nonumber\\
\leq\,&-\int_0^t\!\int_{\mathbb{R}^3} \partial^{\alpha}(\widetilde{a}^\mu  u^\mu) \cdot\partial^{\alpha}\widetilde{u}^{\mu} {\rm d}x{\rm d}\tau+\frac{1}{2}\int_0^t \! \int_{\mathbb{R}^3}
\big\langle\partial^{\alpha}\big((1+\varrho)u^{\mu}\cdot v\widetilde{f}^{\mu}\big),\partial^{\alpha}\widetilde{f}^{\mu}\big\rangle {\rm d}x{\rm d}\tau\nonumber\\
&+\int_0^t\!\int_{\mathbb{R}^3\times \mathbb{R}^3}\partial^{\alpha}\big((1+\varrho) {u}^{\mu}\big)\cdot\nabla_v \widetilde{f}^{\mu}\partial^{\alpha} \widetilde{f}^{\mu}{\rm d}x{\rm d}v{\rm d}\tau+\int_0^t\int_{\mathbb{R}^3} \partial^{\alpha}\Big(\frac{\mu\Delta u^\mu }{1+\varrho^\mu}\Big)\cdot\partial^{\alpha}\widetilde{u}^{\mu} {\rm d}x{\rm d}\tau\nonumber\\
&+\int_0^t\!\int_{\mathbb{R}^3}\partial^{\alpha} u^{\mu}\cdot\nabla\widetilde{u}^{\mu}\cdot\partial^{\alpha}\widetilde{u}^{\mu}{\rm d}x{\rm d}\tau +\int_0^t\!\int_{\mathbb{R}^3} \partial^{\alpha}\widetilde{F}_2 \cdot \partial^{\alpha}\widetilde{u}^{\mu}{\rm d}x{\rm d}\tau+\int_0^t\!\int_{\mathbb{R}^3} \partial^{\alpha}\widetilde{F}_3 \cdot \partial^{\alpha} \widetilde{u}^\mu{\rm d}x{\rm d}\tau\nonumber\\
&+\int_0^t\!\int_{\mathbb{R}^3\times\mathbb{R}^3} \partial^{\alpha}\widetilde{F}_4 \partial^{\alpha}\widetilde{f}^\mu{\rm d}x{\rm d}v{\rm d}\tau+\int_0^t\!\int_{\mathbb{R}^3\times\mathbb{R}^3} \partial^{\alpha}\widetilde{F}_5 \partial^{\alpha}\widetilde{f}^\mu{\rm d}x{\rm d}v{\rm d}\tau\nonumber\\
&+\frac{1}{2}\big( \|\nabla\widetilde{u}^\mu(0)\|_{L^2}^2+\|\nabla\widetilde{f}^\mu(0)\|_{L_{x,v}^2}^2\big) \nonumber\\
\equiv:\,& \sum_{k=1}^9\widetilde{J}_k+\frac{1}{2}\big( \|\nabla\widetilde{u}^\mu(0)\|_{L^2}^2+\|\nabla\widetilde{f}^\mu(0)\|_{L_{x,v}^2}^2\big).
\end{align}
In view of \eqref{4.11} and 
\eqref{4.12}, and the macro-micro decomposition \eqref{mmd}, we obtain
\begin{align}\label{NJK4.30}
|\widetilde{J}_1|\lesssim\,&\|\nabla \widetilde{a}^{\mu}\|_{L_t^2(L^2)}\|u^{\mu}\|_{L_t^{\infty}(L^{\infty})} \|\nabla\widetilde{u}^{\mu}\|_{L_t^2(L^2)} +  \|  \widetilde{a}^{\mu}\|_{L_t^2(L^6)}\|u^{\mu}\|_{L_t^{\infty}(L^{3})} \|\nabla\widetilde{u}^{\mu}\|_{L_t^2(L^2)}\nonumber\\
\lesssim\,&\|u^{\mu}\|_{L_t^{\infty}(H^3)}\widetilde{\mathcal{X}}^\mu(t)\nonumber\\
\lesssim\,&\eps_0^{\frac{1}{2}}\widetilde{\mathcal{X}}^\mu(t),\\\label{NJK4.31}
|\widetilde{J}_2+\widetilde{J}_3|\lesssim\,&  \|\nabla\varrho\|_{L_t^{\infty}(L^6)}\| {u}^{\mu}\|_{L_t^{\infty}(L^6)} \big(\| \{\mathbf{I-P}\}\widetilde{f}^\mu\|_{L_t^2(L^2_{v,\nu}(L^6))}+\| (\widetilde{a}^{\mu},\widetilde{b}^{\mu})\|_{L_t^2(L^6)}  \big)   \nonumber\\
&\quad\times \big(\|\nabla \{\mathbf{I-P}\}\widetilde{f}^\mu\|_{L_t^2(L^2_{v,\nu}(L^2))}+\| \nabla(\widetilde{a}^{\mu},\widetilde{b}^{\mu})\|_{L_t^2(L^2)}       \big) \nonumber\\
&+\big(1+ \| \varrho\|_{L_t^{\infty}(L^{\infty})}\big)\| \nabla{u}^{\mu}\|_{L_t^{\infty}(L^3)} 
\big(\| \{\mathbf{I-P}\}\widetilde{f}^\mu\|_{L_t^2(L^2_{v,\nu}(L^6))} +\| (\widetilde{a}^{\mu},\widetilde{b}^{\mu})\|_{L_t^2(L^6)}        \big) \nonumber\\
&\quad\times \big(\|\nabla \{\mathbf{I-P}\}\widetilde{f}^\mu\|_{L_t^2(L^2_{v,\nu}(L^2))} +\|\nabla(\widetilde{a}^{\mu},\widetilde{b}^{\mu})\|_{L_t^2(L^2)}        \big) \nonumber\\
&+\big(1+ \| \varrho\|_{L_t^{\infty}(L^{\infty})}\big)\|  {u}^{\mu}\|_{L_t^{\infty}(L^{\infty})}
 \big(\| \nabla \{\mathbf{I-P}\}\widetilde{f}^\mu\|_{L_t^2(L^2_{v,\nu}(L^2))} +\|\nabla (\widetilde{a}^{\mu},\widetilde{b}^{\mu})\|_{L_t^2(L^2)}        \big) \nonumber\\
&+\|\nabla\varrho\|_{L_t^{\infty}(L^{\infty})}\|  {u}^{\mu}\|_{L_t^{\infty}(L^{3})} 
\big(\| \{\mathbf{I-P}\}\widetilde{f}^\mu\|_{L_t^2(L^2_{v,\nu}(L^6))} +\| (\widetilde{a}^{\mu},\widetilde{b}^{\mu})\|_{L_t^2(L^6)}        \big) \nonumber\\
&\quad\times \big(\| \nabla\{\mathbf{I-P}\}\widetilde{f}^\mu\|_{L_t^2(L^2_{v,\nu}(L^2))} +\| \nabla(\widetilde{a}^{\mu},\widetilde{b}^{\mu})\|_{L_t^2(L^2)}        \big) \nonumber\\
&+\big(1+ \| \varrho\|_{L_t^{\infty}(L^{\infty})}\big)\|  \nabla{u}^{\mu}\|_{L_t^{\infty}(L^{\infty})} 
\big(\| \{\mathbf{I-P}\}\widetilde{f}^\mu\|_{L_t^2(L^2_{v,\nu}(H^1))}^2+\|\nabla(\widetilde{a}^{\mu},
\widetilde{b}^{\mu})\|_{L_t^2(L^2)}^2       \big) \nonumber\\
\lesssim\,&  \big(1+\| \varrho\|_{L_t^{\infty}(H^{3})}\big)\|u^{\mu}\|_{L_t^{\infty}(H^{3})}  
\big(\| \{\mathbf{I-P}\}\widetilde{f}^\mu\|_{L_t^2(L^2_{v,\nu}(H^1))}^2+\|\nabla(\widetilde{a}^{\mu},
\widetilde{b}^{\mu})\|_{L_t^2(L^2)}^2       \big) \nonumber\\
\lesssim\,&\big(1+\| \varrho\|_{L_t^{\infty}(H^{3})}\big)\|u^{\mu}\|_{L_t^{\infty}(H^{3})}\widetilde{\mathcal{X}}^\mu(t)\nonumber\\
\lesssim\,&\eps_0^{\frac{1}{2}}\widetilde{\mathcal{X}}^\mu(t),
\end{align}
and 
\begin{align}\label{NJK4.32}
|\widetilde{J}_5|\lesssim \|\nabla u^{\mu}\|_{L_t^{\infty}(L^{\infty})}\|\nabla\widetilde{u}^{\mu}\|_{L_t^2(L^2)}^2\lesssim \|u^{\mu}\|_{L_t^{\infty}(H^{3})}  \widetilde{\mathcal{X}}^\mu(t)\lesssim \eps_0^{\frac{1}{2}}\widetilde{\mathcal{X}}^\mu(t).  
\end{align}
For the term $\widetilde{J}_4$, similar to the estimate of $\widetilde{I}_3$ in \eqref{NJKH4.17}, one has
\begin{align}\label{NJK4.33}
|\widetilde{J}_4|\leq\,& C\mu \big(1+\|\nabla\varrho^{\mu}\|_{L_t^{\infty}(L^{\infty})}\big)\eps_0^{\frac{1}{2}}\int_0^t \|\nabla^2 u^\mu \|_{L^2} {\rm d}\tau \sup_{\tau\in[0,t]} \| \nabla\widetilde{u}^{\mu}\|_{L^2} \nonumber\\
&+C\mu \eps_0^{\frac{1}{2}} \int_0^t \|\nabla^3 u^\mu \|_{L^2} {\rm d}\tau \sup_{\tau\in[0,t]} \| \nabla\widetilde{u}^{\mu}\|_{L^2} \nonumber\\
\leq\,&  C\mu \eps_0^{\frac{1}{2}}\int_0^t (1+\tau)^{-\frac{s}{2}-\frac{1}{2}}{\rm d}\tau \,\sup_{\tau\in[0,t]} \| \nabla\widetilde{u}^{\mu}\|_{L^2}\nonumber\\
&\leq\, C\mu^2+\frac{1}{4}\sup_{\tau\in[0,t]} \| \nabla\widetilde{u}^{\mu}\|_{L^2}^2.   
\end{align}
For  the term $\widetilde{J}_6$, similar to the estimate of $\widetilde{I}_4$ in \eqref{NJKH4.18}, it holds
\begin{align}\label{NJK4.34}
|\widetilde{J}_6|\lesssim\, & \Big\|(1+\tau)^{-\frac{1}{2}}\Big(\frac{\varrho^\mu}{1+\varrho^\mu}
-\frac{\varrho}{1+\varrho}\Big)\Big\|_{L^{\infty}_t(L^6)}\big\|(1+\tau)^{\frac{1}{2}}\nabla^2 P^\mu\big\|_{L^2_t(L^3)}\|\nabla\widetilde{u}^\mu\|_{L^2_t(L^2)} \nonumber\\
&+ \big\|(1+\tau)^{-\frac{1}{2}}\nabla\widetilde{\varrho}^{\mu}\big\|_{L^{\infty}_t(L^2)}
\big\|(1+\tau)^{\frac{1}{2}}\nabla  P^\mu\big\|_{L^2_t(L^{\infty})}\|\nabla\widetilde{u}^\mu\|_{L^2_t(L^2)} 
\nonumber\\
&+\big(\|\nabla\varrho\|_{L_t^{\infty}(H^2)}+\|\nabla\varrho^{\mu}\|_{L_t^{\infty}(H^2)}\big) \big\|(1+\tau)^{-\frac{1}{2}} \widetilde{\varrho}^{\mu}\big\|_{L^{\infty}_t(L^6)}
\big\|(1+\tau)^{\frac{1}{2}}\nabla  P^\mu\|_{L^2_t(L^3)}\big\|\nabla\widetilde{u}^\mu\|_{L^2_t(L^2)} \nonumber\\
\lesssim\,&\big(1+\eps_0^{\frac{1}{2}}+\eps_1^{\frac{1}{2}}\big) \big\|(1+\tau)^{-\frac{1}{2}}\widetilde{\varrho}^\mu\big\|_{L^{\infty}_t(H^1)}
\big\|(1+\tau)^{\frac{1}{2}}\nabla P^\mu\big\|_{L^2_t(H^2)}\|\nabla\widetilde{u}^\mu\|_{L^2_t(L^2)} \nonumber\\
\lesssim\,&  \mu^2+ \big(\eps_0^{\frac{1}{2}}+\eps_1^{\frac{1}{2}}\big)\widetilde{\mathcal{X}}^\mu(t).      
\end{align}
Notice that for $|\alpha|=1$, owing to integration by parts and the fact   $ \partial^{\alpha}{\rm div}\,\widetilde{u}^{\mu}=0$, we obtain  
\begin{align*}
&\int_0^t\!\int_{\mathbb{R}^3}\partial^{\alpha}\Big(\frac{\varrho} {1+\varrho}\nabla\widetilde{P}^{\mu}\Big)\cdot\partial^{\alpha} \widetilde{u}^{\mu}{\rm d}x{\rm d\tau}\nonumber\\
=\,& \int_0^t\!\int_{\mathbb{R}^3}\partial^{\alpha}\Big(\frac{\varrho} {1+\varrho}\Big)\nabla\widetilde{P}^{\mu}\cdot\partial^{\alpha} \widetilde{u}^{\mu}{\rm d}x{\rm d\tau}
+\int_0^t\!\int_{\mathbb{R}^3}\frac{\varrho} {1+\varrho}\nabla\partial^{\alpha}\widetilde{P}^{\mu}\cdot\partial^{\alpha} \widetilde{u}^{\mu}{\rm d}x{\rm d\tau}\nonumber\\
=\,&\int_0^t\!\int_{\mathbb{R}^3}\partial^{\alpha}\Big(\frac{\varrho} {1+\varrho}\Big)\nabla\widetilde{P}^{\mu}\cdot\partial^{\alpha} \widetilde{u}^{\mu}{\rm d}x{\rm d\tau}-\int_0^t\!\int_{\mathbb{R}^3}\nabla\Big(\frac{\varrho} {1+\varrho}\Big) \partial^{\alpha}\widetilde{P}^{\mu}\cdot\partial^{\alpha} \widetilde{u}^{\mu}{\rm d}x{\rm d\tau} \nonumber\\
\lesssim\,& \|\varrho\|_{L_t^{\infty}(H^3)}\|\nabla \widetilde{P}^{\mu}\|_{L_{t}^2(L^2)}\|\nabla \widetilde{u}^{\mu}\|_{L_t^2(L^2)}\nonumber\\
\lesssim\,& (\eps_0^{\frac{1}{2}}+\eps_1^{\frac{1}{2}})\widetilde{\mathcal{X}}^\mu(t)+\mu^2. 
\end{align*}
Therefore, the term $\widetilde{J}_7+\widetilde{J}_8$  can be estimated by
\begin{align}\label{NJK4.35}
|\widetilde{J}_7+\widetilde{J}_8|\lesssim\,&  \int_0^t\big(\|\nabla\widetilde{F}_{3,1}\|_{L^2}\|\nabla\widetilde{u}^\mu\|_{L^{2}}+\| \nabla\widetilde{F}_4\|_{L^2_{x,v} } \|\nabla\widetilde{f}^\mu\|_{L^2_{x,v}}\big){\rm d}\tau \nonumber\\
&+\sum
_{|\alpha|=1}\int_{\mathbb{R}^3}\partial^{\alpha}\Big(\frac{\varrho} {1+\varrho}\nabla\widetilde{P}^{\mu}\Big)\cdot\partial^{\alpha} \widetilde{u}^{\mu}{\rm d}x  {\rm d}\tau \nonumber\\
\lesssim\, &\|\widetilde{F}_{3,1}\|_{L^2_t(\dot{H}^1)}\| \nabla\widetilde{u}^{\mu}\|_{L^2_t(L^2)}+\| \widetilde{F}_4\|_{L^2_t(L^2_v(\dot{H}^{1}))} \|\nabla\widetilde{f}^\mu\|_{L^2_t(L^2_{x,v})}
+\big(\eps_0^{\frac{1}{2}}
+\eps_1^{\frac{1}{2}}\big)\widetilde{\mathcal{X}}^\mu(t)+\mu^2\nonumber\\
\lesssim\,&\big(\|\widetilde{F}_{3,1}\|_{L^2_t(\dot{H}^1)}+\| \widetilde{F}_4\|_{L^2_t(L^2_v(\dot{H}^{1}))}\big)\widetilde{\mathcal{X}}^\mu(t)^{\frac{1}{2}}
+\big(\eps_0^{\frac{1}{2}}+\eps_1^{\frac{1}{2}}\big)\widetilde{\mathcal{X}}^\mu(t)+\mu^2.
\end{align}
Here, we decompose $\widetilde{F}_{3} =\widetilde{F}_{3,1}+\frac{\varrho }{1+\varrho }\nabla \widetilde{P}^\mu $.
By direct calculations, from \eqref{4.11}, \eqref{4.12} and \eqref{NJK4.21}, we have
\begin{align}\label{NJK4.36}
&\|\widetilde{F}_{3,1}\|_{L^2_t(\dot{H}^1)}\nonumber\\\lesssim \,&\|\nabla\widetilde{u}^{\mu}\|_{L_t^2(L^2)}\|\nabla u\|_{L_t^{\infty}(L^{\infty})}+\|\widetilde{u}^{\mu}\|_{L_t^2(L^6)}\|\nabla^2 u\|_{L_t^{\infty}(L^3)}\nonumber\\ 
&+\|\nabla a^{f}\|_{L_t^2(L^3)}\|\widetilde{u}^{\mu}\|_{L_t^{\infty}(L^6)}+\|  a^{f}\|_{L_t^2(L^6)}\|\nabla\widetilde{u}^{\mu}\|_{L_t^{\infty}(L^2)} +\|\nabla\varrho\|_{L_t^{\infty}(L^{\infty})}\|\nabla \widetilde{P}^{\mu}\|_{L_t^2(L^2)}\nonumber\\
\lesssim\,& (\|u\|_{L_t^{\infty}(H^3)}+\|f\|_{L_t^{\infty}(H_{x
,v}^3)})\widetilde{\mathcal{X}}^\mu(t)^{\frac{1}{2}}
+\big(1+\|\nabla\varrho\|_{L_t^{\infty}(L^{\infty})}\big)\|\varrho\|_{L_t^{\infty}(H^3)}\|\nabla \widetilde{P}^{\mu}\|_{L_t^2(L^2)}\nonumber\\
\lesssim\,& \eps_1^{\frac{1}{2}}\widetilde{\mathcal{X}}^\mu(t)^{\frac{1}{2}}+\eps_1^{\frac{1}{2}} \mu,
\end{align}
and
\begin{align}\label{NJK4.37}
 \| \widetilde{F}_4\|_{L^2_t(L^2_v(\dot{H}^{1}))} \lesssim\,&  \big\|(1+\tau)^{-\frac{1}{2}}\widetilde{\varrho}^{\mu}\big\|_{L_t^{\infty}(\dot{H}^1)} \big\|(1+\tau)^{\frac{1}{2}}\nabla(u^{\mu}\cdot\nabla_v f^{\mu})\big\|_{L_t^2(L_v^2(\dot{H}^{-1}))} \nonumber\\
&+ \big\|(1+\tau)^{-\frac{1}{2}}\widetilde{\varrho}^{\mu}\big\|_{L_t^{\infty}(\dot{H}^1)} \big\|(1+\tau)^{\frac{1}{2}}\nabla(u^{\mu}\cdot v f^{\mu})\big\|_{L_t^2(L_v^2(\dot{H}^{-1}))}  \nonumber\\
&+\big\|(1+\tau)^{-\frac{1}{2}}\nabla\varrho\|_{L_t^{\infty}(L^2)} \big\|(1+\tau)^{\frac{1}{2}} (u^{\mu}\cdot\nabla_v f^{\mu})\big\|_{L_t^2(L_{x,v}^2)}\nonumber\\
&+\big\|(1+\tau)^{-\frac{1}{2}}\nabla\varrho\big\|_{L_t^{\infty}(L^2)} 
\big\|(1+\tau)^{\frac{1}{2}} (u^{\mu}\cdot v f^{\mu})\big\|_{L_t^2(L_{x,v}^2)}\nonumber\\
&+\big(1+\|\varrho\|_{L_t^{\infty}(H^3)}\big)\big(\|\widetilde{u}^{\mu}\cdot\nabla_v f\|_{L_t^2(H^1)}+\|\widetilde{u}^{\mu}\cdot v f\|_{L_t^2(H^1)}\big)\nonumber\\
\lesssim\,&\big\|(1+\tau)^{-\frac{1}{2}}\widetilde{\varrho}^{\mu}\big\|_{L_t^{\infty}( {H}^1)} \big\|(1+\tau)^{\frac{1}{2}} (u^{\mu}\cdot\nabla_v f^{\mu})\big\|_{L_t^2(L_{x,v}^2)}\nonumber\\
&+\big\|(1+\tau)^{-\frac{1}{2}}\widetilde{\varrho}^{\mu}\big\|_{L_t^{\infty}( {H}^1)} \big\|(1+\tau)^{\frac{1}{2}} (u^{\mu}\cdot  v f^{\mu})\big\|_{L_t^2(L_{x,v}^2)}\nonumber\\
&+\|\widetilde{u}^{\mu}\|_{L_t^{\infty}(H^1)}\big( \|\nabla(a ,b )\|_{L_t^2(H^2)} +\|\{\mathbf{I-P}\}f \|_{L_t^2(L^2_{v,\nu}(H^3 ))}   \big) \nonumber\\
\lesssim\,&\big(\mu+\big(\eps_0^{\frac{1}{2}}
+\eps_1^{\frac{1}{2}}\big)\widetilde{\mathcal{X}}^\mu(t)^{\frac{1}{2}}\big)\|u^{\mu}\|_{L_t^{\infty}(H^3)}  \big\|(1+\tau)^{\frac{1}{2}}\nabla(a^{\mu} ,b^{\mu} )\big\|_{L_t^2(H^2)} \nonumber\\
&+\big(\mu+\big(\eps_0^{\frac{1}{2}}
+\eps_1^{\frac{1}{2}}\big)\widetilde{\mathcal{X}}^\mu(t)^{\frac{1}{2}}\big)
\|u^{\mu}\|_{L_t^{\infty}(H^3)}\big\|(1+\tau)^{\frac{1}{2}}\{\mathbf{I-P}\}f^{\mu} \big\|_{L_t^2(L^2_{v,\nu}(H^3 ))}\nonumber\\
&+\big(\eps_0^{\frac{1}{2}}+\eps_1^{\frac{1}{2}}\big)\widetilde{\mathcal{X}}^\mu(t)^{\frac{1}{2}}\nonumber\\
\lesssim\,&\big(\eps_0^{\frac{1}{2}}+\eps_1^{\frac{1}{2}}\big)\widetilde{\mathcal{X}}^\mu(t)^{\frac{1}{2}}
+\big(\eps_0^{\frac{1}{2}}+\eps_1^{\frac{1}{2}}\big)\mu.
\end{align}
Inserting \eqref{NJK4.36} and \eqref{NJK4.37} into \eqref{NJK4.35} gives rise to
\begin{align}\label{NJK4.38}
|\widetilde{J}_7+\widetilde{J}_8|\lesssim   \mu^2+ \big(\eps_0^{\frac{1}{2}}+\eps_1^{\frac{1}{2}}\big)\widetilde{\mathcal{X}}^\mu(t).  
\end{align}
For the last term $\widetilde{J}_9$, by using integration
by parts and the fact   $v\sqrt{M}\lesssim 1$, we obtain
\begin{align}\label{NJK4.39}
|\widetilde{J}_9|\lesssim\,& \big\|(1+\tau)^{-\frac{1}{2}}\widetilde{\varrho}^{\mu}\big\|_{L_t^{\infty}(H^1)}
\big( \big\|(1+\tau)^{\frac{1}{2}}\nabla(a^\mu,b^\mu)\big\|_{L_t^2(H^2)} \nonumber\\ &\quad +\big\|(1+\tau)^{\frac{1}{2}}\nabla\{\mathbf{I-P}\}f^\mu\big\|_{L_t^2(L^2_{v,\nu}(H^2 ))}    \big)\nonumber\\
&\times \big(\| \nabla \{\mathbf{I-P}\}\widetilde{f}^\mu\|_{L_t^2(L^2_{v,\nu}(L^2 ))}+\|\nabla(\widetilde{a}^{\mu},\widetilde{b}^{\mu})\|_{L_t^2(L^2)}^2\big)\nonumber\\
&+\big\|(1+\tau)^{-\frac{1}{2}}\widetilde{\varrho}^{\mu}\big\|_{L_t^{\infty}(H^1)}
\big\|(1+\tau)^{\frac{1}{2}}(u^{\mu}-b^{\mu})\big\|_{L_t^2(H^2)}\|\nabla(\widetilde{u}^{\mu},
\widetilde{b}^{\mu})\|_{L_t^2(L^2)}\nonumber\\
&+ \| \varrho\|_{L_t^{\infty}(L^{\infty})} \big(\| \{\mathbf{I-P}\}\widetilde{f}^\mu\|_{L_t^2(L^2_{v,\nu}(H^1 ))}^2+\|\widetilde{u}^{\mu}-\widetilde{b}^{\mu}\|_{L_t^2(H^1)}\|\nabla\widetilde{b}^{\mu}\|_{L_t^2(L^2)}\big)\nonumber\\
&+ \|\nabla \varrho\|_{L_t^{\infty}(L^{\infty})} \big(\| \{\mathbf{I-P}\}\widetilde{f}^\mu\|_{L_t^2(L^2_{v,\nu}(H^1 ))}^2+\|\widetilde{u}^{\mu}-\widetilde{b}^{\mu}\|_{L_t^2(H^1)}\|\nabla\widetilde{b}^{\mu}\|_{L_t^2(L^2)}\big)\nonumber\\
\lesssim\,&(\eps_0+\eps_1 )\big\|(1+\tau)^{-\frac{1}{2}}\widetilde{\varrho}^{\mu}\big\|_{L_t^{\infty}(H^1)}
\widetilde{\mathcal{X}}^\mu(t)^{\frac{1}{2}}\nonumber\\
&+(\eps_0+\eps_1)\big\|(1+\tau)^{-\frac{1}{2}}\widetilde{\varrho}^{\mu}\big\|_{L_t^{\infty}(H^1)}
\|\nabla(\widetilde{u}^{\mu},\widetilde{b}^{\mu})\|_{L_t^2(L^2)}
+\|\varrho\|_{L_t^{\infty}(H^3)} \widetilde{\mathcal{X}}^\mu(t)        \nonumber\\
\lesssim\,&\big(\eps_0^{\frac{1}{2}}+\eps_1^{\frac{1}{2}}\big)\widetilde{\mathcal{X}}^\mu(t)+\mu^2.    
\end{align}
By substituting \eqref{NJK4.30}--\eqref{NJK4.34}, \eqref{NJK4.38}
and \eqref{NJK4.39} into \eqref{NJK4.29},
summing them over $|\alpha|=1$,
we consequently derive the desired \eqref{NJK4.28}.
\end{proof}

\begin{lem}\label{L:error4}
It holds that
\begin{align}\label{nablatildeau}
\int_0^t \|\nabla(\widetilde{a}^\mu,\widetilde{b}^\mu)\|_{L^2}^2{\rm d}\tau\leq C\mu^2+C \widetilde{\mathcal{X}}^\mu(t),
\end{align}
where $C>0$ is a constant independent of $t$ and $\mu$.
\end{lem}
\begin{proof}
Similarly to \eqref{G3.36}, from \eqref{I-3mu} and \eqref{I-3*} we can derive 
the equations of $\widetilde{a}^{\mu} $ and $\widetilde{b}^{\mu}$:
\begin{equation}\label{G4.41}
\left\{\begin{aligned}
&\partial_{t}\widetilde{a}^{\mu}+{\rm div}\, \widetilde{b}^{\mu}=0,\\
&\partial_{t} \widetilde{b}^{\mu}_i+\partial_{i} \widetilde{a}^{\mu}+\sum_{j=1}^3\partial_j\Gamma_{ij} (\{\mathbf{I}-\mathbf{P}\}\widetilde{f}^{\mu})\\
&\quad=\widetilde{\varrho}^{\mu}(u_i^{\mu}-b_i^{\mu})+(1+\varrho)(\widetilde{u}^{\mu}_i-\widetilde{b}_i^{\mu})-\widetilde{\varrho}^{\mu}u_i^{\mu} a^{\mu}+(1+\varrho)(\widetilde{u}^{\mu}_ia^{\mu}+u_i\widetilde{a}^{\mu}),  \\
&\partial_{i}\widetilde{b}^{\mu}_j+\partial_j \widetilde{b}^{\mu}_i-\widetilde{\varrho}^{\mu}(u_i^{\mu}b_j^{\mu}+u_j^{\mu}b_i^{\mu})-(1+\varrho)(\widetilde{u}_i^{\mu}b^{\mu}_j+u_i\widetilde{b}_j^{\mu}+\widetilde{u}_j^{\mu}b^{\mu}_i+u_j\widetilde{b}_i^{\mu})\\
& \quad=-\partial_t \Gamma_{ij}(\{\mathbf{I}-\mathbf{P}\}\widetilde{f}^{\mu})+\Gamma_{ij}(\widetilde{\mathfrak{l}}^{\mu}
+\widetilde{\mathfrak{r}}^{\mu}+\widetilde{\mathfrak{s}}^{\mu}),  
 \end{aligned}
 \right.
\end{equation}
for $1\leq i,j\leq 3$, where $\widetilde{\mathfrak{l}}^{\mu}$, $\widetilde{\mathfrak{r}}^{\mu}$ and $\widetilde{\mathfrak{s}}^{\mu}$ are given by
\begin{align*} 
\widetilde{\mathfrak{l}}^{\mu}:=\,&\mathcal{L}\{\mathbf{I}-\mathbf{P}\}\widetilde{f}^{\mu}-v\cdot\nabla\{\mathbf{I}-\mathbf{P}\}\widetilde{f}^{\mu}, \nonumber \\
\widetilde{\mathfrak{r}}^{\mu}:=\,&-(\widetilde{u}^{\mu}\cdot\nabla_v\{\mathbf{I}-\mathbf{P}\}f^{\mu}+u\cdot\nabla_v\{\mathbf{I}-\mathbf{P}\}\widetilde{f}^{\mu})+\frac{1}{2}(\widetilde{u}^{\mu}\cdot v\{\mathbf{I}-\mathbf{P}\}f^{\mu}+u\cdot v\{\mathbf{I}-\mathbf{P}\}\widetilde{f}^{\mu}),  \nonumber\\
\widetilde{\mathfrak{s}}^{\mu}:=\,&\frac{\widetilde{\varrho}^{\mu}}{\sqrt{M}}\nabla_v\cdot (\nabla_v (\sqrt{M}\{\mathbf{I}-\mathbf{P}\}f^{\mu} )+v\sqrt{M}\{\mathbf{I}-\mathbf{P}\}f^{\mu}-u^{\mu}\sqrt{M}\{\mathbf{I}-\mathbf{P}\}f^{\mu}             )\nonumber\\
&+\frac{ {\varrho} }{\sqrt{M}}\nabla_v\cdot (\nabla_v (\sqrt{M}\{\mathbf{I}-\mathbf{P}\}\widetilde{f}^{\mu} )+v\sqrt{M}\{\mathbf{I}-\mathbf{P}\}\widetilde{f}^{\mu}           )\nonumber\\
&-\frac{ {\varrho} }{\sqrt{M}}\nabla_v\cdot ( \widetilde{u}^{\mu}\sqrt{M}\{\mathbf{I}-\mathbf{P}\}f^{\mu}  +u  \sqrt{M}\{\mathbf{I}-\mathbf{P}\}\widetilde{f}^{\mu}                  ).
\end{align*}

We first capture the dissipation of $\widetilde{a}^{\mu}$.
It follows from \eqref{G4.41}$_1$ and \eqref{G4.41}$_2$ that
\begin{align*}
&\frac{{\rm d}}{{\rm d}t}\sum_{i=1}^3\int_{\mathbb{R}^3} \partial_i \widetilde{a}^{\mu}  \widetilde{b}^{\mu}_i\mathrm{d}x+ \|  \nabla \widetilde{a}^{\mu}\|_{L^2}^2 -\| {\rm div}\,\widetilde{b}^{\mu}\|_{L^2}^2\nonumber\\
\,&\quad =\sum_{i=1}^3\int_{\mathbb{R}^3} \partial_i \widetilde{a}^{\mu} \bigg(-\sum_{j=1}^3\partial_j\Gamma_{ij}(\{\mathbf{I}-\mathbf{P}\}\widetilde{f}^{\mu}) 
+ \widetilde{\varrho}^{\mu} (u^{\mu}_i-b^{\mu}_i)+(1+\varrho)(\widetilde{u}^{\mu}_i-\widetilde{b}_i^{\mu})        \bigg)\mathrm{d}x\nonumber\\
&\quad \quad +\sum_{i=1}^3\int_{\mathbb{R}^3} \partial_i \widetilde{a}^{\mu}\big( -\widetilde{\varrho}^{\mu}u_i^{\mu}a^{\mu}+(1+\varrho)(\widetilde{u}^{\mu}_ia^{\mu}+u_i\widetilde{a}^{\mu})   \big){\rm d}x.    
\end{align*}
By direct calculations, the right-hand side of the above equality can be controlled as follows 
\begin{align*}
&\sum_{i=1}^3\int_{\mathbb{R}^3} \partial_i \widetilde{a}^{\mu} \bigg(-\sum_{j=1}^3\partial_j\Gamma_{ij}(\{\mathbf{I}-\mathbf{P}\}\widetilde{f}^{\mu}) 
+ \widetilde{\varrho}^{\mu} (u^{\mu}_i-b^{\mu}_i)+(1+\varrho)(\widetilde{u}^{\mu}_i-\widetilde{b}_i^{\mu})        \bigg)\mathrm{d}x\nonumber\\
&+\sum_{i=1}^3\int_{\mathbb{R}^3} \partial_i \widetilde{a}^{\mu}\big( -\widetilde{\varrho}^{\mu}u_i^{\mu}a^{\mu}+(1+\varrho)(\widetilde{u}^{\mu}_ia^{\mu}+u_i\widetilde{a}^{\mu})   \big)  {\rm d}x  \nonumber\\
\leq\,& \frac{1}{4}\|\nabla \widetilde{a}^{\mu}\|_{L^2}^2+C\|\nabla\{\mathbf{I}-\mathbf{P}\}\widetilde{f}^{\mu}\|_{L_{x,v}^2}^2+C(1+\|\varrho\|_{H^3})^2\|\widetilde{u}^{\mu}-\widetilde{b}^{\mu}\|_{L^2}^2\nonumber\\
&+C \big\|(1+t)^{-\frac{1}{2}}\widetilde{\varrho}^{\mu}\big\|_{H^1}^2  \big\|(1+t)^{-\frac{1}{2}}\nabla (u ^{\mu}, a^{\mu}, b^{\mu})\big\|_{H^{2}}^2\nonumber\\
&+C\big(1+\|\varrho\|_{H^3}\big)^2\big(\|\widetilde{u}^{\mu}\|_{H^1}^2\|\nabla a^{\mu}\|_{H^2}^2+\|u\|_{H^3}^2\|\widetilde{a}^{\mu}\|_{L^6}^2\big)\nonumber\\
\leq\,&\Big(\frac{1}{4}+C(\eps_0+\eps_1)\Big)\|\nabla \widetilde{a}^{\mu}\|_{L^2}^2+C\|\widetilde{u}^{\mu}-\widetilde{b}^{\mu}\|_{L^2}^2+C\| \{\mathbf{I}-\mathbf{P}\}\widetilde{f}^{\mu} \|_{L_v^2(H^1)}^2\nonumber\\
&+C\|\widetilde{u}^{\mu}\|_{H^1}^2\|\nabla a^{\mu}\|_{L^2}^2+C\|(1+t)^{-\frac{1}{2}}\widetilde{\varrho}^{\mu}\|_{H^1}^2  \|(1+t)^{-\frac{1}{2}}\nabla (u ^{\mu}, a^{\mu}, b^{\mu})\|_{H^{2}}^2.
\end{align*}
Therefore, we arrive at
\begin{align}\label{G4.54}
&-\frac{{\rm d}}{{\rm d}t} \int_{\mathbb{R}^3}\widetilde{a}^{\mu} {\rm div}\, \widetilde{b}^{\mu}\mathrm{d}x+\Big(\frac{3}{4}-C\eps_0\Big)\|\nabla \widetilde{a}^{\mu}\|_{L^2}^2\nonumber\\
\,&\quad \leq\|{\rm div}\, \widetilde{b}^{\mu}\|_{L^2}^2+C\|\widetilde{u}^{\mu}-\widetilde{b}^{\mu}\|_{L^2}^2+C\| \{\mathbf{I}-\mathbf{P}\}\widetilde{f}^{\mu} \|_{L_v^2(H^1)}^2\nonumber\\
&\qquad + C\|\widetilde{u}^{\mu}\|_{H^1}^2\|\nabla a^{\mu}\|_{L^2}^2+C\|(1+t)^{-\frac{1}{2}}\widetilde{\varrho}^{\mu}\|_{H^1}^2  \|(1+t)^{\frac{1}{2}}\nabla (u ^{\mu}, a^{\mu}, b^{\mu})\|_{H^{2}}^2.
\end{align}

Next step is to obtain the dissipation of $\widetilde{b}^\mu$. To this end, we deduce from  \eqref{G4.41}$_2$ and \eqref{G4.41}$_3$ that
\begin{align*} 
&\frac{{\rm d}}{{\rm d}t}\sum_{i,j=1}^3\int_{\mathbb{R}^3} (\partial_i\widetilde{b}^{\mu}_j+\partial_j\widetilde{b}^{\mu}_i) \Gamma_{ij}(\{\mathbf{I}-\mathbf{P}\}\widetilde{f}^{\mu})\mathrm{d}x+\sum_{i,j=1}^3\| \partial_i\widetilde{b}^{\mu}_j+\partial_j\widetilde{b}^{\mu}_i \|_{L^2}^2\nonumber\\
=\,&
\sum_{i,j=1}^3\int_{\mathbb{R}^3} (\partial_i\partial_t \widetilde{b}^{\mu}_j+\partial_j\partial_t \widetilde{b}^{\mu}_i)\partial^\alpha\Gamma_{ij}(\{\mathbf{I}-\mathbf{P}\}\widetilde{f}^{\mu})\mathrm{d}x\nonumber\\
&+\sum_{i,j=1}^3\int_{\mathbb{R}^3} (\partial_i\widetilde{b}^{\mu}_j+\partial_j\widetilde{b}^{\mu}_i) 
\big((1+\varrho)(\widetilde{u}_i^{\mu}b^{\mu}_j+u_i\widetilde{b}_j^{\mu}+\widetilde{u}_j^{\mu}b^{\mu}_i+u_j\widetilde{b}_i^{\mu})+\widetilde{\varrho}^{\mu}(u_i^{\mu}b_j^{\mu}+u_j^{\mu}b_i^{\mu}) \big)\mathrm{d}x\nonumber\\
&+\sum_{i,j=1}^3\int_{\mathbb{R}^3} (\partial_i\widetilde{b}^{\mu}_j+\partial_j\widetilde{b}^{\mu}_i)  \Gamma_{ij}(\widetilde{\mathfrak{l}}^{\mu}+\widetilde{\mathfrak{r}}^{\mu}+\widetilde{\mathfrak{s}}^{\mu})  \mathrm{d}x.
\end{align*}
By \eqref{uniform1} and Lemma \ref{LA.1}, it holds that
\begin{align*}
&\sum_{i,j=1}^3\int_{\mathbb{R}^3} (\partial_i\partial_t \widetilde{b}^{\mu}_j+\partial_j\partial_t \widetilde{b}^{\mu}_i)\partial^\alpha\Gamma_{ij}(\{\mathbf{I}-\mathbf{P}\}\widetilde{f}^{\mu})\mathrm{d}x\nonumber\\
=\,& 2\sum_{i,j=1}^3\int_{\mathbb{R}^3} \bigg(\partial_i \widetilde{a}^{\mu}+\sum_{m=1}^3\partial_m\Gamma_{im}(\{\mathbf{I}-\mathbf{P}\}\widetilde{f}^{\mu})       \bigg) \partial_j\Gamma_{ij}(\{\mathbf{I}-\mathbf{P}\}\widetilde{f}^{\mu})\mathrm{d}x\nonumber\\
&-2\sum_{i,j=1}^3\int_{\mathbb{R}^3} \big((1+\varrho)(\widetilde{u}^{\mu}_i-\widetilde{b}^{\mu}_i)+\widetilde{\varrho}(u_i^{\mu}-b_i^{\mu})  \big) \partial_j\Gamma_{ij}(\{\mathbf{I}-\mathbf{P}\}\widetilde{f}^{\mu})\mathrm{d}x\nonumber\\
&-2 \sum_{i,j=1}^3\int_{\mathbb{R}^3}    \big(-\widetilde{\varrho}^{\mu}u_i^{\mu} a^{\mu}+(1+\varrho)(\widetilde{u}^{\mu}_ia^{\mu}+u_i\widetilde{a}^{\mu})\big)\partial_j\Gamma_{ij}(\{\mathbf{I}-\mathbf{P}\}\widetilde{f}^{\mu})\mathrm{d}x\nonumber\\
\leq\,& \frac{1}{4}\|\nabla \widetilde{a}^{\mu}\|_{L^2}^2+C\|\nabla \{\mathbf{I}-\mathbf{P}\}f \|_{L_{x,v}^2}^2+C(1+\|\varrho\|_{H^{3}})^2\|\widetilde{u}-\widetilde{b}^{\mu}\|_{L^2}^2\nonumber\\
&+C\big\|(1+t)^{-\frac{1}{2}}\widetilde{\varrho}^{\mu}\big\|_{L^2}^2
\big\|(1+t)^{\frac{1}{2}}(u_i^{\mu}-b_i^{\mu})\big\|_{L^{\infty}}^2\nonumber\\
& +C\big\|(1+t)^{-\frac{1}{2}}\widetilde{\varrho}^{\mu}\big\|_{L^6}^2
\big\|(1+t)^{\frac{1}{2}}u_i^{\mu} \big\|_{L^{6}}^2\|a^{\mu}\|_{L^6}^2\nonumber\\
&+C(1+\|\varrho\|_{H^3})^2(\|\widetilde{u}^{\mu}\|_{L^2}^2\|  a^{\mu}\|_{L^{\infty}}^2+\|\widetilde{a}^{\mu}\|_{L^6}^2\|  u\|_{L^3}^2)\nonumber\\
\leq\,&\Big(\frac{1}{4}+C(\eps_0+\eps_1)\Big)\|\nabla \widetilde{a}^{\mu}\|_{L^2}^2+C\|\nabla \{\mathbf{I}-\mathbf{P}\}\widetilde{f}^{\mu} \|_{L_{x,v}^2 }^2+C\|\widetilde{u}^{\mu}-\widetilde{b}^{\mu}\|_{L^2}^2\nonumber\\
&+C(1+ \eps_0+\eps_1 )\big\|(1+t)^{-\frac{1}{2}}\widetilde{\varrho}^{\mu}\big\|_{H^1}^2
\big\|(1+t)^{\frac{1}{2}}\nabla(u ^{\mu},b ^{\mu})\big\|_{H^{1}}^2+C\|\nabla a^{\mu}\|_{L^2}^2\|\widetilde{u}^{\mu}\|_{L^2}^2,
\end{align*}
and
\begin{align*} 
&\sum_{i,j=1}^3\int_{\mathbb{R}^3} (\partial_i\widetilde{b}^{\mu}_j+\partial_j\widetilde{b}^{\mu}_i) 
\big((1+\varrho)(\widetilde{u}_i^{\mu}b^{\mu}_j+u_i\widetilde{b}_j^{\mu}+\widetilde{u}_j^{\mu}b^{\mu}_i+u_j\widetilde{b}_i^{\mu})+\widetilde{\varrho}^{\mu}(u_i^{\mu}b_j^{\mu}+u_j^{\mu}b_i^{\mu}) \big)\mathrm{d}x\nonumber\\
&\quad +\sum_{i,j=1}^3\int_{\mathbb{R}^3} (\partial_i\widetilde{b}^{\mu}_j+\partial_j\widetilde{b}^{\mu}_i)  \Gamma_{ij}(\widetilde{\mathfrak{l}}^{\mu}+\widetilde{\mathfrak{r}}^{\mu}+\widetilde{\mathfrak{s}}^{\mu})  \mathrm{d}x\nonumber\\
 &\qquad \leq\frac{1}{2}\sum_{i,j=1}^3\|\partial_i\widetilde{b}^{\mu}_j+\partial_j\widetilde{b}^{\mu}_i\|_{L^2}^2+C\sum_{i,j=1}^3
 \|  (1+\varrho)(\widetilde{u}_i^{\mu}b^{\mu}_j+u_i\widetilde{b}_j^{\mu}+\widetilde{u}_j^{\mu}b^{\mu}_i+u_j\widetilde{b}_i^{\mu})  \|_{L^2}^2\nonumber\\
&\qquad\quad  +C\sum_{i,j=1}^3\big(\|\partial^\alpha\Gamma_{ij}(\widetilde{\mathfrak{l}}^{\mu})\|_{L_{x,v}^2}^2
+\|\partial^\alpha\Gamma_{ij}(\widetilde{\mathfrak{r}}^{\mu})\|_{L_{x,v}^2}^2
+\|\partial^\alpha\Gamma_{ij}(\widetilde{\mathfrak{s}}^{\mu})\|_{L_{x,v}^2}^2                  \big)\nonumber\\
&\qquad \quad +C\sum_{i,j=1}^3\|\widetilde{\varrho}^{\mu}(u_i^{\mu}b_j^{\mu}+u_j^{\mu}b_i^{\mu})\|_{L^2}^2.
\end{align*}
Direct calculations indicate that 
\begin{align*}
&\sum_{i,j=1}^3 \|(1+\varrho)(\widetilde{u}_i^{\mu}b^{\mu}_j+u_i\widetilde{b}_j^{\mu}
+\widetilde{u}_j^{\mu}b^{\mu}_i+u_j\widetilde{b}_i^{\mu})  \|_{L^2}^2+\sum_{i,j=1}^3\|\widetilde{\varrho}^{\mu}(u_i^{\mu}b_j^{\mu}+u_j^{\mu}b_i^{\mu})\|_{L^2}^2\nonumber\\
&+\sum_{i,j=1}^3\big(\|\partial^\alpha\Gamma_{ij}(\widetilde{\mathfrak{l}}^{\mu})\|_{L_{x,v}^2}^2
+\|\partial^\alpha\Gamma_{ij}(\widetilde{\mathfrak{r}}^{\mu})\|_{L_{x,v}^2}^2
+\|\partial^\alpha\Gamma_{ij}(\widetilde{\mathfrak{s}}^{\mu})\|_{L_{x,v}^2}^2  \big) \nonumber\\
\lesssim\,&  (1+\|\varrho\|_{H^3} )^2\big(\|\widetilde{u}^{\mu}\|_{L^2}^2\|\nabla b^{\mu}\|_{H^1}^2+\|\widetilde{b}^{\mu}\|_{L^2}^2\|\nabla u^{\mu}\|_{H^1}^2 \big)\nonumber\\
&+\big\|(1+t)^{-\frac{1}{2}}\widetilde{\varrho}^{\mu}\big\|_{L^6}^2
\big\|(1+t)^{\frac{1}{2}}u ^{\mu} \big\|_{L^{6}}^2\|b^{\mu}\|_{L^6}^2+
\big(1+\|u\|_{H^3}^2\big)\|\{\mathbf{I}-\mathbf{P}\}\widetilde{f}^{\mu}\|_{L_v^2(H^1)}^2\nonumber\\
&+\|\widetilde{u}^{\mu}\|_{L^6}^2\|\{\mathbf{I}-\mathbf{P}\} {f}^{\mu}\|_{L_v^2(H^3)}^2 +\big(1+\|u^{\mu}\|_{H^3}^2\big)\big\|(1+t)^{-\frac{1}{2}}\widetilde{\varrho}^{\mu}\big\|_{L^6}^2
\big\|(1+t)^{\frac{1}{2}}\{\mathbf{I}-\mathbf{P}\} {f}^{\mu}\big\|_{L_v^2(H^3)}^2\nonumber\\
&+\big(1+\|u\|_{H^3}^2\big)\|\varrho\|_{L^{\infty}}^2
\|\{\mathbf{I}-\mathbf{P}\}\widetilde{f}^{\mu}\|_{L_v^2(H^1)}^2
+\|\varrho\|_{L^{\infty}}^2\|\{\mathbf{I}-\mathbf{P}\} {f}^{\mu}\|_{L_v^2(H^3)}^2\|\widetilde{u}^{\mu}\|_{H^1}^2\nonumber\\
\lesssim\,&\|(\widetilde{u}^{\mu},\widetilde{b}^{\mu})\|_{H^1}^2\big(\|\nabla( {u}^{\mu}, {b}^{\mu})\|_{H^1}^2+\|\{{\mathbf{I}-\mathbf{P}\} {f}^{\mu}\|_{L_v^2(H^3)}^2}\big)+\|\{\mathbf{I}-\mathbf{P}\}\widetilde{f}^{\mu}\|_{L_v^2(H^1)}^2\nonumber\\
&+\big\|(1+t)^{-\frac{1}{2}}\widetilde{\varrho}^{\mu}\big\|_{H^1}^2 
\big(\big\|(1+t)^{\frac{1}{2}}\nabla (u ^{\mu} ,b^{\mu})\big\|_{H^{2}}^2+\big\|(1+t)^{\frac{1}{2}}\{{\mathbf{I}-\mathbf{P}\} {f}^{\mu}\big\|_{L_v^2(H^3)}^2} \big). 
\end{align*}
The above inequality, together with 
$$\sum_{i,j=1}^3\| \partial_i\widetilde{b}^{\mu}_j+\partial_j\widetilde{b}^{\mu}_i \|_{L^2}^2=2\|\nabla  \widetilde{b}^{\mu}\|_{L^2}^2+2\|{\rm div}\,  \widetilde{b}^{\mu}\|_{L^2}^2,$$ 
implies that 
\begin{align}\label{G4.50}
&\frac{{\rm d}}{{\rm d}t} \sum_{i,j=1}^3\int_{\mathbb{R}^3}(\partial_i\widetilde{b}^{\mu}_j+\partial_j\widetilde{b}^{\mu}_i) \Gamma_{ij}(\{\mathbf{I}-\mathbf{P}\}\widetilde{f}^{\mu})\mathrm{d}x+\|\nabla \widetilde{b}^{\mu}\|_{L^2}^2+\|{\rm div}\, \widetilde{b}^{\mu}\|_{L^2}^2\nonumber\\
\,&\quad \leq \Big(\frac{1}{4}+C(\eps_0+\eps_1)\Big)\|\nabla \widetilde{a}^{\mu}\|_{L^2}^2+C\|\widetilde{u}^{\mu}-\widetilde{b}^{\mu}\|_{L^2}^2+C\| \{\mathbf{I}-\mathbf{P}\}\widetilde{f}^{\mu} \|_{L_v^2(H^1)}^2\nonumber\\
&\qquad +C\|(\widetilde{u}^{\mu},\widetilde{b}^{\mu})\|_{H^1}^2\big(\|\nabla( {u}^{\mu}, a^{\mu},{b}^{\mu})\|_{H^2}^2+\|\{{\mathbf{I}-\mathbf{P}\} {f}^{\mu}\|_{L_v^2(H^3)}^2}\big)\nonumber\\
&\qquad +C\big\|(1+t)^{-\frac{1}{2}}\widetilde{\varrho}^{\mu}\big\|_{H^1}^2 \big(\big\|(1+t)^{\frac{1}{2}}\nabla (u ^{\mu} ,b^{\mu})\big\|_{H^{2}}^2+\big\|(1+t)^{\frac{1}{2}}\{{\mathbf{I}-\mathbf{P}\} {f}^{\mu}\big\|_{L_v^2(H^3)}^2} \big).   
\end{align}

We now introduce the new temporal functional $\widetilde{\mathcal{E}}^{\mu}_0(t)$:
\begin{align*} 
\widetilde{\mathcal{E}}^{\mu}_0(t):= \sum_{i,j=1}^3\int_{\mathbb{R}^3} (\partial_i\widetilde{b}^{\mu}_j+\partial_j\widetilde{b}^{\mu}_i) \Gamma_{ij}(\{\mathbf{I}-\mathbf{P}\}\widetilde{f}^{\mu})\mathrm{d}x-  \int_{\mathbb{R}^3}\widetilde{a}^{\mu} {\rm div}\, \widetilde{b}^{\mu}\mathrm{d}x. 
\end{align*}
The summation  of  \eqref{G4.54} and \eqref{G4.50} results in
\begin{align*}
& \frac{{\rm d}}{{\rm d}t}\widetilde{\mathcal{E}}^{\mu}_0(t)+\|\nabla \widetilde{b}^{\mu}\|_{L^2}^2+\Big(\frac{1}{2}-C(\eps_0+\eps_1)  \Big)\|\nabla \widetilde{a}^{\mu}\|_{L^2}^2\nonumber\\
\,& \quad \leq  C\|(\widetilde{u}^{\mu},\widetilde{b}^{\mu})\|_{H^1}^2\big(\|\nabla( {u}^{\mu}, a^{\mu},{b}^{\mu})\|_{H^2}^2+\|\{{\mathbf{I}-\mathbf{P}\} {f}^{\mu}\|_{L_v^2(H^3)}^2}\big)\nonumber\\
&\qquad +C\big\|(1+t)^{-\frac{1}{2}}\widetilde{\varrho}^{\mu}\big\|_{H^1}^2 
\big(\big\|(1+t)^{\frac{1}{2}}\nabla (u ^{\mu} ,b^{\mu})\big\|_{H^{2}}^2+\big\|(1+t)^{\frac{1}{2}}\{{\mathbf{I}-\mathbf{P}\} {f}^{\mu}\big\|_{L_v^2(H^3)}^2} \big)\nonumber\\  
&\qquad +C\| \{\mathbf{I}-\mathbf{P}\}\widetilde{f}^{\mu} \|_{L_v^2(H^1)}^2+C\|\widetilde{u}^{\mu}-\widetilde{b}^{\mu}\|_{L^2}^2,
\end{align*}
which implies that 
\begin{align}\label{G4.55}
&\widetilde{\mathcal{E}}^{\mu}_0(t) +\int_0^t \|\nabla(\widetilde{a}^{\mu},\widetilde{b}^{\mu})(\tau)\|_{L^2}^2{\rm d}\tau\nonumber\\
\leq\,&
C\|(\widetilde{u}^{\mu},\widetilde{b}^{\mu})\|_{L_t^{\infty}(H^1)}^2\big(\|\nabla( {u}^{\mu}, a^{\mu},{b}^{\mu})\|_{L_t^2(H^2)}^2+\|\{{\mathbf{I}-\mathbf{P}\} {f}^{\mu}\|_{L_t^2(L_v^2(H^3))}^2}\big)+\widetilde{\mathcal{E}}^{\mu}_0(0)\nonumber\\
&+C\big\|(1+t)^{-\frac{1}{2}}\widetilde{\varrho}^{\mu}\big\|_{L_t^{\infty}(H^1)}^2 \big(\big\|(1+t)^{\frac{1}{2}}\nabla (u ^{\mu} ,b^{\mu})\big\|_{L_t^2(H^{2})}^2+\big\|(1+t)^{\frac{1}{2}}\{{\mathbf{I}-\mathbf{P}\} {f}^{\mu}\big\|_{L_t^2(L_v^2(H^3))}^2} \big)\nonumber\\
&+C\| \{\mathbf{I}-\mathbf{P}\}\widetilde{f}^{\mu} \|_{L_t^2(L_v^2(H^1))}^2+C\|\widetilde{u}^{\mu}-\widetilde{b}^{\mu}\|_{L_t^2(L^2)}^2\nonumber\\
\leq\,& C(\eps_0+\eps_1) \widetilde{\mathcal{X}}^\mu(t)+C \mu^2+C\| \{\mathbf{I}-\mathbf{P}\}\widetilde{f}^{\mu} \|_{L_t^2(L_v^2(H^1))}^2+C\|\widetilde{u}^{\mu}-\widetilde{b}^{\mu}\|_{L_t^2(L^2)}^2+\widetilde{\mathcal{E}}^{\mu}_0(0).
\end{align}
Since $|\widetilde{\mathcal{E}}^{\mu}_0(t)|\lesssim   \| \widetilde{f}^{\mu}\|_{{L}_{v}^2(H^1)}^2$, 
combining \eqref{4.19}, \eqref{NJK4.28} and \eqref{G4.55}, we justify \eqref{nablatildeau} and finish the proof of Lemma \ref{L:error4}.
\end{proof}

\smallskip
\begin{proof}[Proof of Theorem \ref{T1.3}]  We deduce from Lemmas \ref{L:error1}--\ref{L:error4} that  
\begin{align*}
\sup_{t\geq0}\big( (1+t)^{-1}\|\widetilde{\varrho}^\mu(t)\|_{H^1}^2\big)
+\widetilde{\mathcal{X}}^\mu(t)\leq C\mu^2+ C\big(\eps_0 ^{\frac{1}{2}}+\eps_1^{\frac{1}{2}}\big)\widetilde{\mathcal{X}}^\mu(t).
\end{align*}
As $\eps_0$ and $\eps_1$ are sufficiently small, it follows that
\begin{align*}
\sup_{t\geq0}\big( (1+t)^{-1}\|\widetilde{\varrho}^\mu(t)\|_{H^1}^2\big)+\widetilde{\mathcal{X}}^\mu(t)\leq C\mu^2.    
\end{align*}
Consequently, we derive the desired \eqref{rate}. Thus, the proof of Theorem \ref{T1.3} is completed. 
\end{proof}

\appendix
\section{Analytic Tools}

In this appendix, we present some basic facts which have been used frequently in this paper. 

\begin{lem} [{\!\!{\cite[Lemma 2.1]{CDM-krm-2011}}} and {\cite[Lemmas 2.1--2.2]{Dk-MZ-1992}}]
\label{LA.1}   
There exists  a positive constant C, 
such that for any $g,h\in H^3(\mathbb{R}^3)$ and any multi-index $\alpha$  with $1\leq|\alpha|\leq3$, it holds
\begin{align*}
\|g\|_{L^{\infty}(\mathbb{R}^{3})} \leq\,& C\|\nabla g\|_{L^{2}(\mathbb{R}^{3})}^{\frac{1}{2}}
\|\nabla^{2}g\|_{L^{2}(\mathbb{R}^{3})}^{\frac{1}{2}}, \\
\|gh\|_{H^{2}(\mathbb{R}^{3})} \leq\,& C\|g\|_{H^{2}(\mathbb{R}^{3})}\|\nabla h\|_{H^{2}(\mathbb{R}^{3})}, \\
\|\partial^{\alpha}(gh)\|_{L^{2}(\mathbb{R}^{3})}
 \leq\,& C\|\nabla  g\|_{H^{2}(\mathbb{R}^{3})}\|\nabla  h\|_{H^{2}(\mathbb{R}^{3})},\\
\|g\|_{L^6(\mathbb{R}^3)} \leq\,& C\|\nabla  g\|_{L^2(\mathbb{R}^3)}\leq C\|g\|_{H^1(\mathbb{R}^3)},\\
 \|g\|_{L^q(\mathbb{R}^3)} \leq\,& C\|g\|_{H^1(\mathbb{R}^3)}, 
\end{align*}
for some $2\leq q\leq 6$.
\end{lem}

\begin{lem}[{\!\!\cite[Appendix]{commutator1}}]\label{LA.2}
Let $m\geq 1$ be an integer and define the commutator 
\begin{align*}
[\nabla^m, g]h=\nabla^m(gh)-g\nabla^mh.    
\end{align*}
Then for $r\in (1,\infty)$, we have 
\begin{align*}
\|[\nabla^m, g]h\|_{L^r}\lesssim \|\nabla g\|_{L^{\infty}}\|\nabla^{m-1}h\|_{L^{r}}+\|\nabla^m g\|_{L^{r}}\|h\|_{L^{\infty}}.   
\end{align*}
\end{lem}

\begin{lem}[{\!\!\cite[Lemma A.1]{GW-CPDE-2012}}]\label{LA.3}
Let  $2\leq p\leq+\infty$ and  $0\leq m,k\leq l$; when $p=\infty$ we require further that $m\leq k+1$ and $l\geq k+2$. Then we have,
for any $g\in C_{0}^{\infty}(\mathbb{R}^3)$,
\begin{align*}
\|\nabla^k g\|_{L^p}\lesssim \|\nabla^m g\|_{L^2}^{1-\theta}\|\nabla^l g\|_{L^2}^\theta,    
\end{align*}
where $0\leq\theta\leq1$ and $k$ satisfy
\begin{align*}
\frac{k}{3}-\frac{1}{p}=\Big(\frac{m}{3}-\frac{1}{2}\Big)(1-\theta)+\Big(\frac{l}{3}-\frac{1}{2}\Big)\theta. 
\end{align*}
\end{lem}

\begin{lem}[{\!\!\cite[Lemma 4.6]{SS-adv-2014}}]\label{LA.4}
Suppose that $s>0$ and $1\leq p<2$. We have the embedding $L^p\subset\dot{B}^{-s}_{2,\infty}$ with $\frac{1}{2}+\frac{s}{3}=\frac{1}{p}$. In particular, we have the estimate
\begin{align*}
\|g\|_{\dot{B}^{-s}_{2,\infty}}\lesssim \|g\|_{L^p}.    
\end{align*}
\end{lem}

\begin{lem}[{\!\!\cite[Lemma 4.5]{SS-adv-2014}}]\label{LA.5}
Suppose $m\geq 0$ and $s>0$, then we have    
\begin{align*}
\|\nabla^m g\|_{L^2}\lesssim\|\nabla^{m+1}g\|_{L^2}^{1-\theta}\|g\|_{\dot{B}^{-s}_{2,\infty}}^{\theta},    
\end{align*}
where $\theta=\frac{1}{m+1+s}$.
\end{lem}

\begin{lem}[\!\!{\cite[Section 2]{BCD-Book-2011}}]\label{LA.7}
The following statements hold:

$(i)$ Let $s>0$ and $1\leq p,r\leq \infty$. Then $\dot{B}^{s}_{p,r} \cap L^{\infty}$ is an algebra and
\begin{align*}
\|gh\|_{\dot{B}^{s}_{p,r}}\lesssim\|g\|_{L^{\infty}}\|h\|_{\dot{B}^{s}_{p,r}}+\|h\|_{L^{\infty}}\|g\|_{\dot{B}^{s}_{p,r}}.  
\end{align*}

$(ii)$ Let $s_1$, $s_2$ and $p$ satisfy $2\leq p \leq\infty$, $s_1\leq \frac{3}{p}$, $s_2\leq \frac{3}{p}$ and $s_1+s_2>0$. Then we have
\begin{align*}
\|gh\|_{\dot{B}^{s_1+s_2-\frac{3}{p}}_{p,1}}\lesssim\|g\|_{\dot{B}^{s_1}_{p,1}}\|h\|_{\dot{B}^{s_2}_{p,1}}.  \end{align*}

$(iii)$ Assume that $s_1$, $s_2$, $p$ and $r$ satisfy $2\leq p\leq\infty$ and
$1\leq r\leq \infty$,
$-\frac{3}{p}<s_1<\frac{3}{p}$, $-\frac{3}{p}< s_2<\frac{3}{p}$, $s_1+s_2>0$, or $r=\infty$, $-\frac{3}{p}\leq s_1<\frac{3}{p}$, $-\frac{3}{p}< s_2\leq \frac{3}{p}$, $s_1+s_2\geq 0$. Then it holds
\begin{align*}
\|gh\|_{\dot{B}^{s_1+s_2-\frac{3}{p}}_{p,r}}\lesssim\|g\|_{\dot{B}^{s_1}_{p,r}}\|h\|_{\dot{B}^{s_2}_{p,1}}.  \end{align*}

$(iv)$ The following real interpolation property is satisfied:
 for $1\leq p\leq \infty$, $s_1<s_2$ and $\theta\in(0,1)$, 
\begin{align*}
\|g\|_{\dot{B}_{p,1}^{\theta s_1+(1-\theta)s_2}}\lesssim\frac{1}{\theta(1-\theta)(s_2-s_1)}\|g\|_{\dot{B}_{p,\infty}^{ s_1}}^{\theta}\|g\|_{\dot{B}_{p,\infty}^{ s_2}}^{1-\theta}.  
\end{align*}
In particular, since $  \dot{H}^s=\dot{B}^{s}_{2,2}\hookrightarrow\dot{B}^{s}_{2,\infty}$, we have
\begin{align*}
\|g\|_{\dot{B}_{2,1}^{\theta s_1+(1-\theta)s_2}}\lesssim\frac{1}{\theta(1-\theta)(s_2-s_1)}\|g\|_{\dot{H}^{ s_1}}^{\theta}\|g\|_{\dot{H}^{ s_2}}^{1-\theta}.  
\end{align*}

\end{lem}

\begin{lem}[\!\!{\cite[Section 3]{BCD-Book-2011}}]\label{LA.8}
Let $1\leq p\leq\infty$ and $-\frac{3}{p}\leq s< 1+\frac{3}{p}$.
Then it holds
\begin{align*}
\sum_{j\in\mathbb{Z}}2^{js}\|[\dot{\Delta}_j,u\cdot\nabla]\varrho\|_{L^p}\lesssim\,&\|u\|_{\dot{B}^{\frac{3}{p}+1}_{p,1}}\|\varrho\|_{\dot{B}^{s}_{p,\infty}}.
\end{align*}
\end{lem}

\bigskip 

\section*{Acknowledgements}
F. Li was supported by NSFC (Grant No.  12331007) and the ``333 Project" of Jiangsu Province.
J.  Ni  was supported by NSFC (Grant No.  12331007).  
L.-Y. Shou was supported by NSFC (Grant No. 12301275).
D. Wang was supported in part by NSF grants DMS-2219384 and DMS-2510532.

%
%
%
%
%
%
%

\bibliographystyle{plain}

\begin{thebibliography}{aaa}


\bibitem{AB-mmas-1998} O. Anoshchenko, A. Boutet de Monvel-Berthier, 
The existence of the global generalized solution of the system of equations describing suspension motion,
{\it Math. Methods Appl. Sci.} {\bf 20 (6)} (1997)  495--519.


\bibitem{BCD-Book-2011} H. Bahouri, J. Chemin, R. Danchin,
{\it Fourier Analysis and Nonlinear Partial Differential Equations,}
Grundlehren Math. Wiss, 343,
Springer, Heidelberg, 2011.


\bibitem{BBBDLLT-irma-2005} C. Baranger, G. Baudin, L. Boudin, 
B. Despr\'{e}s, F. Lagouti\`{e}re, E. Lapébie, T. Takahashi,
Liquid jet generation and break-up, 
Numerical methods for hyperbolic and kinetic problems, 
{\it IRMA Lect. Math. Theor. Phys.} {\bf 7} (2005) 149--176.


\bibitem{BBJM-esaim-2005} C. Baranger, L. Boudin, P. Jabin, S. Mancini,
A modeling of biospray for the upper airways, CEMRACS 2004 --- {\it Mathematics and Applications to Biology and Medicine}, 41--47, ESAIM Proc.,  14, EDP Sciences, Les Ulis, 2005.

\bibitem{BD-JHDE-2006} C. Baranger, L. Desvillettes,  
Coupling Euler and Vlasov equations in the context of sprays: the local-in-time, classical solutions,
{\it J. Hyperbolic Differ. Equ.} {\bf 3 (1)} (2006)  1--26.


\bibitem{BDGM-DIE-2009} L. Boudin, L. Desvillettes, C. Grandmont, A. Moussa, 
Global existence of solutions for the coupled Vlasov and Navier-Stokes equations,
{\it Differential Integral Equations} {\bf 22 (11--12)} (2009) 1247--1271.

\bibitem{BGM-JDE-2017} L. Boudin, C. Grandmont, A. Moussa, 
Global existence of solutions to the incompressible Navier-Stokes-Vlasov equations in a time-dependent domain,
{\it J. Differential Equations}  {\bf 262 (3)} (2017) 1317--1340.


\bibitem {brandolesehand} L. Brandolese, M. E. Schonbek, Large time behavior of the Navier-Stokes flow, {\it Handbook of Mathematical Analysis in Mechanics of Viscous Fluids}, 579--645, Edited by Y. Giga   and A. Novotn\'{y},  Springer, Cham, 2018. 


\bibitem{CJ-SIAM-2021} W. Cao,  P. Jiang, 
Global bounded weak entropy solutions to the Euler-Vlasov equations in fluid-particle system,
{\it SIAM J. Math. Anal.} {\bf 53 (4)} (2021)  3958--3984.


\bibitem{CG-CPDE-2006} J.-A. Carrillo, T. Goudon, 
Stability and asymptotic analysis of a fluid-particle interaction model,
{\it Comm. Partial Differential Equations} {\bf 31 (7--9)} (2006) 1349--1379.

\bibitem{CGL-JCP-2008} J.-A. Carrillo, T. Goudon, P. Lafitte,
Simulation of fluid and particles flows: asymptotic preserving schemes for bubbling and flowing regimes,
{\it J. Comput. Phys.}   {\bf 227 (16)} (2008) 7929--7951.

\bibitem{CDM-krm-2011} J.-A. Carrillo, R. Duan, A. Moussa, Global classical solutions close to equilibrium 
to the Vlasov-Fokker-Planck-Euler system, {\it Kinet. Relat. Models} {\bf 4} (2011) 227--258.

\bibitem{CKL-JDE-2011} M. Chae, K. Kang, J. Lee, 
Global existence of weak and classical solutions for the Navier-Stokes-Vlasov-Fokker-Planck equations,
{\it J. Differential Equations} {\bf 251 (9)} (2011) 2431--2465.

\bibitem{CKL-JHDE-2013} M. Chae, K. Kang, J. Lee,
Global classical solutions for a compressible fluid-particle interaction model,
{\it J. Hyperbolic Differ. Equ.} {\bf 10 (3)} (2013) 537--562.




\bibitem{CK-Nonlinearity-2015}  Y.-P. Choi, B. Kwon, 
Global well-posedness and large-time behavior for the inhomogeneous Vlasov-Navier-Stokes equations,
{\it Nonlinearity}  {\bf 28 (9)} (2015)  3309--3336.

\bibitem {Const:invisid} P. Constantin, Note on loss of regularity for solutions of the 3-D incompressible Euler and related
equations,  {\it Commun. Math. Phys.} {\bf 104 (2)} (1986) 311--326.

\bibitem {Const:invisid1} P. Constantin, J. Wu, The inviscid limit for non-smooth vorticity,   {\it Indiana Univ. Math. J.} {\bf 45 (1)} (1996) 67--81.

\bibitem{Danchin-2024-kato} R. Danchin, Fujita-Kato solutions and optimal time decay for the Vlasov-Navier-Stokes system in the whole space, arXiv:2405.09937.

\bibitem{Dk-MZ-1992} K. Deckelnick, Decay estimates for the compressible Navier-Stokes equations in unbounded domains,{\it Math. Z.}    {\bf 209 (1)} (1992) 115--130.

\bibitem{DF-jmp-2010} R. Duan, M. Fornasier, G. Toscani, A kinetic flocking model with diffusion, {\it Comm. Math. Phys} 
{\bf 300} (2010) 95--145.

\bibitem{DLUY-2007-JDE} R. Duan, H. Liu, S. Ukai, T. Yang,
Optimal $L^p$-$L^q$ convergence rates for the compressible Navier-Stokes equations with potential force,
{\it J. Differential Equations}  {\bf 238 (1)} (2007) 220--233.

\bibitem{DL-krm-2013} R. Duan, S. Liu, Cauchy problem on the Vlasov-Fokker-Planck equation coupled with the compressible Euler equations through the friction force,
{\it Kinet. Relat. Models}  {\bf 6} (2013) 687--700.



\bibitem{EDM-Nonlinearity-2021} L. Ertzbischoff, D. Han-Kwan, A. Moussa,
Concentration versus absorption for the Vlasov-Navier-Stokes system on bounded domains,
{\it Nonlinearity} {\bf 34 (10)} (2021) 6843--6900.

\bibitem{GHM-ARMA-2018} O. Glass, D. Han-Kwan, A. Moussa,
The Vlasov-Navier-Stokes system in a 2D pipe: existence and stability of regular equilibria,
{\it Arch. Ration. Mech. Anal.} {\bf 230 (2)} (2018) 593--639.

\bibitem{GHMZ-2010-SIMA} T. Goudon, L. He, A. Moussa, P. Zhang, 
The Navier-Stokes-Vlasov-Fokker-Planck system near equilibrium,
{\it SIAM J. Math. Anal.} {\bf 42 (5)} (2010) 2177--2202.

\bibitem{GJV-2004-2-IUMJ} T. Goudon, P.-E. Jabin, A. Vasseur, 
Hydrodynamic limit for the Vlasov-Navier-Stokes equations. I. Light particles regime,
{\it Indiana Univ. Math. J.} {\bf 53 (6)} (2004) 1495--1515.

\bibitem{GJV-IUMJ-2004} T. Goudon, P.-E. Jabin, A. Vasseur, 
Hydrodynamic limit for the Vlasov-Navier-Stokes equations. II. Fine particles regime,
{\it Indiana Univ. Math. J.} {\bf 53 (6)} (2004)  1517--1536.

\bibitem{Gy-IUMJ-2004} Y. Guo, 
The Boltzmann equation in the whole space,
{\it Indiana Univ. Math. J.}   {\bf 53 (4)} (2004) 1081--1094.

\bibitem{GW-CPDE-2012} Y. Guo, Y. Wang, 
Decay of dissipative equations and negative Sobolev spaces,
{\it Comm. Partial Differential Equations}   {\bf 37 (12)} (2012) 2165--2208.


\bibitem{Hk-JJIAM-1998} K. Hamdache,
Global existence and large time behaviour of solutions for the Vlasov-Stokes equations,
{\it Japan J. Indust. Appl. Math.}  {\bf 15 (1)} (1998) 51--74.

\bibitem{Han-Kwanwhole} D. Han-Kwan, Large-time behavior of small-data solutions to the Vlasov-Navier-Stokes system on the whole space,  {\it Probab. Math. Phys.} {\bf 3 (1)} (2022) 35--67.

\bibitem{HM-MAMS-2024} D. Han-Kwan, D. Michel, 
On hydrodynamic limits of the Vlasov-Navier-Stokes system,
{\it Mem. Amer. Math. Soc.} {\bf302 (1516)} (2024) 115 pp.


\bibitem{Han-Kwan:2020} D. Han-Kwan, A. Moussa, I. Moyano, Large time behavior of the
Vlasov–Navier–Stokes system on the torus,  {\it Arch. Rational Mech. Anal. } {\bf 236} (2020) 1273--1323.





\bibitem {Kato:inviscid} T. Kato, Nonstationary flows of viscous and ideal fluids in $\mathbb{R}^3$,  {\it J. Funct. Anal.} {\bf 9 (3)} (1972) 296--305.

\bibitem {commutator1} T. Kato, G. Ponce,  Commutator estimates and the Euler and Navier–Stokes equations,  {\it Commun. Pure Appl. Math.} {\bf 41} (1988) 891--907.


\bibitem{LMW-SIAM-2017} F. Li, Y. Mu, D. Wang, 
Strong solutions to the compressible Navier-Stokes-Vlasov-Fokker-Planck equations: global existence near the equilibrium and large time behavior,
{\it SIAM J. Math. Anal.}   {\bf 49 (2)} (2017) 984--1026.

\bibitem{LNW-2025-preprint} F. Li, J. Ni, M. Wu, 
Global strong solutions to a compressible fluid-particle interaction model with density-dependent friction force, arXiv:2502.19886.

\bibitem{LNW-24-1} F. Li, J. Ni,  M. Wu,
Global existence and large time behavior of classical solutions to the incompressible inhomogeneous kinetic-fluid model with energy exchanges, {\it Stud. Appl. Math.}  {\bf 155 (1)} (2025)  Paper No. e70077, 51 pp.
 

\bibitem{LLY-JSP-2022} H.-L. Li, S. Liu, T. Yang, 
The Navier-Stokes-Vlasov-Fokker-Planck system in bounded domains, {\it J. Stat. Phys.}  {\bf 186 (3)} (2022)  Paper No. 42, 32 pp.

\bibitem{LS-JDE-2021} H.-L. Li, L.-Y. Shou,
Global well-posedness of one-dimensional compressible Navier-Stokes-Vlasov system,
{\it J. Differential Equations} {\bf 280} (2021) 841--890.

\bibitem{LS-CMR-2023} H.-L. Li, L.-Y. Shou,
Global weak solutions for compressible Navier-Stokes-Vlasov-Fokker-Planck system,
{\it Commun. Math. Res.} {\bf 39 (1)} (2023) 136--172.

\bibitem{LSZ-KRM-2025} H.-L. Li, L.-Y. Shou, Y. Zhang, Exponential stability of the inhomogeneous Navier-Stokes-Vlasov system in vacuum, {\it Kinet. Relat. Models} {\bf 18 (2) } (2025)  252--285.

\bibitem{LWW-ARMA-2022} H.-L. Li, T. Wang, Y. Wang,
Wave phenomena to the three-dimensional fluid-particle model,
{\it Arch. Ration. Mech. Anal.} {\bf 243 (2)} (2022) 1019--1089.

\bibitem{MB-Book-2002} A. J. Majda, A. L. Bertozzi,
{\it Vorticity and Incompressible Flow},
Cambridge University Press, 2002.

\bibitem{mas:inviscid} N. Masmoudi, Remarks about the inviscid limit of the Navier–Stokes system,  {\it Commun. Math. Phy.} {\bf 270} (2007) 777--788.

\bibitem{MV-MMMAS-2007} A. Mellet, A. Vasseur, 
Global weak solutions for a Vlasov-Fokker-Planck/Navier-Stokes system of equations,
{\it Math. Models Methods Appl. Sci.} {\bf 17 (7)} (2007) 1039--1063.

\bibitem{MV-CMP-2008} A. Mellet, A. Vasseur, 
Asymptotic analysis for a Vlasov-Fokker-Planck/compressible Navier-Stokes system of equations,
{\it Comm. Math. Phys.} {\bf 281 (3)} (2008) 573--596.

\bibitem {Rouke1} P. J. O’Rourke, Collective drop effects on vaporizing liquid sprays, PhD thesis, Los Alamos National Laboratory, 1981.

\bibitem{RM-1952} W. E. Ranz, W. R. Marshall, Evaporation from drops, part I, Chem. Eng. Prog., {\bf 48} (1952) 141--146.

\bibitem{RM-1952-a} W. E. Ranz, W. R. Marshall, Evaporation from drops, part II, Chem. Eng. Prog., {\bf 48}
(1952) 173--180.

\bibitem{SS-adv-2014} 
V. Sohinger, R. M. Strain, 
The Boltzmann equation, Besov spaces, and optimal time decay rates in $\mathbb{R}^n_x$,
{\it Adv. Math.}   {\bf 261} (2014) 274--332.


\bibitem{SWYZ-JDE-2024} Y. Su, G. Wu, L. Yao, Y. Zhang,
Large time behavior of weak solutions to the inhomogeneous incompressible Navier-Stokes-Vlasov equations in $\mathbb{R}^3$,
{\it J. Differential Equations} {\bf 402} (2024) 361--399.
 
\bibitem{Swann1} H. S. G. Swann, The convergence with vanishing viscosity of nonstationary Navier–Stokes flow to ideal flow in $\mathbb{R}^3$, {\it Trans. Amer. Math. Soc.} {\bf 157} (1971) 373--397.

\bibitem{WY-JDE-2015}  D. Wang, C. Yu, 
Global weak solution to the inhomogeneous Navier-Stokes-Vlasov equations,
{\it J. Differential Equations} {\bf 259 (8)} (2015) 3976--4008.

\bibitem{Wfa-1958} F.-A. Williams, Spray combustion and atomization, {\it Phys. Fluid} {\bf 1} (1958) 541--555.

\bibitem{Wfa-1985} F.-A. Williams, {\it Combustion Theory}, Benjamin Cummings, 1985.

\bibitem{Xu} J. Xu, S. Kawashima, The optimal decay estimates on the framework of Besov spaces for generally dissipative systems, {\it Arch. Rational Mech. Anal.} {\bf 218} (2015) 275--315.

\bibitem{Yc-JMPA-2013} C. Yu,
Global weak solutions to the incompressible Navier-Stokes-Vlasov equations,
{\it J. Math. Pures Appl.} {\bf 100 (2)} (2013) 275--293.


\end{thebibliography}

\end{document}